\documentclass[11pt]{amsart}

\usepackage{a4wide,amsmath,amssymb}
\usepackage{dsfont}
\usepackage{epsfig}
\usepackage{graphicx}
\usepackage{subfig}
\usepackage{verbatim}
\usepackage{multicol}
\usepackage{hhline}

\newtheorem{theorem}{Theorem}
\newtheorem{proposition}[theorem]{Proposition}
\newtheorem{lemma}[theorem]{Lemma}
\newtheorem{corollary}[theorem]{Corollary}
\newtheorem{algorithm}[theorem]{Remark}

\newtheorem{conjecture}{Conjecture}
\newtheorem*{example*}{Example}

\newcommand{\bM}{\mathbf{M}}
\newcommand{\rM}{\mathrm{M}}
\newcommand{\mmm}{{\sc mmm}}

\usepackage{pgf,tikz}
\usepackage{tkz-fct}
\usepackage{pgfplots}
\usetikzlibrary{arrows,trees}
\definecolor{darkergreen}{RGB}{0,153,0}
\pgfplotsset{ every non boxed x axis/.append style={x axis line style=<->},
     every non boxed y axis/.append style={y axis line style=<->}, every axis/.append style={font=\small}}

\begin{document}
\title[]
{Geometrical properties of the mean-median map}
\author{Jonathan Hoseana and Franco Vivaldi}
\address{School of Mathematical Sciences, Queen Mary,
University of London,
London E1 4NS, UK}

\begin{abstract}
We study the mean-median map as a dynamical system on the space of 
finite sets of piecewise-affine continuous functions with rational coefficients.
We determine the structure of the limit function in the neighbourhood
of a distinctive family of rational points, the local minima.
By constructing a simpler map which represents the dynamics in such 
neighbourhoods, we extend the results of Cellarosi and Munday \cite{CellarosiMunday} 
by two orders of magnitude. 
Based on these computations, we conjecture that the Hausdorff dimension of the 
graph of the limit function of the set $[0,x,1]$ is greater than 1.
\end{abstract}
\date{\today}

\maketitle

\section{Introduction}\label{section:Introduction}

Consider a finite multiset\footnote{Hereafter shall be referred to simply as a \textit{set}. We use square brackets to distinguish it from an ordinary set.} $\left[x_1,\ldots,x_n\right]$ of real numbers. 
Its \textit{arithmetic mean} and \textit{median} are defined, respectively, as
$$\langle x_1,\ldots,x_n\rangle:=\frac{1}{n}\sum_{i=1}^n x_i\quad\text{and}\quad\mathcal{M}\left(x_1,\ldots,x_n\right):=\begin{cases}
x_{j_{\frac{n+1}{2}}}&n\text{ odd}\\
\frac{1}{2}\left(x_{j_{\frac{n}{2}}}+x_{j_{\frac{n}{2}+1}}\right)&n\text{ even},
\end{cases}$$
where $k\mapsto j_k$ is a permutation of indices $\{1,\ldots,n\}$ for which $x_{j_1}\leqslant x_{j_2}\leqslant \cdots\leqslant x_{j_n}$.

The set $\left[x_1,\ldots,x_n\right]$ may be enlarged by adjoining to it a new real number $x_{n+1}$ uniquely determined 
by the stipulation that the arithmetic mean of the enlarged set be equal to the median 
of the original set.
In symbols:
\begin{equation}\label{eq:mmm}
x_{n+1}=(n+1)\mathcal{M}\left(x_1,\ldots,x_n\right)-n\left\langle x_1,\ldots,x_n\right\rangle.
\end{equation}
The map $\left[x_1,\ldots,x_n\right]\mapsto\left[x_1,\ldots,x_n,x_{n+1}\right]$ is known as the \textit{mean-median map} (\mmm). This map and its iteration was introduced in \cite{SchultzShiflett}, and subsequently studied in \cite{ChamberlandMartelli,CellarosiMunday,Hoseana}.
Such an iteration is meant to generate a set whose mean and median coincide.
The order of the recursion grows with the iteration, so that the \mmm\ is a dynamical 
system over the infinite-dimensional space of all finite sets of 
real numbers. 

The \mmm\ dynamics is novel and intriguing.
The evolution of the system depends sensitively on a small cluster of points located 
around the median, but only on the average ---rather than the detailed values---  
of the rest of the data. 
From \eqref{eq:mmm} we see that each new element of the set typically results from the 
difference of two diverging quantities, which is a distinctive source of instability. 
Finally, the dependence of the median on smooth variations of the data is not smooth, 
adding complexity.

Since the \mmm\ \textit{preserves affine-equivalences}\footnote{If $\left[x_1,\ldots,x_n\right]\mapsto \left[x_1,\ldots,x_n,x_{n+1}\right]$ then for every $a,b\in\mathbb{R}$ we have $\left[ax_1+b,\ldots,ax_n+b\right]\mapsto\left[ax_1+b,\ldots,ax_n+b,ax_{n+1}+b\right]$.} \cite[section 3]{ChamberlandMartelli}, 
the simplest non-trivial initial sets ---those of size three--- may be
studied in full generality by considering the single-variable initial set 
$[0,x,1]$, with $x\in[0,1]$ referred to as the \textit{initial condition}.
Here one finds already substantial difficulties, which are synthesised in the following 
conjectures \cite{SchultzShiflett,ChamberlandMartelli}.

\begin{conjecture} [Strong terminating conjecture]
For every $x\in[0,1]$ there is an integer $\tau$ such that $x_{\tau+k}=x_\tau$ for all $k\in\mathbb{N}$.
\end{conjecture}

Such an integer $\tau$, if it exists, will be assumed to be minimal, thereby defining
a real function $x\mapsto\tau(x)$, called the \textit{transit time}. At values of $x\in[0,1]$ for which such an integer does not exist, we set $\tau(x)=\infty$.
If the \textit{orbit} $\left(x_n\right)_{n=1}^\infty$ converges at $x$ ---with finite or infinite transit time--- then we have a real function $x\mapsto m(x)$, called the \textit{limit function}, which gives the limit of 
this sequence as a function of the initial condition.
This function has an intricate, distinctive structure, shown in figure \ref{fig:m}. Moreover, since the \textit{median sequence} $\left(\mathcal{M}_n\right)_{n=3}^\infty$, where $\mathcal{M}_n:=\mathcal{M}\left(x_1,\ldots,x_n\right)$, is monotonic \cite[theorem 2.1]{ChamberlandMartelli}, writing 
\begin{eqnarray}\label{eq:msum}
m&=&\lim_{N\to\infty}\mathcal{M}_N\nonumber\\
&=&\lim_{N\to\infty}\left[\mathcal{M}_3+\sum_{n=4}^N\left(\mathcal{M}_n-\mathcal{M}_{n-1}\right)\right]\nonumber\\
&=&\mathcal{M}_{3}+\sum_{n=4}^\infty\left(\mathcal{M}_n-\mathcal{M}_{n-1}\right),
\end{eqnarray}
we see that, in a domain of monotonicity, the limit function is the sum of non-negative (say) piecewise-affine continuous functions with finitely many singularities\footnote{In this paper, a singularity means a point of non-differentiability.}.

\begin{figure}[t]
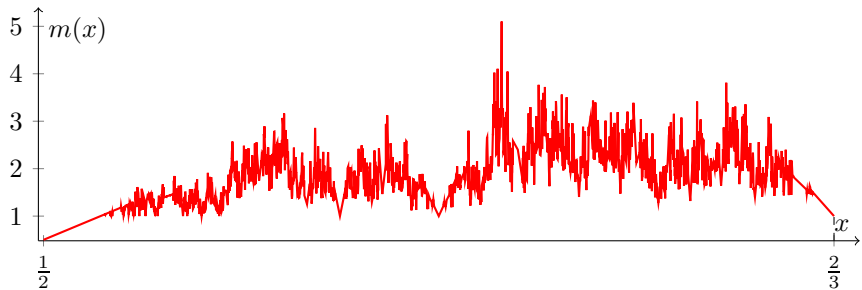

\centering
\include{fig-m}
\caption{\label{fig:m}\rm\small
A computer-generated image of $m(x)$ for $x\in\left[\frac{1}{2},\frac{2}{3}\right]$.
The limit function $m$ may be reconstructed entirely from its values 
in this interval, due to symmetries. The local minima occur at
a prominent set of rational numbers.
}
\end{figure}

\begin{conjecture} [Continuity conjecture]
The function $x\mapsto m(x)$ is continuous.
\end{conjecture}

In \cite{ChamberlandMartelli}, both conjectures were proved to hold within 
an explicitly-determined real neighbourhood of $x=\frac{1}{2}$, where $m$ turns out to be affine. 
Using a computer-assisted proof, this result was then substantially extended in \cite{CellarosiMunday}, where the limit function was constructed in neighbourhoods of 
all rational numbers with denominator at most $18$ lying in the interval $\left[\frac{1}{2},\frac{2}{3}\right]$. The measure of these neighbourhoods adds up to only $11.75\%$ of the measure of the interval.
The authors also identified $17$ rational numbers at which $m$ is non-differentiable.
In \cite{Hoseana}, a long initial segment of the \mmm\ sequence was computed for $x=\frac{\sqrt{5}-1}{2}$,
and, up to stabilisation, for some convergents of the continued fraction expansion of $x$, using exact 
arithmetic in $\mathbb{Q}\left(\sqrt{5}\right)$ and $\mathbb{Q}$, respectively. 
The results suggest that $\tau$ becomes unbounded along the sequence of 
convergents, that is, the transit time at $x$ could be infinite.
Thus, although the strong terminating conjecture seems to hold over $\mathbb{Q}$, it may fail in larger fields.

Motivated by these studies of the system $[0,x,1]$, we introduce the \textit{functional} \mmm\ which acts on the space of finite sets of piecewise-affine continuous functions with rational coefficients; we refer to such sets as \textit{bundles}\footnote{This is unrelated to the standard notion of \textit{fibre bundles} in topology.}. Observing the early evolution of the bundle $[0,x,1]$ (figure \ref{fig:earlydynamics0x1}), we can see that a local minimum of the limit function (figure \ref{fig:m}) first appears as a transversal intersection of bundle elements; we refer to such a point as an \textit{X-point}. The aim of this paper is to present a theory of the dynamics of the \mmm\ in the vicinity of X-points.

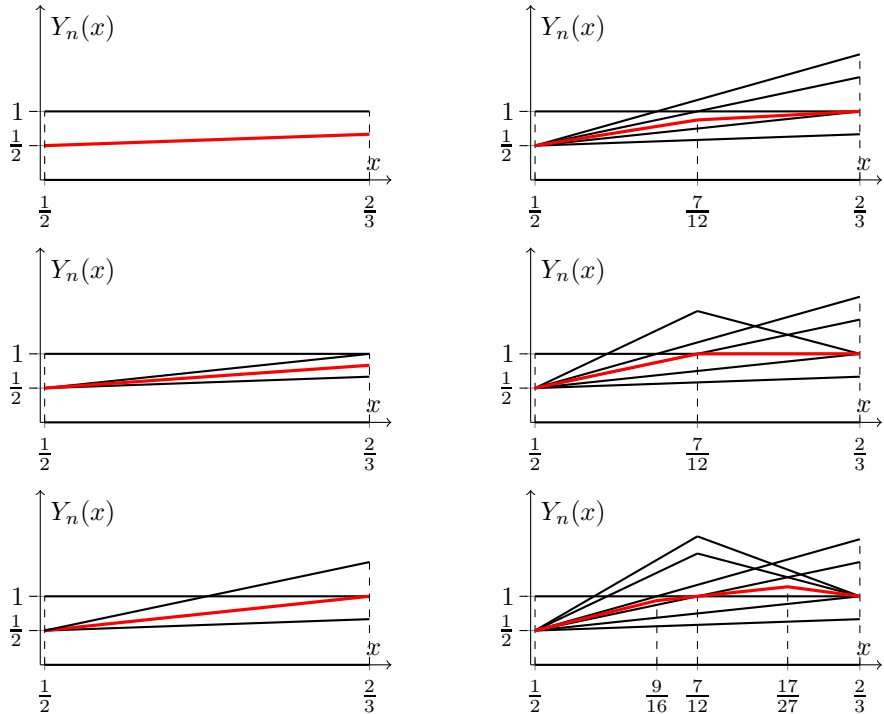
\begin{figure}[t]
\centering
\begin{tabular}{ccc}
\begin{tikzpicture}
\begin{axis}[
	xmin=0.4977477477,
	xmax=0.6779279279,
	ymin=0,
	ymax=2.55,
    ytick={0.5,1},yticklabels={$\frac{1}{2}$,$1$},
    xtick={0.5,0.6666666666666667},xticklabels={$\frac{1}{2}$,$\frac{2}{3}$},
	axis lines=middle,
	samples=100,
	xlabel=$x$,
	ylabel=$Y_n(x)$,
	width=12.5/16*8cm,
	height=12.5/16*5cm,
	clip=false,
	axis lines=middle,
    x axis line style=->,
    y axis line style=->,
]
\addplot [thick,domain=0.5:0.6666666666666667] {0};
\addplot [very thick,color=red,domain=0.5:0.6666666666666667] {x};
\addplot [thick,domain=0.5:0.6666666666666667] {1};
\draw[dashed] (axis cs:0.6666666666666667,0)--(axis cs:0.6666666666666667,1);
\draw[dashed] (axis cs:0.5,0) -- (axis cs:0.5,1);
\draw[dashed] (axis cs:0.492,0.5) -- (axis cs:0.5,0.5);
\draw[dashed] (axis cs:0.492,1) -- (axis cs:0.5,1);
\end{axis}
\end{tikzpicture}&\phantom{aaa}&\begin{tikzpicture}
\begin{axis}[
	xmin=0.4977477477,
	xmax=0.6779279279,
	ymin=0,
	ymax=2.55,
    ytick={0.5,1},yticklabels={$\frac{1}{2}$,$1$},
    xtick={0.5,0.6666666666666667,0.5833333333333333},xticklabels={$\frac{1}{2}$,$\frac{2}{3}$,$\frac{7}{12}$},
	axis lines=middle,
	samples=100,
	xlabel=$x$,
	ylabel=$Y_n(x)$,
	width=12.5/16*8cm,
	height=12.5/16*5cm,
	clip=false,
	axis lines=middle,
    x axis line style=->,
    y axis line style=->,
]
\addplot [thick,domain=0.5:0.6666666666666667] {0};
\addplot [thick,domain=0.5:0.6666666666666667] {x};
\addplot [thick,domain=0.5:0.6666666666666667] {1};
\addplot [thick,domain=0.5:0.6666666666666667] {3*x-1};
\addplot [thick,domain=0.5:0.6666666666666667] {6*x-5/2};
\addplot [thick,domain=0.5:0.6666666666666667] {8*x-7/2};
\draw[dashed] (axis cs:0.6666666666666667,0)--(axis cs:0.6666666666666667,1.8333333333333);
\draw[dashed] (axis cs:0.5,0) -- (axis cs:0.5,1);
\draw[dashed] (axis cs:0.492,0.5) -- (axis cs:0.5,0.5);
\draw[dashed] (axis cs:0.492,1) -- (axis cs:0.5,1);
\draw[dashed] (axis cs:0.5833333333333333,0)--(axis cs:0.5833333333333333,0.875);
\draw[very thick,color=red] (axis cs:0.5,0.5) -- (axis cs:0.5833333333333333,0.875)-- (axis cs:0.6666666666666667,1);
\end{axis}
\end{tikzpicture}\\
\begin{tikzpicture}
\begin{axis}[
	xmin=0.4977477477,
	xmax=0.6779279279,
	ymin=0,
	ymax=2.55,
    ytick={0.5,1},yticklabels={$\frac{1}{2}$,$1$},
    xtick={0.5,0.6666666666666667},xticklabels={$\frac{1}{2}$,$\frac{2}{3}$},
	axis lines=middle,
	samples=100,
	xlabel=$x$,
	ylabel=$Y_n(x)$,
	width=12.5/16*8cm,
	height=12.5/16*5cm,
	clip=false,
	axis lines=middle,
    x axis line style=->,
    y axis line style=->,
]
\addplot [thick,domain=0.5:0.6666666666666667] {0};
\addplot [thick,domain=0.5:0.6666666666666667] {x};
\addplot [thick,domain=0.5:0.6666666666666667] {1};
\addplot [thick,domain=0.5:0.6666666666666667] {3*x-1};
\draw[dashed] (axis cs:0.6666666666666667,0)--(axis cs:0.6666666666666667,1);
\draw[dashed] (axis cs:0.5,0) -- (axis cs:0.5,1);
\draw[dashed] (axis cs:0.492,0.5) -- (axis cs:0.5,0.5);
\draw[dashed] (axis cs:0.492,1) -- (axis cs:0.5,1);
\draw[very thick,color=red] (axis cs:0.5,0.5) -- (axis cs:0.6666666666666667,0.8333333333333333);
\end{axis}
\end{tikzpicture}&&\begin{tikzpicture}
\begin{axis}[
	xmin=0.4977477477,
	xmax=0.6779279279,
	ymin=0,
	ymax=2.55,
    ytick={0.5,1},yticklabels={$\frac{1}{2}$,$1$},
    xtick={0.5,0.6666666666666667,0.5833333333333333},xticklabels={$\frac{1}{2}$,$\frac{2}{3}$,$\frac{7}{12}$},
	axis lines=middle,
	samples=100,
	xlabel=$x$,
	ylabel=$Y_n(x)$,
	width=12.5/16*8cm,
	height=12.5/16*5cm,
	clip=false,
	axis lines=middle,
    x axis line style=->,
    y axis line style=->,
]
\addplot [thick,domain=0.5:0.6666666666666667] {0};
\addplot [thick,domain=0.5:0.6666666666666667] {x};
\addplot [thick,domain=0.5:0.6666666666666667] {1};
\addplot [thick,domain=0.5:0.6666666666666667] {3*x-1};
\addplot [thick,domain=0.5:0.6666666666666667] {6*x-5/2};
\addplot [thick,domain=0.5:0.6666666666666667] {8*x-7/2};
\addplot [thick,domain=0.5:0.5833333333333333] {27/2*x-25/4};
\addplot [thick,domain=0.5833333333333333:0.6666666666666667] {-15/2*x+6};
\draw[dashed] (axis cs:0.5833333333333333,0)--(axis cs:0.5833333333333333,1);
\draw[dashed] (axis cs:0.6666666666666667,0)--(axis cs:0.6666666666666667,1.8333333333333);
\draw[dashed] (axis cs:0.5,0) -- (axis cs:0.5,1);
\draw[dashed] (axis cs:0.492,0.5) -- (axis cs:0.5,0.5);
\draw[dashed] (axis cs:0.492,1) -- (axis cs:0.5,1);

\draw[very thick,color=red] (axis cs:0.5,0.5) -- (axis cs:0.5833333333333333,1) -- (axis cs:0.6666666666666667,1);
\end{axis}
\end{tikzpicture}\\
\begin{tikzpicture}
\begin{axis}[
	xmin=0.4977477477,
	xmax=0.6779279279,
	ymin=0,
	ymax=2.55,
    ytick={0.5,1},yticklabels={$\frac{1}{2}$,$1$},
    xtick={0.5,0.6666666666666667},xticklabels={$\frac{1}{2}$,$\frac{2}{3}$},
	axis lines=middle,
	samples=100,
	xlabel=$x$,
	ylabel=$Y_n(x)$,
	width=12.5/16*8cm,
	height=12.5/16*5cm,
	clip=false,
	axis lines=middle,
    x axis line style=->,
    y axis line style=->,
]
\addplot [thick,domain=0.5:0.6666666666666667] {0};
\addplot [thick,domain=0.5:0.6666666666666667] {x};
\addplot [thick,domain=0.5:0.6666666666666667] {1};
\addplot [very thick,color=red,domain=0.5:0.6666666666666667] {3*x-1};
\addplot [thick,domain=0.5:0.6666666666666667] {6*x-5/2};
\draw[dashed] (axis cs:0.6666666666666667,0)--(axis cs:0.6666666666666667,1.5);
\draw[dashed] (axis cs:0.5,0) -- (axis cs:0.5,1);
\draw[dashed] (axis cs:0.492,0.5) -- (axis cs:0.5,0.5);
\draw[dashed] (axis cs:0.492,1) -- (axis cs:0.5,1);
\end{axis}
\end{tikzpicture}&&\begin{tikzpicture}
\begin{axis}[
	xmin=0.4977477477,
	xmax=0.6779279279,
	ymin=0,
	ymax=2.55,
    ytick={0.5,1},yticklabels={$\frac{1}{2}$,$1$},
    xtick={0.5,0.6666666666666667,0.5833333333333333,0.5625,0.6296296296296296},xticklabels={$\frac{1}{2}$,$\frac{2}{3}$,$\frac{7}{12}$,$\frac{9}{16}$,$\frac{17}{27}$},
	axis lines=middle,
	samples=100,
	xlabel=$x$,
	ylabel=$Y_n(x)$,
	width=12.5/16*8cm,
	height=12.5/16*5cm,
	clip=false,
	axis lines=middle,
    x axis line style=->,
    y axis line style=->,
]
\addplot [thick,domain=0.5:0.6666666666666667] {0};
\addplot [thick,domain=0.5:0.6666666666666667] {x};
\addplot [thick,domain=0.5:0.6666666666666667] {1};
\addplot [thick,domain=0.5:0.6666666666666667] {3*x-1};
\addplot [thick,domain=0.5:0.6666666666666667] {6*x-5/2};
\addplot [thick,domain=0.5:0.6666666666666667] {8*x-7/2};
\addplot [thick,domain=0.5:0.5833333333333333] {27/2*x-25/4};
\addplot [thick,domain=0.5833333333333333:0.6666666666666667] {-15/2*x+6};
\addplot [thick,domain=0.5:0.5833333333333333] {33/2*x-31/4};
\addplot [thick,domain=0.5833333333333333:0.6666666666666667] {8-21/2*x};
\draw[dashed] (axis cs:0.5833333333333333,0)--(axis cs:0.5833333333333333,1);
\draw[dashed] (axis cs:0.6666666666666667,0)--(axis cs:0.6666666666666667,1.8333333333333);
\draw[dashed] (axis cs:0.5,0) -- (axis cs:0.5,1);
\draw[dashed] (axis cs:0.492,0.5) -- (axis cs:0.5,0.5);
\draw[dashed] (axis cs:0.492,1) -- (axis cs:0.5,1);
\draw[dashed] (axis cs:0.5625,0) -- (axis cs:0.5625,0.9375);
\draw[dashed] (axis cs:0.6296296296296296,0) -- (axis cs:0.6296296296296296,1.138888888888889);

\draw[very thick,color=red] (axis cs:0.5,0.5) -- (axis cs:0.5625,0.9375) -- (axis cs:0.5833333333333333,1) -- (axis cs:0.6296296296296296,1.138888888888889) -- (axis cs:0.6666666666666667,1);
\end{axis}
\end{tikzpicture}
\end{tabular}
\caption{\label{fig:earlydynamics0x1}\rm\small
Early evolution of the bundle $[0,x,1]$. In each picture, the red function is the current median.
}
\end{figure}

We now describe the organisation and main results of this paper. In section \ref{section:Preliminaries} we introduce the preliminary concepts: cores, bundles, and X-points, including \textit{bundle equivalence} which generalises affine-equivalence of two real sets.
In section \ref{section:Symmetries} we show that for any two functions forming an X-point and an auxiliary function ---any function not through the X-point--- there is a local symmetry which induces a self-equivalence of the subbundle containing these three functions.
The symmetry is a homology (theorem \ref{thm:GeneralSymmetry}) which is harmonic in a significant special case (corollary 
\ref{cor:RegularX-point}), except at X-points where the median sequence reverses its monotonicity (theorem \ref{thm:monotonicityreversing}).

In section \ref{section:0x11} we prove that the strong terminating conjecture holds globally 
for the bundle $[0,x,1,1]$ (theorem \ref{thm:0x11}) using a method which minimises computer assistance. 

In section \ref{section:X-points} we develop conditions for which the symmetry of an X-point becomes a symmetry of the limit function (lemmas \ref{lemma:affinecombinations} and \ref{lemma:independencefromhistory}, theorem \ref{theorem:inheritance}). Then we use our knowledge of the dynamics of $[0,x,1,1]$ to establish the local structure of the limit function near general X-points 
(theorems \ref{thm:Generalrk2} and \ref{thm:Generalrk1}). Of note is the existence of a hierarchy of X-points, whereby an X-point typically generates an \textit{auxiliary sequence} (theorem \ref{thm:AuxSeq}) of like points. Auxiliary sequences form the scaffolding of the intricate structure of local minima of the limit function seen in figure \ref{fig:m}. Subsequently, we justify the assumptions of our main theorems by illustrating various pathologies.

In section \ref{section:ReducedSystem} we introduce the \textit{normal form} of the \mmm, a one-parameter 
family of dynamical systems over $\mathbb{Q}$ which, after a suitable rescaling, represent the dynamics near any X-point
with a given transit time (proposition \ref{prop:equivalenceNF}). This simplification results from the fact that the local dynamics is largely unaffected by its earlier history. We show that a normal form orbit has a regular initial phase (lemma \ref{lemma:haltingproblem}), which we then exploit to derive lower bounds for the limit and the transit time (theorem \ref{thm:boundformandtau}). We also show that during this phase, the
iterates have low arithmetic complexity\footnote{In our understanding, 
the ubiquity of stabilising sequences over $\mathbb{Q}$ is closely related 
to the slow growth of the arithmetic complexity of their terms.} (corollary \ref{cor:effexp}).

In the last section, we describe the results of exact computations on the system 
$[0,x,1]$, made possible by the theory just developed. 
We have established the strong terminating conjecture in specified 
neighbourhoods of $2791$ rational numbers in the interval 
$\left[\frac{1}{2},\frac{2}{3}\right]$, thereby extending the results of
\cite{CellarosiMunday} by two orders of magnitude. This large data collection
makes it clear that the domains over which the limit function is regular do not account for
the whole Lebesgue measure, suggesting the existence of a drastically different, yet unknown,
dynamical behaviour. For a quantitative assessment of this phenomenon, we have computed the variation of the limit function with respect to Farey fractions, and with respect to dyadic fractions. In both cases, the variation is observed to grow algebraically with the size of the partition, with comparable exponents, suggesting the following conjecture.

\begin{conjecture} \label{conj:Hausdorff}
The Hausdorff dimension of the graph of the limit function of the system $[0,x,1]$
is greater than $1$.
\end{conjecture}

The representation \eqref{eq:msum} of the limit function (a sum of 
piecewise-affine saw-tooth functions of irregularly decreasing amplitude)
makes unbounded variation plausible 
(as with Takagi-type functions \cite{AllaartKawamura,Lagarias}).
However, the nature of the conjectured domains of unbounded variation is 
far more elusive, and is clearly relevant to the continuity conjecture.
This problem deserves further investigation.

\section{Preliminaries}\label{section:Preliminaries}

We first define the \textit{core} of a real set and use it to write an alternative formula for \mmm\ recursion which will be useful later. Then we define our two main objects of study: \textit{bundles} and \textit{X-points}.

\subsection{Cores}

The \textit{core} of a real set $\xi=\left[x_1',\ldots,x_n'\right]$, $n\geqslant 2$, whose elements are written in non-decreasing order, i.e., $x_{1}'\leqslant\ldots\leqslant x_{n}'$, is the subset containing its two (three, respectively) central elements if $n$ is even (odd, respectively), i.e.,
\begin{equation}\label{eq:defcore}
\lambda:=\begin{cases}
\left[x_{\frac{n}{2}}',x_{\frac{n}{2}+1}'\right]&\text{if }n\text{ is even}\\
\left[x_{\frac{n-1}{2}}',x_{\frac{n+1}{2}}',x_{\frac{n+3}{2}}'\right]&\text{otherwise.}
\end{cases}
\end{equation}
The core is said to be \textit{odd} if $n$ is odd, and \textit{even} if $n$ is even. In the context of $\xi$ being an iterate, i.e., when it is written as $\xi_n$, where $n=|\xi|$, we may denote its core by $\lambda_n$.

Let $\xi$ be an initial set for which the median sequence is non-decreasing. We prove two lemmas which clarify the importance of cores. 

\begin{lemma}\label{lemma:translation}
Let $n\geqslant |\xi|+2$ be odd, and let $i\in\{1,\ldots,n-2\}$, $j\in\{1,\ldots,n-1\}$ such that $\mathcal{M}_{n-1}=\left\langle x_i,x_j\right\rangle$ and $\mathcal{M}_{n-2}=x_i$. Then $x_n\geqslant x_j$ with equality if and only if $x_i=x_j$.
\end{lemma}

\begin{proof}
Let $n\geqslant |\xi|+2$ be odd. Notice that
$$x_n=n\mathcal{M}_{n-1}-(n-1)\mathcal{M}_{n-2}
   =n\left\langle x_i,x_j\right\rangle -(n-1)x_i
   =\left(\frac{n}{2}-1\right)\left(x_j-x_i\right)+x_j
   \geqslant x_j,$$
proving the lemma.
\end{proof}

\begin{lemma}\label{lemma:translation2}
Let $n\geqslant |\xi|+3$ be an integer.
\begin{enumerate}
\item[i)] If $n$ is even, then there exist $i\in\{1,\ldots,n-3\}$ and $j\in\{1,\ldots,n-2\}$ such that $\mathcal{M}_{n-3}=x_i$, $\mathcal{M}_{n-2}=\left\langle x_i,x_j\right\rangle$, $\mathcal{M}_{n-1}=x_j$. Moreover,
\begin{equation}\label{eq:lemmatranslation2eq1}
x_n-x_{n-1}=\mathcal{M}_{n-1}-\mathcal{M}_{n-3},
\end{equation}
and $x_n\geqslant x_{n-1}\geqslant x_j\geqslant x_i$ with equality if and only if $x_i=x_j$.
\item[ii)] If $n$ is odd, then there exist $i,j\in\{1,\ldots,n-3\}$ and $k\in\{1,\ldots,n-1\}$ such that $\mathcal{M}_{n-3}=\left\langle x_i,x_j\right\rangle$, $\mathcal{M}_{n-2}=x_j$, $\mathcal{M}_{n-1}=\left\langle x_j,x_k\right\rangle$. Moreover,
\begin{equation}\label{eq:lemmatranslation2eq2}
x_n-x_{n-1}=\frac{n}{2}\left(x_k-x_j\right)-\frac{n-2}{2}\left(x_j-x_i\right),
\end{equation}
and $x_n<x_{n-1}$ if and only if
\begin{equation}\label{eq:cond}
\frac{x_k-x_j}{x_j-x_i}<1-\frac{2}{n}.
\end{equation}
\end{enumerate}
\end{lemma}

\begin{proof}
Let $n\geqslant |\xi|+3$. First suppose $n$ is even. Then there are indices $i\in\{1,\ldots,n-3\}$ and $j\in\{1,\ldots,n-2\}$ for which $\mathcal{M}_{n-3}=x_i$, $\mathcal{M}_{n-2}=\left\langle x_i,x_j\right\rangle$, and $\mathcal{M}_{n-1}=x_j$. Notice that lemma \ref{lemma:translation} guarantees that the median $\mathcal{M}_{n-1}$ must be the larger element of which the previous median $\mathcal{M}_{n-2}$ is the average; we can not have $\mathcal{M}_{n-1}=x_{n-1}<x_j$. Moreover, we have
\begin{equation}\label{eq:prooflemmaordering}
x_n=n\mathcal{M}_{n-1}-(n-1)\mathcal{M}_{n-2}\quad\text{and}\quad x_{n-1}=(n-1)\mathcal{M}_{n-2}-(n-2)\mathcal{M}_{n-3}.
\end{equation}
After substituting the above median expressions, subtracting these two equations 
gives \eqref{eq:lemmatranslation2eq1}. Moreover, the second equation in \eqref{eq:prooflemmaordering} is equivalent to
$$x_{n-1}-x_j=\frac{n-3}{2}\left(x_j-x_i\right).$$
Since $x_j\geqslant x_i$, then the last equation implies $x_{n-1}\geqslant x_j$, whereas the second-to-last implies $x_n\geqslant x_{n-1}$, and hence $x_n\geqslant x_{n-1}\geqslant x_j\geqslant x_i$. If $x_i=x_j$, then the second-to-last equation gives $x_{n}=x_{n-1}$, whereas the last one implies $x_n=x_j$, meaning that these numbers are all equal.

Now suppose $n$ is odd. Then there are indices $i,j\in\{1,\ldots,n-3\}$ and $k\in\{1,\ldots,n-1\}$ for which $\mathcal{M}_{n-3}=\left\langle x_i,x_j\right\rangle$, $\mathcal{M}_{n-2}=x_j$, and $\mathcal{M}_{n-1}=\left\langle x_j,x_k\right\rangle$, where, once again, lemma \ref{lemma:translation} forbids the case $\mathcal{M}_{n-2}= x_{n-2}<x_j$. Substituting these median expressions to \eqref{eq:prooflemmaordering} gives \eqref{eq:lemmatranslation2eq2}. Notice that $x_n<x_{n-1}$ if and only if its right-hand side is negative.
\end{proof}

Thus, for every $n\geqslant |\xi|+2$, the \mmm\ recursion \eqref{eq:mmm} can be written in terms of the elements of the core $\lambda_{n-2}$:
\begin{equation}\label{eq:pardeprecursion}
x_n=\begin{cases}
x_{n-1}+\left(x_{j}-x_{i}\right)&n\text{ even, }\lambda_{n-2}=\left[x_i,x_j\right]\\
x_{n-1}+\frac{n}{2}\left(x_{k}-x_{j}\right)-\frac{n-2}{2}\left(x_{j}-x_{i}\right)&n\text{ odd, }\lambda_{n-2}=\left[x_i,x_j,x_k\right],
\end{cases}
\end{equation}
where the core elements are written in non-decreasing order.

\subsection{Bundles and X-points}
A \textit{bundle} is a finite (multi)set of 
univariate piecewise-affine continuous real functions having 
rational coefficients and finitely many pieces on any interval. A bundle is \textit{regular} if all its elements are affine. A subset of a bundle is called a \textit{subbundle}.

The arithmetic mean $\langle \Xi\rangle$ and median $\mathcal{M}(\Xi)$ of a bundle $\Xi=\left[Y_1,\ldots,Y_n\right]$ are 
defined pointwise, i.e.,
$$
\langle \Xi\rangle(x):=\langle\Xi(x)\rangle\qquad\text{and}\qquad
\mathcal{M}(\Xi)(x):=\mathcal{M}(\Xi(x)),
$$
where $\Xi(x):=\left[Y_1(x),\ldots,Y_n(x)\right]$, for every $x\in\mathbb{R}$. 
One verifies that these are also piecewise-affine continuous real functions with 
rational coefficients, and hence so is [cf.~\eqref{eq:mmm}]
\begin{equation}\label{eq:F}
\rM(\Xi):=|\Xi|(\mathcal{M}(\Xi)-\langle\Xi\rangle)+\mathcal{M}(\Xi).
\end{equation}
Consequently, we can define the \textit{functional} \mmm\ $\bM$ as a self-map on the 
space of all bundles via
\begin{equation}\label{eq:MMM}
\bM(\Xi):=\Xi\uplus [\rM(\Xi)],
\end{equation}
where the set union operator increases the multiplicity of the function $\rM(\Xi)$ in $\Xi$ by $1$ to obtain $\bM(\Xi)$.

Furthermore, any subfield $\mathbb{L}$ of $\mathbb{R}$ is invariant under 
all bundle functions as well as their \mmm\ images \eqref{eq:F}, so it is possible
to consider the restricted setting $Y_i:\mathbb{L}\to\mathbb{L}$.
It then makes sense to begin the study of bundle dynamics from the case 
$\mathbb{L}=\mathbb{Q}$, that is, to consider the restriction of
bundles to rational points.

Given an initial bundle $\Xi_{n_0}=\left[Y_1,\ldots,Y_{n_0}\right]$, iterating the 
\mmm\ produces a recursive sequence of bundles $\left(\Xi_n\right)_{n= n_0}^\infty$ and a functional orbit $\left(Y_n\right)_{n=1}^\infty$ given by
$$
\Xi_{n+1}=\Xi_n\uplus\left[Y_{n+1}\right],\quad\text{where}\quad 
   Y_{n+1}=\rM\left(\Xi_n\right),\quad\text{for every }n\geqslant n_0,
$$
together with the associated median sequence $\left(\mathcal{M}_n\right)_{n=n_0}^\infty$, 
where $\mathcal{M}_n:=\mathcal{M}\left(\Xi_n\right)$, bearing in mind that this is no longer a sequence of numbers, but a sequence
of functions which is piecewise monotonic by
\begin{equation}\label{eq:mmm2}
Y_{n+1}=n\left(\mathcal{M}_n-\mathcal{M}_{n-1}\right)+\mathcal{M}_n,\qquad\text{for every }n\geqslant n_0+1.
\end{equation}
The transit time $\tau$ and the limit $m$ as functions of the initial bundle, if they exist, are defined pointwise. To emphasise the dependence of the \mmm\ dynamics on the initial bundle, we also call an initial bundle a \textit{system}.

From \eqref{eq:mmm2} it also follows that the functional orbit stabilises if and only if 
the median sequence stabilises \cite[theorem 2.3]{ChamberlandMartelli}. At points where this holds, the sum \eqref{eq:msum} is finite.
While \textit{local} stabilisation ---which holds in an open interval---
has been established near some rational points in the works mentioned earlier
\cite{SchultzShiflett,CellarosiMunday}, global stabilisation is a much stronger property.

Let $\Xi$ and $\Xi'$ be two bundles of the same size. If there is a M\"obius transformation $\mu$ with rational coefficients and an affine transformation $f$ with coefficients in $\mathbb{Q}(x)$ such that for all $x\in\mathbb{R}$ we have
\begin{equation}\label{eq:Equivalence}
f\left(\Xi'(x)\right)=\Xi(\mu(x)),
\end{equation} 
where $f$ acts on a set elementwise without cancellations, then we say that $\Xi$ and $\Xi'$ are \textit{equivalent}, written $\Xi\sim\Xi'$, \textit{via the pair} $(\mu,f)$. Bundle equivalence is a generalisation of affine-equivalence of real sets.

The \textit{upper} and \textit{lower concatenations} of 
two bundle functions $Y$ and $Y'$ are
$$
Y\lor Y':=\max\left\{Y,Y'\right\}
\qquad\text{and}\qquad
Y\land Y':=\min\left\{Y,Y'\right\},
$$
respectively, where the maximum and minimum are defined pointwise. 
Notice that $\left[Y,Y'\right]\sim\left[Y\land Y',Y\lor Y'\right]$,
showing that $\Xi\sim\Xi'$ with $\mu$ and $f$ being the identity function does not imply $\Xi=\Xi'$.

By \eqref{eq:MMM} and the commutation relations
$$
\mathcal{M}(f(\Xi(x)))=f(\mathcal{M}(\Xi)(x))
\qquad\text{and}\qquad
\langle f(\Xi(x))\rangle=f(\langle\Xi(x)\rangle)
$$
valid for any bundle $\Xi$ and any $x\in\mathbb{R}$, we have
$$
f\left(\bM(\Xi)(x)\right)=\bM(f(\Xi(x))).
$$
Thus, if $\Xi\sim\Xi'$ then
$$f\left(\bM\left(\Xi'\right)(x)\right)=\bM(\Xi)(\mu(x)),$$
i.e., $\bM(\Xi)\sim \bM\left(\Xi'\right)$, from which it follows inductively that $\bM^n(\Xi)\sim \bM^n\left(\Xi'\right)$ for every $n\geqslant \mathbb{N}_0$. In other words, the functional \mmm\ \textit{preserves bundle equivalences}. We shall also say that a bundle equivalence is \textit{inherited} through the \mmm\ dynamics.

\sloppy Most bundle equivalences discussed in this paper are non-trivial local
\textit{self-equivalences} $\Xi\sim\Xi$. 
Since these are inherited by the orbit of $\Xi$, they result in a local 
functional equation for the limit function $m$:
\begin{equation}\label{eq:FunctionalEquation}
f(m(x))=m(\mu(x)).
\end{equation}
In section \ref{section:X-points}, we shall 
see that it is possible to achieve \eqref{eq:FunctionalEquation} only by establishing a 
self-equivalence $\Omega\sim\Omega$ of an appropriately chosen subbundle $\Omega\subseteq\Xi$, 
rather than that of the whole bundle.

The phase space of the \mmm\ is very large, and not all initial bundles deserve 
attention. To minimise redundancies, initial bundles $\Xi$ are to be chosen to satisfy the following properties:
\begin{enumerate}
\item [i)] $\Xi$ is not the image under $\bM$ of another bundle, except at isolated points;
\item [ii)] the regions where $\mathcal{M}(\Xi)=\langle\Xi\rangle$ are isolated points; 
\item [iii)] not all lines in $\Xi$ are concurrent (including points at infinity).
\end{enumerate}
Condition ii) prevents immediate stabilisation on a positive-measure subdomain, while iii) excludes bundles which are
equivalent to a rational set.

For example, the initial bundles $[0,x,1]$ and $[0,x,1,1]$ satisfy these conditions. The median sequence of the latter bundle is globally non-decreasing. By contrast, for the former bundle, the real line may be subdivided into regions where the median sequence is non-increasing and non-decreasing, separated by isolated points [from ii)] where the mean of the bundle is equal to its median. Since these two regions are connected by a self-equivalence of the bundle \cite[theorem 3.1]{ChamberlandMartelli}, there is no need to study these two regions separately. In the rest of this paper we shall therefore assume that the median sequence is locally non-decreasing. We shall deal separately with the points at which the median sequence reverses its monotonicity (theorem \ref{thm:monotonicityreversing} in section \ref{section:Symmetries}).

If two functions $Y_i$ and $Y_j$ intersect \textit{transversally} at $p\in\mathbb{Q}$, meaning that there exists $\epsilon>0$ such that $$(p-\epsilon,p+\epsilon)\cap \left\{x\in\mathbb{Q}:Y_i(x)=Y_j(x)\right\}=\{p\},$$ then we write $p=Y_i\bowtie Y_j$ and refer to $p$ as an \textit{X-point}\footnote{This definition is local, so $Y_i\bowtie Y_j$ denotes a point rather than a set of points.} (figure \ref{fig:X-points}). Notice that this definition also includes a transversal intersection of more than two functions, say $Y_i$, $Y_j$, and $Y_k$; such a point is also an X-point which can be written using any pair of subscripts. In a geometrical context, the term X-point and the same notation shall also be used to mean the actual point on the plane rather than its abscissa. An X-point $p=Y_i\bowtie Y_j$ is \textit{regular} if both $Y_i$ and $Y_j$ are regular at $p$, and is \textit{singular} otherwise.

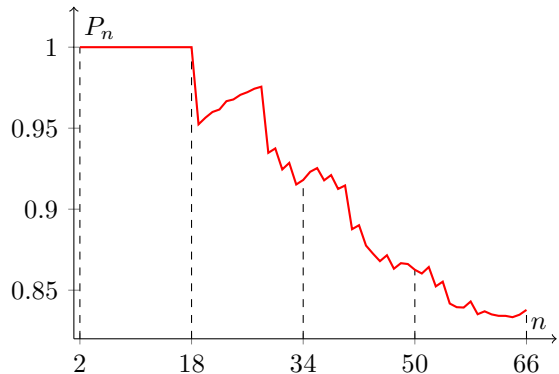
\begin{figure}
\centering
\begin{tikzpicture}
\begin{axis}[
	xmin=1.135135135,
	xmax=70.32432432,
	ymin=0.82,
	ymax=1.025,
    xtick={2,18,34,50,66},
    ytick={0.85,0.9,0.95,1},
	axis lines=middle,
    x axis line style=->,
    y axis line style=->,
	samples=100,
	xlabel=$n$,
	ylabel=$P_n$,
	width=8cm,
    height=6cm
]
\draw[dashed] (axis cs:0,1) -- (axis cs:2,1);
\draw[dashed] (axis cs:2,0.82) -- (axis cs:2,1);
\draw[dashed] (axis cs:18,0.82) -- (axis cs:18,1);
\draw[dashed] (axis cs:34,0.82) -- (axis cs:34., .91803);
\draw[dashed] (axis cs:50,0.82) -- (axis cs:50., .86260);
\draw[dashed] (axis cs:66,0.82) -- (axis cs:66., .83784);
\draw[thick,color=red] plot coordinates  {(axis cs:2., 1.) (axis cs:3., 1.) (axis cs:4., 1.) (axis cs:5., 1.) (axis cs:6., 1.) (axis cs:7., 1.) (axis cs:8., 1.) (axis cs:9., 1.) (axis cs:10., 1.) (axis cs:11., 1.) (axis cs:12., 1.) (axis cs:13., 1.) (axis cs:14., 1.) (axis cs:15., 1.) (axis cs:16., 1.) (axis cs:17., 1.) (axis cs:18., 1.) (axis cs:19., .95238) (axis cs:20., .95652) (axis cs:21., .96000) (axis cs:22., .96154) (axis cs:23., .96667) (axis cs:24., .96774) (axis cs:25., .97059) (axis cs:26., .97222) (axis cs:27., .97436) (axis cs:28., .97561) (axis cs:29., .93478) (axis cs:30., .93750) (axis cs:31., .92453) (axis cs:32., .92857) (axis cs:33., .91525) (axis cs:34., .91803) (axis cs:35., .92308) (axis cs:36., .92537) (axis cs:37., .91781) (axis cs:38., .92105) (axis cs:39., .91250) (axis cs:40., .91463) (axis cs:41., .88764) (axis cs:42., .89011) (axis cs:43., .87755) (axis cs:44., .87255) (axis cs:45., .86792) (axis cs:46., .87156) (axis cs:47., .86325) (axis cs:48., .86667) (axis cs:49., .86614) (axis cs:50., .86260) (axis cs:51., .86029) (axis cs:52., .86429) (axis cs:53., .85235) (axis cs:54., .85526) (axis cs:55., .84177) (axis cs:56., .83951) (axis cs:57., .83929) (axis cs:58., .84302) (axis cs:59., .83516) (axis cs:60., .83696) (axis cs:61., .83505) (axis cs:62., .83417) (axis cs:63., .83415) (axis cs:64., .83333) (axis cs:65., .83486) (axis cs:66., .83784)};
\end{axis}
\end{tikzpicture}
\caption{\label{fig:proportionXpts}\small
The decay of the proportion $P_n$ of fractions with denominator at most $n$ in the interval $\left[\frac{1}{2},\frac{2}{3}\right]$ which are X-points.}
\end{figure}

In the system $[0,x,1]$, most low-complexity rationals in $\left[\frac{1}{2},\frac{2}{3}\right]$ are X-points; the proportion decreases as complexity increases. This is shown in figure \ref{fig:proportionXpts}, where it is also seen that the smallest denominator of any rational in that interval which is not an X-point is $19$. It is therefore fitting that the authors of \cite{CellarosiMunday} restricted their computations to rational numbers with denominator at most $18$. In their algorithm, X-points appear as the endpoints of intervals where the combinatorics of the \mmm\ orbits remains the same.

X-points are the source of singularities. Indeed, if a bundle $\Xi_n$ is regular then 
its image bundle $\Xi_{n+1}$ is singular at $p\in\mathbb{Q}$ if and only if the median 
$\mathcal{M}_{n}$ is singular at $p$, in which case $p$ is an X-point incident with $\mathcal{M}_n$ or 
with one of the functions of which $\mathcal{M}_n$ is the average (see figure \ref{fig:X-points}). Figure \ref{fig:earlydynamics0x1} illustrates how such singularities develop gradually in the system $[0,x,1]$ to become as ubiquitous as seen in figure \ref{fig:m}.

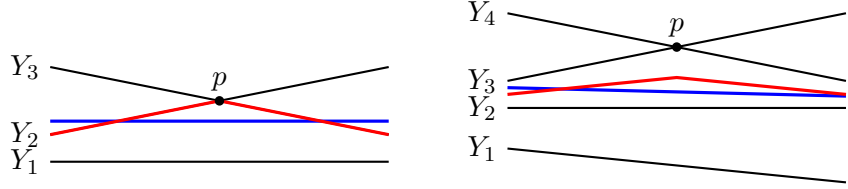
\begin{figure}[t]
\centering
\begin{tabular}{ccc}
\begin{tikzpicture}[scale=0.9]
\draw[thick] (-2.5,0.6) node[left] {$Y_1$} -- (2.5,0.6);
\draw[thick] (-2.5,1) node[left] {$Y_2$} -- (2.5,2);
\draw[thick] (-2.5,2) node[left] {$Y_3$} -- (2.5,1);
\draw[blue, very thick] (-2.5,1.2) -- (2.5,1.2);
\draw[red, very thick] (-2.5,1) -- (0,1.5) -- (2.5,1);
\fill[black] (0,1.5) node[above] {$p$} circle (2pt);
\end{tikzpicture}&\phantom{a}&
\begin{tikzpicture}[scale=0.9]
\draw[thick] (-2.5,0) node[left] {$Y_1$} -- (2.5,-0.5);
\draw[thick] (-2.5,0.6) node[left] {$Y_2$} -- (2.5,0.6);
\draw[thick] (-2.5,1) node[left] {$Y_3$} -- (2.5,2);
\draw[thick] (-2.5,2) node[left] {$Y_4$} -- (2.5,1);
\draw[blue, very thick] (-2.5,0.9) -- (2.5,0.775);
\draw[red, very thick] (-2.5,0.8) -- (0,1.05) -- (2.5,0.8);
\fill[black] (0,1.5) node[above] {$p$} circle (2pt);
\end{tikzpicture}
\end{tabular}

\caption{\label{fig:X-points}\rm\small
Origin of a singularity in a regular bundle $\Xi_n$, if $n$ is odd ($n=3$, left) and if $n$ is even ($n=4$, right). 
The blue and red functions are the mean and median of the bundle, respectively. 
The latter is singular, due to the presence of the X-point $p$. 
In either case, the image function $Y_{n+1}$ will be singular at $p$.
}
\end{figure}

Let us introduce some terminology concerning X-points.
An X-point is \textit{monotonic} if it has a neighbourhood where the median sequence 
is either non-decreasing or non-increasing. A monotonic X-point $p$ appearing above (below, respectively) the current median in the former (latter, respectively) case guarantees the existence of $m(p)$ by \cite[theorem 2.4]{ChamberlandMartelli}. For this reason, the limit function of the system $[0,x,1]$ exists at every X-point in $\left[\frac{1}{2},\frac{2}{3}\right]$.
There are X-points, such as $\frac{1}{2}$ for the bundle $[0,x,1]$, at which the 
increments of the median change sign; these are not monotonic.

An X-point $p$ \textit{stabilises} (or is \textit{stabilising}) if $\tau(p)<\infty$. A stabilising X-point $p=Y_i\bowtie Y_j$ is \textit{proper} if $\left\{Y_i,Y_j\right\}$ is the only pair of functions in the bundle $\Xi_{\tau(p)-1}$ which intersect transversally at the point $\left(p,Y_i(p)\right)$ and $Y_i$, $Y_j$ both have multiplicity $1$ in the same bundle. If the median sequence is non-decreasing, we associate to every stabilising monotonic X-point $p=Y_i\bowtie Y_j$ in the bundle $\Xi_{\tau(p)-1}$ the set\footnote{In its proper meaning, i.e., ignoring multiplicities.}
\begin{equation}\label{eq:SetOfFunctionsImmediatelyAbove}
\mathcal{Y}_p:=\left\{Y\in \Xi_{\tau(p)-1}:\Xi_{\tau(p)-1}(p)\cap \left(Y_i(p),Y(p)\right)
  =\varnothing,\,\,m(p)\in\left[Y_i(p),Y(p)\right)\right\}.
\end{equation}
If $\mathcal{Y}_p\neq\varnothing$, then we say that the stabilising X-point $p$ is \textit{active}\footnote{In such a case, we usually have $\left|\mathcal{Y}_p\right|=1$, but this is not always the case. For instance, $\frac{13}{23}=Y_7\bowtie Y_9$ is an X-point in the system $[0,x,1]$ for which we have $\mathcal{Y}_{\frac{13}{23}}=\left\{Y_{10},Y_{19}\right\}$.}, and as we shall see in section \ref{section:X-points}, the function $\min \mathcal{Y}_p$ will be of particular importance. On the left-hand side of $p$, the multiplicity of this function in the bundle $\Xi_{\tau(p)-1}$ is said to be the \textit{left-rank} of $p$. Similarly we define the \textit{right-rank} of $p$ in its right-hand side. If these two numbers are equal, then the common value is said to be the \textit{rank} of $p$. In the system $[0,x,1]$, every rational number in $\left[\frac{1}{2},\frac{2}{3}\right]$ with denominator between $3$ and $18$ inclusive studied in \cite{CellarosiMunday} is an X-point which is regular, monotonic, proper, active, and of rank $1$. Two examples of X-points which are not of rank $1$ are presented in figure \ref{fig:rank}.

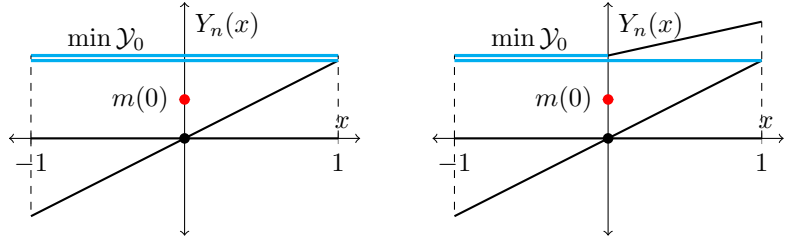
\begin{figure}
\centering
\begin{tabular}{ccc}
\begin{tikzpicture}
\begin{axis}[
	xmin=-1.142857143,
	xmax=1.142857143,
	ymin=-1.250000000,
	ymax=1.750000000,
    xtick={-1,1},
    ytick={0.5},
    yticklabels={$m(0)$},
	axis lines=middle,
	samples=100,
	xlabel=$x$,
	ylabel=$Y_n(x)$,
	width=12.5/16*8cm,
	height=12.5/16*6cm,
	clip=false,
	axis lines=middle,
    x axis line style=<->,
    y axis line style=<->,
]
\addplot [thick,domain=-1:1] {0};
\addplot [thick,domain=-1:1] {x};
\addplot [very thick,cyan,domain=-1:1] {1};
\addplot [very thick,cyan,domain=-1:1] {1.065};
\fill[black] (axis cs:0,0) circle (2pt);
\fill[red] (axis cs:0,0.5) circle (2pt);

\draw[dashed] (axis cs:-1,0) -- (axis cs:-1,1.065);
\draw[dashed] (axis cs:-1,0) -- (axis cs:-1,-1);
\draw[dashed] (axis cs:1,0) -- (axis cs:1,1.065);

\node[above right] at (axis cs:-1,1) {$\,\,\,\,\,\,\min\mathcal{Y}_0$};
\end{axis}
\end{tikzpicture}&\phantom{a}&
\begin{tikzpicture}
\begin{axis}[
	xmin=-1.142857143,
	xmax=1.142857143,
	ymin=-1.250000000,
	ymax=1.750000000,
    xtick={-1,1},
    ytick={0.5},
    yticklabels={$m(0)$},
	axis lines=middle,
	samples=100,
	xlabel=$x$,
	ylabel=$Y_n(x)$,
	width=12.5/16*8cm,
	height=12.5/16*6cm,
	clip=false,
	axis lines=middle,
    x axis line style=<->,
    y axis line style=<->,
]
\addplot [thick,domain=-1:1] {0};
\addplot [thick,domain=-1:1] {x};
\draw[thick] (axis cs:-1,1.065) -- (axis cs:0,1.065) -- (axis cs:1,1.5);
\addplot [very thick,cyan,domain=-1:0] {1.065};
\addplot [very thick,cyan,domain=-1:1] {1};

\fill[black] (axis cs:0,0) circle (2pt);
\fill[red] (axis cs:0,0.5) circle (2pt);

\draw[dashed] (axis cs:-1,0) -- (axis cs:-1,1.065);
\draw[dashed] (axis cs:-1,0) -- (axis cs:-1,-1);
\draw[dashed] (axis cs:1,0) -- (axis cs:1,1.5);

\node[above right] at (axis cs:-1,1) {$\,\,\,\,\,\,\min\mathcal{Y}_0$};
\end{axis}
\end{tikzpicture}
\end{tabular}
\caption{\label{fig:rank}\small
The bundle $[0,x,1,1]$ in which the origin is an X-point of rank $2$ (left) and the bundle $\left[0,x,1,Y_4(x)\right]$, where $Y_4(x)$ is equal to $1$ for $x\leqslant0$ and to $\frac{x}{2}+1$ for $x>0$, in which it is an X-point of left-rank $2$ and right-rank $1$.}
\end{figure}

A regular stabilising X-point formed by functions whose multiplicities in the bundle $\Xi_{\tau(p)-1}$ are not all $1$ has a simpler dynamics, which we shall now present. At such an X-point the limit function is regular but the transit time function has a jump discontinuity. To lighten up the notation, we write inequalities of functions ($Y_i>Y_j$ or $Y_i>c$, with $c$ a constant), to denote pointwise inequalities ($Y_i(x)>Y_j(x)$ or $Y_i(x)>c$), valid for all points in the domain under consideration.

\begin{lemma}\label{lemma:Type3}
Let $p=Y_i\bowtie Y_j$ be a regular stabilising X-point with $Y_i<Y_j$ if $x>p$, where the multiplicities of $Y_i$ and $Y_j$
in $\Xi_{\tau(p)-1}$ are $1$ and at least $2$, respectively.
Also assume that, locally,
\begin{equation}\label{eq:M_(l+1)}
\mathcal{M}_n=Y_i\land Y_j\qquad\text{and}\qquad
\mathcal{M}_{n+1}(x)=\begin{cases}
Y_j(x)&\text{if }x< p\\
\left\langle Y_i,Y_j\right\rangle(x)&\text{if }x\geqslant p,
\end{cases}
\end{equation}
for an odd integer $n$.
Then near $p$ we have $m(x)=Y_j(x)$ and
\begin{equation}\label{eq:tauType3}
\tau(x)=\begin{cases}
n+2&\text{if }x<p\\
n+4&\text{if }x\geqslant p.
\end{cases}
\end{equation}
\end{lemma}
\begin{proof}
If $x< p$, then \eqref{eq:mmm2} easily gives $Y_{n+2}=Y_j$, and hence 
$m(x)=Y_j(x)$ and $\tau(x)=n+2$. 
If $x\geqslant p$, we use \eqref{eq:mmm2} to compute $Y_{n+2}$, $Y_{n+3}$, and $Y_{n+4}$. 
Firstly, $Y_{n+2}=\frac{n}{2}\left(Y_j-Y_i\right)+Y_j$, which lies above $Y_j$ since $Y_j>Y_i$. 
Next, since the median sequence is non-decreasing, 
then $\mathcal{M}_{n+3}=\mathcal{M}_{n+2}=Y_j$, and so we obtain 
$Y_{n+3}=\frac{n+2}{2}\left(Y_j-Y_i\right)+Y_j$, which also lies above $Y_j$. 
Finally, $Y_{n+4}=Y_j$, and hence $m(x)=Y_j(x)$ and $\tau(x)=n+4$, completing the proof.
\end{proof}

In lemma \ref{lemma:Type3}, the limit function is regular but the transit time increases by $2$. 
If instead we assume that $Y_i<Y_j$ if $x<p$, and we modify accordingly
equation \eqref{eq:M_(l+1)}, we obtain a mirror version of \eqref{eq:tauType3}
describing a \textit{decrease} by 2 of the transit time.

\section{Symmetries}\label{section:Symmetries}

This section is devoted to the discussion of geometrical aspects of a bundle, without 
regard for the dynamics. Specifically, we describe two-dimensional projective collineations
determined by an X-point $p=Y_i\bowtie Y_j$ and an \textit{auxiliary function} 
$Y$, which is any piecewise-affine function not through $p$, assuming that all these functions are locally regular except, possibly, at $p$. 

The subbundle containing only these three functions will be called a \textit{triad} and denoted by 
$\Omega=\left[Y_i,Y_j;Y\right]$. The projective collineation induces a non-trivial 
self-equivalence $\Omega\sim\Omega$, i.e.,
\begin{equation}\label{eq:triadequivalence}
f(\Omega(x))=\Omega(\mu(x)),
\end{equation}
which holds for all $x$ sufficiently close to $p$ and is specified by a pair of M\"obius and affine transformations $(\mu,f)$. Later in section \ref{section:X-points}, we shall develop suitable conditions under which 
such a symmetry is inherited by the limit function [cf.~\eqref{eq:FunctionalEquation}], i.e., under which \eqref{eq:triadequivalence} implies \eqref{eq:FunctionalEquation} for all $x$ sufficiently close to $p$, for the same pair $(\mu,f)$.

We shall use some concepts of projective geometry (for background, see \cite{Coxeter}).
We identify the point $(x,y)\in\mathbb{Q}^2$ with the projective point $(x,y,1)\in \mathbb{P}^2(\mathbb{Q})$, 
represented with homogeneous coordinates, and we denote by $o_{\infty}:=(0,1,0)$ 
the point at infinity on the ordinate axis [the line $(1,0,0)$]. The symbol $Y$ shall 
be used to denote both a function and its graph. 
If $Y$ is locally regular, by the \textit{line} $Y$ we mean the graph of the affine extension of $Y$.

Typically, the symmetry of a triad will be a \textit{homology}, namely a projective collineation 
with a line of fixed points (the axis) and an additional fixed point (the centre) not 
on the axis (theorem \ref{thm:GeneralSymmetry}). This type of symmetry will be useful
for X-points which are monotonic. If the X-point is also regular, then the homology is \textit{harmonic},
and hence the associated M\"obius transformation $\mu$ is involutory (corollary \ref{cor:RegularX-point}). The second
type of symmetry is a more general projectivity which is needed to handle non-monotonic X-points (theorem \ref{thm:monotonicityreversing}).

\subsection{Symmetry of triads}

We shall construct a ratio-preserving symmetry for a triad $\Omega=\left[Y_i,Y_j;Y\right]$. 
After defining the concatenations
$$
U:=Y_i\lor Y_j
\qquad\text{and}\qquad
L:=Y_i\land Y_j,
$$
we have $\left[Y_i,Y_j;Y\right]\sim\left[U,L;Y\right]$, and we shall use the latter triad
for our analysis.
We introduce the regular functions $I$, $J$, $K$, $I'$, $J'$, $K'$ as follows:
\begin{equation}\label{eq:GeneralTriad}
\Omega(x)=\left[U,L;Y\right](x)=\begin{cases}
\left[I,J;K\right](x)&\text{if } x\geqslant p\\
\left[I',J';K'\right](x)&\text{if } x<p.
\end{cases}
\end{equation}
On the affine plane, we have $J<I<K$ for $x>p$ and $J'<I'<K'$ for $x<p$.
The symmetry is described by the following theorem.

\begin{theorem}\label{thm:GeneralSymmetry}
The triad $\Omega$ given by \eqref{eq:GeneralTriad} determines a homology $\lambda$, 
with axis $po_\infty$ which maps $I$, $J$, $K$ to $I'$, $J'$, $K'$, respectively.
This homology induces a self-equivalence of $\Omega$ via the 
pair $(\mu,f)$ given by
\begin{equation}\label{eq:GeneralMu}
\mu=\left(\frac{K'-I'}{K'-J'}\right)^{-1}\circ\frac{K-I}{K-J}
\end{equation}
and
\begin{equation}\label{eq:generalf}
f(z)=\frac{K'(\mu(x))-I'(\mu(x))}{K(x)-I(x)}z+\frac{K(x)\cdot I'(\mu(x))-I(x)\cdot K'(\mu(x))}{K(x)-I(x)}.
\end{equation}
This self-equivalence is unique up to inversion.
The inverse of $\mu$ and the corresponding $f$ are obtained by interchanging 
all primed and unprimed quantities.
\end{theorem}

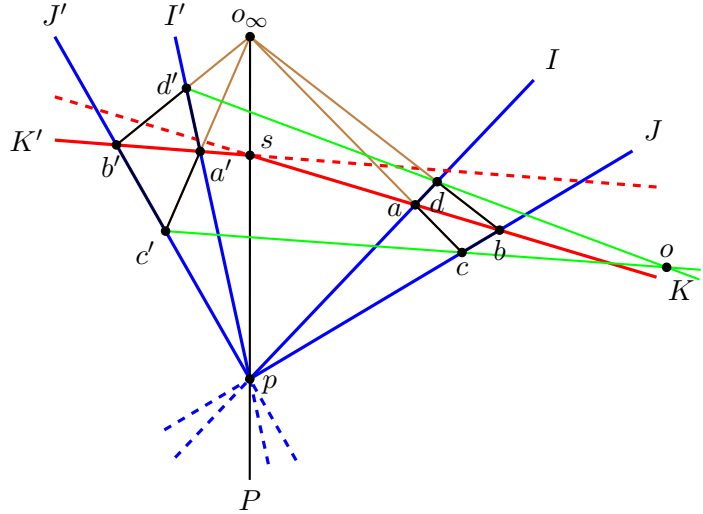
\begin{figure}[t]
\centering
\begin{tikzpicture}[scale=0.8]
\draw [very thick,blue,dashed] (1.82,1.46)--(3.2406828900207514,2.307296736464818);
\draw [very thick,blue] (3.2406828900207514,2.307296736464818)-- (9.6,6.1)node[above right,black] {$J$};
\draw [very thick,blue,dashed] (2.,1.)-- (3.2406828900207514,2.307296736464818);
\draw [very thick,blue] (3.2406828900207514,2.307296736464818)--(7.96,7.28)node[above right,black] {$I$};
\draw [very thick,blue] (3.2406828900207514,2.307296736464818)-- (2.,8.)node[above,black] {$I'$};
\draw [very thick,blue] (3.2406828900207514,2.307296736464818)-- (0.,8.)node[above,black] {$J'$};
\draw [very thick,blue,dashed] (3.2406828900207514,2.307296736464818)-- (3.5542444318582334,0.868562611173946);
\draw [very thick,blue,dashed] (4.010853273489069,0.9543869935988143)-- (3.2406828900207514,2.307296736464818);
\draw [very thick,red,dashed] (0.,7.)--(3.243243243243244,6.027027027027027);
\draw [very thick,red] (3.243243243243244,6.027027027027027)-- (10.,4.)node[right,yshift=-5pt,black] {$K$};
\draw [very thick,red,dashed] (10.,5.5)--(3.243243243243244,6.027027027027027);
\draw [very thick,red] (3.243243243243244,6.027027027027027)-- (0.,6.28)node[left,black] {$K'$};
\draw [thick] (3.2446016317765705,8.000519953463092)-- (3.23952784460491,0.6292245929272064)node[below] {$P$};
\draw [thick,brown] (3.2446016317765705,8.000519953463092)-- (1.0246408635347533,6.200078012644289);
\draw [thick,brown] (3.2446016317765705,8.000519953463092)-- (1.8407317723352095,4.766502578844286);
\draw [thick,brown] (3.2446016317765705,8.000519953463092)-- (7.391167192429021,4.782649842271294);
\draw [thick,brown] (3.2446016317765705,8.000519953463092)-- (6.767009586520685,4.4104016043002545);
\draw [thick,green] (2.1868512294322255,7.142658763057012)-- (10.711856834545715,3.9613329235126304);
\draw [thick,green] (10.733039803937594,4.123713098543144)-- (1.8407317723352095,4.766502578844286);
\draw [fill=black] (3.2406828900207514,2.307296736464818) circle (2.0pt)node[right,xshift=1pt,yshift=-2pt] {$p$};
\draw [fill=black] (3.243243243243244,6.027027027027027) circle (2.0pt)node[right,yshift=5pt] {$s$};
\draw [fill=black] (3.2446016317765705,8.000519953463092) circle (2.0pt) node[above] {$o_\infty$};
\draw [fill=black] (2.4159310610503626,6.091557377238071) circle (2.0pt) node[below,xshift=8pt,yshift=1pt] {$a'$};
\draw [fill=black] (1.0246408635347533,6.200078012644289) circle (2.0pt)node[below,xshift=-2pt] {$b'$};
\draw [fill=black] (1.8407317723352095,4.766502578844286) circle (2.0pt)node[below left] {$c'$};
\draw [fill=black] (2.1868512294322255,7.142658763057012) circle (2.0pt)node[left,yshift=2pt,xshift=1pt] {$d'$};
\draw [fill=black] (5.989092711948437,5.203272186415468) circle (2.0pt)node[left,yshift=-2pt,xshift=-1pt] {$a$};
\draw [fill=black] (7.391167192429021,4.782649842271294) circle (2.0pt)node[below] {$b$};
\draw [fill=black] (6.767009586520685,4.4104016043002545) circle (2.0pt)node[below] {$c$};
\draw [fill=black] (6.35386533361351,5.587629915284035) circle (2.0pt)node[below] {$d$};
\draw [fill=black] (10.167101212168962,4.164622542074844) circle (2.0pt)node[above] {$o$};
\draw[thick](2.4159310610503626,6.091557377238071)--(1.8407317723352095,4.766502578844286)--(1.0246408635347533,6.200078012644289)--(2.1868512294322255,7.142658763057012)--cycle;
\draw[thick](5.989092711948437,5.203272186415468)-- (6.767009586520685,4.4104016043002545)--(7.391167192429021,4.782649842271294) -- (6.35386533361351,5.587629915284035)--cycle;
\end{tikzpicture}
\caption{\label{fig:GeneralSymmetry}\rm\small
Construction of the homology $\lambda$ determined by the triad $\left[U,L;Y\right]$.
Here $U=I\lor I'$ and $L=J\lor J'$.
}
\end{figure}

\begin{proof}
With reference to figure \ref{fig:GeneralSymmetry}, we let $U$ be a concatenation
of $I$ and $I'$, $L$ a concatenation of $J$ and $J'$, and $Y$ a concatenation
of $K$ and $K'$, with the stipulation that primed (unprimed) quantities represent 
the functions to the left (right) of the X-point. 
We consider the projective collineation $\lambda$ determined by the following data:
\begin{equation}\label{eq:ProjectivityData}
\lambda(I)=I',
\qquad
\lambda(J)=J',
\qquad
\lambda(K)=K',
\qquad
\lambda\left(o_\infty\right)=o_\infty.
\end{equation}
From \eqref{eq:GeneralTriad} and the remark following it, $\lambda$ is not the identity.
We have $\lambda(p)=I'\cdot J'=p$, and hence the line $P:=po_\infty$ is invariant. 
The point $s:=P\cdot K$ is also invariant, as $\lambda(s)=P\cdot K'=s$, 
and hence $\lambda$ fixes $P$ pointwise.
Since $p$, $s$, $o_\infty$ are distinct, the points
\begin{equation*}\label{eq:abcd}
a:=I\cdot K,
\qquad
b:=J\cdot K,
\qquad
c:=J\cdot \left(o_\infty a\right),
\qquad
d:=I\cdot \left(o_\infty b\right)
\end{equation*}
are vertices of a quadrangle with no vertex on $P$.
The map $\lambda$ sends this quadrangle to its image $a'b'c'd'$, which is also a 
quadrangle, as $\lambda$ preserves incidence. 
Since $\lambda$ is not the identity, it can fix at most one of these vertices.
If such a fixed vertex exists, then we call it $o$, and $\lambda$
is a homology with axis $P$ and centre $o$.

If there is no fixed vertex, then we define the following lines:
\begin{equation}\label{eq:ABCD}
A:=aa',
\qquad
B:=bb',
\qquad
C:=cc',
\qquad
D:=dd'.
\end{equation}
All these lines are invariant under $\lambda$ since their intersection with $P$ is 
a fixed point, and none of these lines coincides with $P$.
Furthermore, no three of them can coincide; indeed if $A=B=C$, say, then $a,b,c$ 
would be collinear. 
Thus least two of these lines must be distinct, and we claim that they must be concurrent
at a point, which we shall call $o$. Indeed if two lines meet at $o$, a third line not
passing through $o$ would result in two (if $o$ is on $P$) or three (if $o$ is not
on $P$) fixed points of $\lambda$ lying outside $P$, making $\lambda$ the identity.

If $o$ and $P$ were incident, then, on the affine plane one would have $J'$ 
lying above $I'$ on the left of $p$, as easily verified, contradicting \eqref{eq:GeneralTriad}.
Thus $o$ and $P$ are not incident, and $\lambda$ is a homology.

Let us now consider the action of $\lambda$ on lines through $o_\infty$. Let $X$ be such a line, and let
\begin{equation}\label{eq:SimpleRatioDerivation1}
x:=X\cdot K,\qquad x_i:=X\cdot I,\qquad x_j:=X\cdot J.
\end{equation}
Letting $x'$, $x_i'$, $x_j'$ be the images of these points under $\lambda$, we have
\begin{equation}\label{eq:SimpleRatioDerivation2}
x'=\lambda(X)\cdot K',\qquad x_i'=\lambda(X)\cdot I',\qquad x_j'=\lambda(X)\cdot J'.
\end{equation}
Since $\lambda$ is a homology, then the lines $xx'$, $x_ix_i'$, $x_jx_j'$ are concurrent at $o$, and so
\begin{equation}\label{eq:SimpleRatioDerivation3}
\left(x';x_i',x_j'\right)=\left(x;x_i,x_j\right).
\end{equation}
This gives
\begin{equation}\label{eq:SimpleRatioDerivation4}
\frac{K'(\mu(x))-I'(\mu(x))}{K'(\mu(x))-J'(\mu(x))}=\frac{K(x)-I(x)}{K(x)-J(x)},
\end{equation}
where $x$ and $\mu(x)$ are the points obtained by projecting the lines $X$ and 
$\lambda(X)$, respectively, from $o_\infty$ to the real axis. Solving this for $\mu(x)$, one obtains \eqref{eq:GeneralMu}. Now let
$$
A(x):=\frac{K'(\mu(x))-I'(\mu(x))}{K(x)-I(x)}
  =\frac{K'(\mu(x))-J'(\mu(x))}{K(x)-J(x)}.
$$
Then
$$
K'(\mu(x))-A(x)K(x)=I'(\mu(x))-A(x)I(x)=J'(\mu(x))-A(x)J(x).
$$
Letting this quantity be $B(x)$, we obtain
$$
K'\left(\mu(x)\right)=f\left(K(x)\right),\qquad I'\left(\mu(x)\right)=f\left(I(x)\right),
\qquad J'\left(\mu(x)\right)=f\left(J(x)\right),
$$
where
\begin{eqnarray*}
f(z)&=&A(x)z+B(x)\\
    &=&\frac{K'(\mu(x))-I'(\mu(x))}{K(x)-I(x)}z+
     \left[K'(\mu(x))-\frac{K'(\mu(x))-I'(\mu(x))}{K(x)-I(x)}K(x)\right]\\
    &=&\frac{K'(\mu(x))-I'(\mu(x))}{K(x)-I(x)}z+\frac{K(x)\cdot 
     I'(\mu(x))-I(x)\cdot K'(\mu(x))}{K(x)-I(x)}.
\end{eqnarray*}
In other words, $\Omega\sim\Omega$ via the pair $(\mu,f)$ given by \eqref{eq:GeneralMu} 
and \eqref{eq:generalf}. The proof of theorem \ref{thm:GeneralSymmetry} is complete.
\end{proof}

The homology $\lambda$ is \textit{harmonic} if, for any given point $x$, the harmonic
conjugate of $o$ with respect to $x$ and its image $x'$ lies on $P$, namely $xx'\cdot P$.
For an explicit formula, we project from $o_\infty$ the points $x$, $x'$, $xx'\cdot P$, and $o$
to the real axis $\mathbb{R}=(0,1,0)$, letting $x$, $\mu(x)$, $p$, $q$ 
be the corresponding images.
Since these points form a harmonic set, their cross-ratio is equal to $-1$ 
\cite[page 38]{Schwerdtfeger}:
$$
\left(x,\mu(x);p,q\right)=\frac{(x;p,q)}{(\mu(x);p,q)}=-1,
$$
from which we obtain
\begin{equation}\label{eq:mu}
\mu(x)=\frac{Tx-2D}{2x-T},\qquad\text{where}\qquad T:=p+q \qquad\text{and}\qquad D:=pq,
\end{equation}
which is indeed an involution since the associated matrix $\left(\begin{array}{cc}T&-2D\\2&-T\end{array}\right)$ has zero trace \cite[page 49]{Schwerdtfeger}. 

Letting $o:=(q,r)$, the aforementioned collinearity translates to
$$
\det\left(\begin{array}{ccc}x&I'(x)&1\\\mu(x)&I(\mu(x))&1\\q&r&1\end{array}\right)=0,
$$
which determines the self-equivalence $I'(\mu(x))=f(I(x))$, where
\begin{equation}\label{eq:f}
f(z)=\frac{\mu(x)-q}{x-q}z+\frac{x-\mu(x)}{x-q}r.
\end{equation}
The same transformation applies to $J$ and $K$.

A significant instance of this phenomenon occurs if the X-point is \textit{regular},
for in this case $\lambda$ exchanges $I$ and $J$, since $I'=J$ and $J'=I$. 
If, in addition, the auxiliary function $Y$ is regular, then $o$ and $Y$ are incident,
so that formulae \eqref{eq:mu} and \eqref{eq:f} apply with $r=Y(q)$.

\begin{corollary} \label{cor:RegularX-point}
If $p$ is regular, then the homology $\lambda$ is harmonic, and hence the associated
function $\mu$ is an involution.
\end{corollary}

\subsection{Symmetry of pseudotriads}\label{section:Pseudotriads}

\begin{figure}[t]
\centering
\begin{tikzpicture}[scale=0.8]
\draw [very thick,red] (9.,2.)-- (3.,5.75);
\draw [very thick,red,dashed] (3.,5.75) -- (1.,7.)node[left,black] {$K$};
\draw [very thick,red] (1.,2.)node[left,black] {$K'$}-- (3.,2.25);
\draw [very thick,red,dashed] (3.,2.25)--(9.,3.);
\draw [very thick,blue] (3.,3.)-- (7.009421487603315,4.44909090909091) node[right,black] {$J$};
\draw [very thick,blue] (3.,3.)-- (5.306942148760338,6.20115702479339) node[above,black] {$I$};
\draw [very thick,blue] (3.,3.)-- (1.472231404958681,7.159834710743803) node[above,black] {$J'$};
\draw [very thick,blue] (3.,3.)-- (1.3895867768595074,1.6391735537190095) node[below,black] {$I'$};
\draw [thick] (3.,8.)node[above] {$o_\infty$}-- (3.,1.65)node[below] {$P$};
\draw [thick,brown] (3.,8.)-- (8.553371569300724,1.851624333988486) node[right,xshift=-3pt,yshift=-5.5pt,black] {$Q$};
\draw [fill=black] (3.,8.) circle (2.0pt);
\draw [fill=black] (3.,3.) node[below right] {$p$} circle (2.0pt);
\draw [fill=black] (3.,5.75) circle (2.0pt);
\draw [fill=black] (3.,2.25) circle (2.0pt);
\draw [fill=black] (7.666666666666667,2.8333333333333335)node[below] {$s$} circle (2.0pt);
\end{tikzpicture}
\caption{\label{fig:monotonicityreversing}\small\rm Construction of the projective collineation $\lambda$ near a non-monotonic X-point $p$. The two branches $K$ and $K'$
of the auxiliary function of the pseudotriad do not meet on $P$.}
\end{figure}
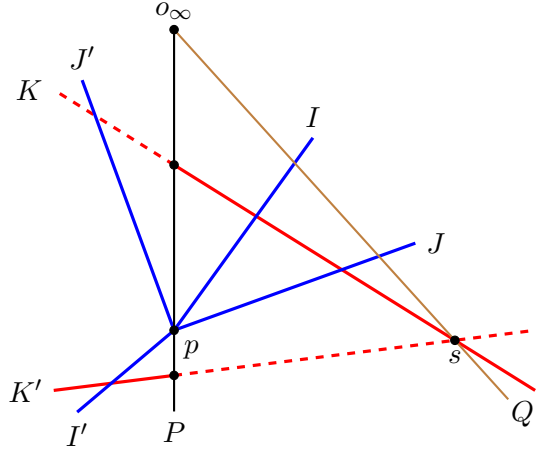

The triads considered so far will be relevant for X-points which are monotonic. The construction of an inheritable symmetry around a non-monotonic 
X-point ---such as the point $p=\frac{1}{2}$ in the system $[0,x,1]$---
requires a variant of the triad construct, a \textit{pseudotriad},
characterised by the fact that the auxiliary function has two distinct components
which do not meet on the line $po_\infty$.
Specifically, one has
\begin{equation}\label{eq:PseudoTriad}
\Omega(x)=\begin{cases}
[I,J;K](x)&\text{if } x\geqslant p\\
\left[I',J';K'\right](x)&\text{if } x<p.
\end{cases}
\end{equation}
As before, $I$, $J$, $I'$, $J'$ are concurrent at $p$, and $p$ is not on $K$ or $K'$.
If the median sequence is non-decreasing on the right of $p$
and non-increasing on the left, then we have $J<I<K$ for $x>p$ and 
$K'<I'<J'$ for $x<p$. Otherwise, the primed and unprimed functions are exchanged.
Our next result is the analogue of theorem \ref{thm:GeneralSymmetry} for non-monotonic X-points.

\begin{theorem}\label{thm:monotonicityreversing}
A pseudotriad \eqref{eq:PseudoTriad} determines a projective collineation
which maps the unprimed components to the corresponding primed ones and 
fixes $o_\infty$.
Such a collineation, which is a homology if and only if the
X-point is regular, induces a self-equivalence $\Omega\sim\Omega$ 
via the pair $(\mu,f)$ given by \eqref{eq:GeneralMu} and \eqref{eq:generalf}.
This self-equivalence is unique up to inversion.
The inverse of $\mu$ and the corresponding $f$ are obtained by interchanging 
all primed and unprimed quantities.
\end{theorem}

\begin{proof}
As for triads, the projective collineation $\lambda$ is uniquely determined by the conditions
\eqref{eq:ProjectivityData}. Since $p$ is fixed, then the line $P:=po_\infty$ is 
invariant under $\lambda$, but not pointwise; indeed $\lambda(K\cdot P)=
K'\cdot P\not=K\cdot P$. 
First assume that the X-point is regular: $I=I'$ and $J=J'$. 
Then three lines through $p$ are invariant, and hence all lines through $p$ are invariant.
Considering the invariant line $ps$, where $s:=K\cdot K'$, we find that 
the point $s$ is also fixed, and hence the line $Q:=so_\infty$ is invariant.
Now, at least one of the fixed points $I\cdot Q$ or $J\cdot Q$ is distinct from $s$ 
(and obviously also from $o_\infty$), and hence $Q$, having three fixed points, is 
pointwise invariant. 
Thus if $p$ is regular, then $\lambda$ is a homology with axis $Q$ and centre $p$.

Now assume that $p$ is not regular. Then the pencil through $p$ has at most two fixed lines: $P$ and
$P'$ (possibly with $P=P'$). Suppose for a contradiction that there exists a line $L$ fixed
by $\lambda$ pointwise. If $L$ does not contain $p$, then all lines through $p$ are fixed, contradicting
the fact that $p$ is not regular. Otherwise, all lines through $o_\infty$ are fixed, and so the two
triples of points $o_\infty$, $I\cdot K$, $I'\cdot K'$ and $o_\infty$, $J\cdot K$, $J'\cdot K'$
are both collinear, which happens only if $I'$ lie above $J'$ on the affine plane, contradicting the
fact that $p$ is non-monotonic.

Thus, $\lambda$ fixes no line pointwise, and hence is not a homology. Moreover, the pencil through
$o_\infty$ has at most two fixed lines: $P$ and $Q$. Then $o:=P'\cdot Q$ is the third
fixed point of $\lambda$ which, generically, is distinct from $p$ and $o_\infty$.
(The non-generic configurations which result in $o$ coinciding with $p$ or with $o_\infty$ 
can be characterised using standard algebraic tools. We shall not
pursue this matter here.)

To obtain the $(\mu,f)$-self-equivalence, let $X$ be a line through $o_\infty$, and let the points
$x$, $x_i$, and $x_j$ be given by \eqref{eq:SimpleRatioDerivation1}. Moreover, let $x'$, $x_i'$, and
$x_j'$ be the their respective images under $\lambda$, i.e., \eqref{eq:SimpleRatioDerivation2}. Since
$\lambda$ is projective, it preserves simple ratio, i.e., \eqref{eq:SimpleRatioDerivation3} holds.
In other words, \eqref{eq:SimpleRatioDerivation4} holds, where $x$ and $\mu(x)$ are the points obtained
by projecting the lines $X$ and $\lambda(X)$, respectively, from $o_\infty$ to the real axis. Then one
continues by repeating the same algebraic argument as in the proof of theorem \ref{thm:GeneralSymmetry},
giving the desired self-equivalence.
\end{proof}

An example is provided by the regular X-point $p=\frac{1}{2}$ in the system $[0,x,1]$,
where the median sequence is non-decreasing (non-increasing) on the right-hand
(left-hand) side of $p$. The pseudotriad given by $I(x)=I'(x)=3x-1$, $J(x)=J'(x)=x$, $K(x)=1$, $K'(x)=0$, 
determines the homology with centre $\left(\frac{1}{2},\frac{1}{2}\right)$ and axis the line at infinity.
The induced self-equivalence is given by $\mu(x)=1-x$ and $f(z)=1-z$, where $\mu$ is an involution. 
By contrast, the symmetry of the non-monotonic X-point $p=0$ in the system 
$[-2,0,x,1]$ induces the M\"{o}bius transformation $\mu(x)=-2x$, which is not an involution.
 
\section{The system $[0,x,1,1]$}\label{section:0x11}

In this section we shall describe a simple model for dynamics near an X-point which will be the basis for
the theory for dynamics near a general X-point in the next section. The model is the initial bundle $[0,x,1,1]$, $x\in\mathbb{R}$, with its X-point $0=Y_1\bowtie Y_2$.

For this bundle, we will prove that the strong terminating conjecture holds globally (theorem \ref{thm:0x11}) without using the computer algorithm of \cite{CellarosiMunday}. First, we show that the bundle possesses a global self-equivalence (lemma \ref{lemma:symmetry0x11}) which enables us to restrict our attention only to the right-hand side of the X-point. Next, we give an analytical description of its dynamics (figure \ref{fig:Tree}) which exposes a dichotomy: at the 
X-point $0$, the limit function has either a local minimum or a discontinuity
(lemma \ref{lemma:0x11}). Finally, a computer is used to establish the first possibility 
by evaluating the limit function at a single judiciously chosen point.

To begin, let us write $\hat{\Xi}(x):=[0,x,1,1]$, where $x\in\mathbb{R}$. The bundle $\hat{\Xi}$ possesses the following global self-equivalence.

\begin{lemma}\label{lemma:symmetry0x11}
The self-equivalence $\hat{\Xi}\sim\hat{\Xi}$ holds globally via
\begin{equation}\label{eq:symmetry0x11}
\mu(x)=\frac{x}{x-1}\qquad\text{and}\qquad f(z)=\frac{1}{1-x}z-\frac{x}{1-x}.
\end{equation}
\end{lemma}
\begin{proof}
For every $x\in\mathbb{R}$, we have\footnote{We write $[a,b,c]\pm x$ to mean $[a\pm x,b\pm x,c\pm x]$.}
\begin{eqnarray*}
\left[0,\frac{x}{x-1},1,1\right]&=&\frac{1}{1-x}\left[0,-x,1-x,1-x\right]\\
                                &=&\frac{1}{1-x}\left([x,0,1,1]-x\right)\\
                                &=&\frac{1}{1-x}\left[0,x,1,1\right]-\frac{x}{1-x},
\end{eqnarray*}
and hence the result.
\end{proof}

We are now ready to prove that the strong terminating conjecture holds globally for $\hat{\Xi}$.
The heart of our argument is the following lemma, which characterises the limit function of 
$\hat{\Xi}$ in the form of a dichotomy. The proof of this lemma is analytical,
providing a complete description of the dynamics of $\hat{\Xi}$ near the 
X-point $0=Y_1\bowtie Y_2$ (figure \ref{fig:Tree}). 
This proof will serve as the basis for the proof of the more general theorem \ref{thm:Generalrk2} in the next section.

\begin{lemma}\label{lemma:0x11}
For the initial bundle $\hat{\Xi}$, exactly one of the following holds:
\begin{itemize}
\item [i)] $m$ is continuous at $x=0$, and there exists $p_\infty\in (0,1)$ such that
$$m(x)=\begin{cases}
\frac{1}{2p_\infty}x+\frac{1}{2}&\text{if }0\leqslant x<p_\infty\\
\frac{1}{2\mu(p_\infty)}x+\frac{1}{2}&\text{if }\mu(p_\infty)<x<0\\
1&\text{otherwise;}
\end{cases}$$
\item [ii)] $m$ is discontinuous at $x=0$ and
$$m(x)=\begin{cases}
\frac{1}{2}&\text{if }x=0\\
1&\text{otherwise.}
\end{cases}$$
\end{itemize}
\end{lemma}
\begin{proof}
Let $\Xi_4:=\hat{\Xi}$. For every $n\geqslant4$, let
$$Y_{n+1}:=\rM\left(\Xi_n\right)\qquad\qquad\text{and}\qquad\qquad\Xi_{n+1}:=\mathbf{M}\left(\Xi_n\right).$$
Moreover, for every $n\geqslant4$, let 
$$p_n:=\max\left\{\overline{p}\in\mathbb{Q}\cup\{\infty\}:\Xi_n\text{ is regular in }\left(0,\overline{p}\right)\right\}$$
and
$$q_n:=\max\left\{\overline{q}\in\mathbb{Q}\cup\{\infty\}:\Lambda_n\text{ is regular in }\left(0,\overline{q}\right)\right\},$$
where $\Lambda_n$ denotes the core of $\Xi_n$, so that
$$U_n:=\left(0,p_n\right)\qquad\qquad\text{and}\qquad\qquad V_n:=\left(0,q_n\right)$$
are the domains of regularity of the bundle and the core, respectively, at time step $n$. In addition, let
$$I_n:=\left\{x>0:Y_n(x)=1\right\},\qquad\text{for every }n\geqslant5,$$
and
$$H_n:=\bigcup_{i=5}^n I_i,\qquad\text{for every }n\geqslant6.$$

Let us now study how $U_n$ and $V_n$ shrink as $n$ increases. First, it is clear that $p_4=\infty$ and $q_4=1$, and after direct calculation, $p_5=1$ and $q_5=\frac{1}{3}$. 
In the next step, we take advantage of the fact that the rational set $\hat{\Xi}(0)$ stabilises immediately at $\frac{1}{2}$, which means that for every $n\geqslant5$, the function $Y_n$ passes through the point $o:=\left(0,\frac{1}{2}\right)$.

Now let $n\geqslant 5$. Our aim is to express $p_{n+1}$ and $q_{n+1}$ in terms of $p_n$ and $q_n$. For this purpose, we shall focus our attention to the interval $V_n$ and divide the analysis into two cases according to the parity of $n$:\medskip

\begin{figure}
\centering
\begin{tabular}{ccc}
\begin{tikzpicture}
\begin{axis}[
	xmin=0,
	xmax=1.066666667,
	ymin=0,
	ymax=2.7,
    xtick={0.0769230769230769,0.33333333333,1},
    ytick={0.5,1},
    xticklabels={$q_n$,$p_n$},
    yticklabels={$\frac{1}{2}$,$1$},
	axis lines=middle,
	samples=100,
	xlabel=$x$,
	ylabel=$Y_n(x)$,
	width=12.5/16*8cm,
	height=12.5/16*6cm,
	clip=false,
	axis lines=middle,
    x axis line style=->,
    y axis line style=->,
]
\addplot [thick,domain=0:1] {0};
\addplot [thick,domain=0:1] {1};
\addplot [thick,domain=0:1] {1.065};
\addplot [thick,domain=0:1] {3/2*x+1/2};
\addplot [thick,domain=0:0.33333333333] {13/2*x+1/2}; 
\addplot [thick,domain=0.33333333333:1] {7/2-5/2*x}; 
\node[right] at (axis cs:1,2) {$Y_i$};
\node[left] at (axis cs:0.33333333333,2.666666666666667) {$Y_j$};
\draw[dashed] (axis cs:0.33333333333,0) -- (axis cs:0.33333333333,2.666666666666667);
\draw[dashed] (axis cs:0.0769230769230769,0) -- (axis cs:0.0769230769230769,1);
\draw [color=cyan,very thick] (axis cs:0,0.5) -- (axis cs:0.33333333333,1);
\draw [color=cyan,very thick] (axis cs:0,0.5) -- (axis cs:0.0769230769230769,1) -- (axis cs:1,1);
\draw [color=cyan,very thick] (axis cs:0.33333333333,1.065) -- (axis cs:1,1.065);
\draw [color=red,dashed,ultra thick] (axis cs:0,0.5) -- (axis cs:0.0769230769230769,0.8076923076923077) -- (axis cs:0.33333333333,1) -- (axis cs:1,1);
\draw[dashed] (axis cs:1,0) -- (axis cs:1,2);
\end{axis}
\end{tikzpicture}&\phantom{a}&
\begin{tikzpicture}
\begin{axis}[
	xmin=0,
	xmax=1.066666667,
	ymin=0,
	ymax=2.7,
    xtick={0.0263157894736842,0.0769230769230769,1},
    ytick={0.5,1},
    xticklabels={$q_n$,$p_n$},
    yticklabels={$\frac{1}{2}$,$1$},
	axis lines=middle,
	samples=100,
	xlabel=$x$,
	ylabel=$Y_n(x)$,
	width=12.5/16*8cm,
	height=12.5/16*6cm,
	clip=false,
	axis lines=middle,
    x axis line style=->,
    y axis line style=->,
]
\addplot [thick,domain=0:1] {0};
\addplot [thick,domain=0:1] {1};
\addplot [thick,domain=0:1] {1.065};
\addplot [thick,domain=0:1] {3/2*x+1/2};
\addplot [thick,domain=0:0.33333333333] {13/2*x+1/2}; 
\addplot [thick,domain=0.33333333333:1] {7/2-5/2*x}; 
\addplot [thick,domain=0.33333333333:1] {1}; 
\addplot [thick,domain=0.0769230769230769:0.33333333333] {-15/4*x+9/4}; 
\addplot [thick,domain=0:0.0769230769230769] {19*x+1/2}; 
\node[right] at (axis cs:1,2) {$Y_i$};
\node[left] at (axis cs:0.33333333333,2.666666666666667) {$Y_j$};
\node[above] at (axis cs:0.09,1.9) {$Y_k$};
\draw[dashed] (axis cs:0.0769230769230769,0) -- (axis cs:0.0769230769230769,1.961538462);
\draw[dashed] (axis cs:0.0263157894736842,0)--(axis cs:0.0263157894736842,1);
\draw [color=cyan,very thick] (axis cs:0,0.5) -- (axis cs:0.0263157894736842,1) -- (axis cs:1,1);
\draw [color=cyan,very thick] (axis cs:0,0.5) -- (axis cs:0.0769230769230769,1);
\draw [color=cyan,very thick] (axis cs:0.0769230769230769,1.065) -- (axis cs:1,1.065);
\draw [color=cyan,very thick] (axis cs:0,0.5) -- (axis cs:0.33333333333,1);
\draw [color=cyan,very thick] (axis cs:0.33333333333,1.0325) -- (axis cs:1,1.0325);
\draw [color=red,dashed,ultra thick] (axis cs:0,0.5) -- (axis cs:0.0769230769230769,1) -- (axis cs:1,1);
\draw[dashed] (axis cs:1,0) -- (axis cs:1,2);
\end{axis}
\end{tikzpicture}
\end{tabular}
\caption{\label{fig:0x11}\rm\small Even-to-odd iteration (left) and odd-to-even iteration (right).}
\end{figure}
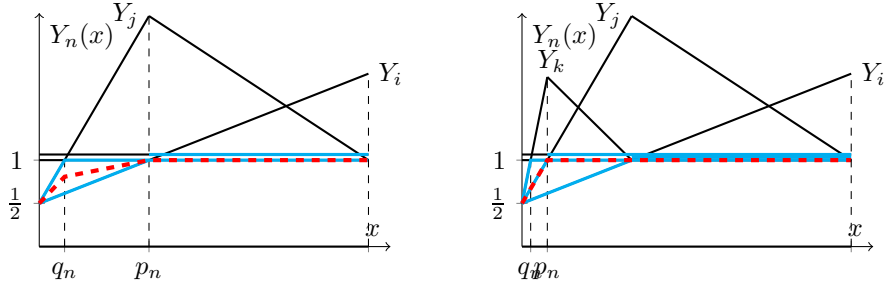

\noindent\textsc{\underline{Case I}}: $n$ is even. Let $\Lambda_n:=\left[Y_i,Y_j\right]$, where $Y_i\leqslant Y_j$, as in Figure \ref{fig:0x11} (left). Then $Y_i\left(p_n\right)=Y_j\left(q_n\right)=1$. Moreover, we have $\mathcal{M}_n=\left\langle Y_i,Y_j\right\rangle$ and $\mathcal{M}_{n-1}=Y_i$, and so by \eqref{eq:mmm2},
$$Y_{n+1}=\frac{n-1}{2}\left(Y_j-Y_i\right)+Y_j.$$
If $Y_i=Y_j$, then $Y_{n+1}=\mathcal{M}_n$, which means that the functional orbit \textit{stabilises}, and
\begin{equation}\label{eq:case1}
p_{n+1}=q_n\hspace{2cm}\text{and}\hspace{2cm}q_{n+1}=q_n.
\end{equation}
If $Y_i<Y_j$, then $Y_{n+1}$ is a new function lying above $Y_j$. Therefore,
\begin{equation}\label{eq:case2}
p_{n+1}=q_n\hspace{2cm}\text{and}\hspace{2cm}q_{n+1}=\max\left(H_{n+1}\cap U_{n+1}\right).
\end{equation}\smallskip

\noindent\textsc{\underline{Case II}}: $n$ is odd. Let $\Lambda_n=\left[Y_i,Y_j,Y_k\right]$, 
where $Y_i\leqslant Y_j\leqslant Y_k$, as in Figure \ref{fig:0x11} (right). 
Then $Y_j\left(p_n\right)=Y_k\left(q_n\right)=1$. If $Y_j=Y_k$, then $\mathcal{M}_{n+1}=Y_j=\mathcal{M}_{n+2}$, so the functional orbit \textit{stabilises}, and
\begin{equation}\label{eq:case3}
p_{n+1}=q_n\hspace{2cm}\text{and}\hspace{2cm}q_{n+1}=q_n.
\end{equation}
Now let $Y_j<Y_k$. Then $\mathcal{M}_n=Y_j$ and $\mathcal{M}_{n-1}=\left\langle Y_i,Y_j\right\rangle$, and so by \eqref{eq:mmm2},
$$Y_{n+1}=\frac{n}{2}\left(Y_j-Y_i\right)+Y_j.$$
If $Y_i=Y_j$, then $Y_{n+1}=\mathcal{M}_n$, which means that the functional orbit \textit{stabilises}, and
\begin{equation}\label{eq:case4}
p_{n+1}=p_n\hspace{2cm}\text{and}\hspace{2cm}q_{n+1}=p_n.
\end{equation}
If $Y_i<Y_j$, then $Y_{n+1}$ is a new function lying above $Y_j$. If $Y_{n+1}<Y_k$, then
\begin{equation}\label{eq:case5}
p_{n+1}=p_n\hspace{2cm}\text{and}\hspace{2cm}q_{n+1}=\min I_{n+1}.
\end{equation}
If $Y_{n+1}\geqslant Y_k$, then
\begin{equation}\label{eq:case6}
p_{n+1}=p_n\hspace{2cm}\text{and}\hspace{2cm}q_{n+1}=q_n.
\end{equation}\medskip

\noindent Therefore, there are six different cases corresponding to six possible pairs of expressions of $p_{n+1}$ and $q_{n+1}$ in terms of $p_n$ and $q_n$. 
These are summarised in figure \ref{fig:Tree}.

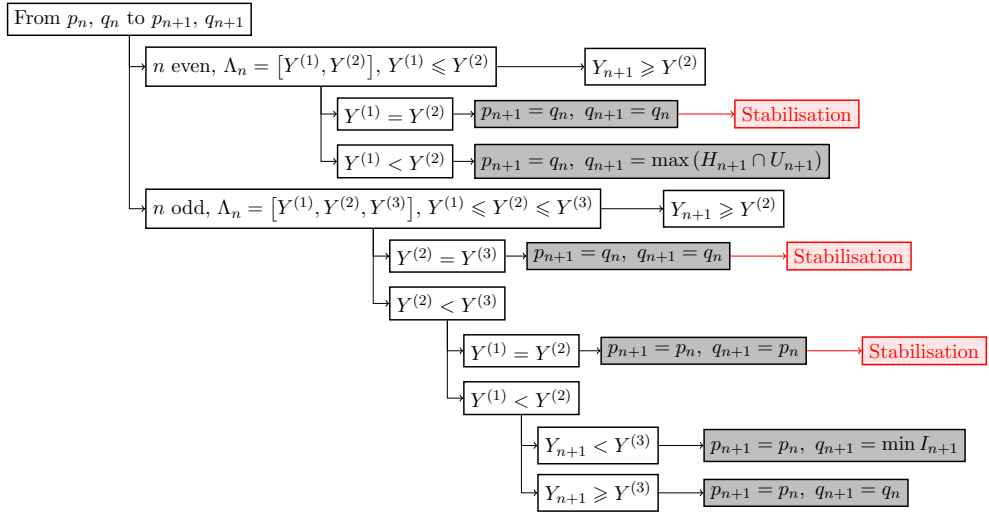
\begin{figure}
\centering
\tikzstyle{every node}=[draw=black,thick,anchor=west]
\tikzstyle{optional}=[draw=black,thick,fill=gray!50]
\tikzstyle{stabilisation}=[draw=red,thick,fill=red!10]
\scalebox{0.7}{\begin{tikzpicture}[
  grow via three points={one child at (0.3,-0.9) and
  two children at (0.3,-0.9) and (0.3,-1.8)},
  edge from parent path={[->] (\tikzparentnode.south) |- (\tikzchildnode.west)}
  ]
  \node {From $p_{n}$, $q_{n}$ to $p_{n+1}$, $q_{n+1}$}
    child { node {$n$ even, $\Lambda_n=\left[Y^{(1)},Y^{(2)}\right]$, $Y^{(1)}\leqslant Y^{(2)}$}
      child { node {$Y^{(1)}=Y^{(2)}$} child[grow=right,edge from parent path={[->] (\tikzparentnode.east) |- (\tikzchildnode.west)}] {node[optional] {$p_{n+1}=q_n$,\,\,\,$q_{n+1}=q_n$} child[grow=right,level distance=30mm,edge from parent path={[->,red] (\tikzparentnode.east) |- (\tikzchildnode.west)}] {node[stabilisation] {\textcolor{red}{Stabilisation}}}}
      }
      child { node {$Y^{(1)}<Y^{(2)}$} child[grow=right,edge from parent path={[->] (\tikzparentnode.east) |- (\tikzchildnode.west)}] {node[optional] {$p_{n+1}=q_n$,\,\,\,$q_{n+1}=\max \left(H_{n+1}\cap U_{n+1}\right)$}}
      }
      child[grow=right,level distance=50mm,edge from parent path={[->] (\tikzparentnode.east) |- (\tikzchildnode.west)}] {node {$Y_{n+1}\geqslant Y^{(2)}$}}
    }
    child [missing] {}
    child [missing] {}
    child { node {$n$ odd, $\Lambda_n=\left[Y^{(1)},Y^{(2)},Y^{(3)}\right]$, $Y^{(1)}\leqslant Y^{(2)}\leqslant Y^{(3)}$}
      child { node {$Y^{(2)}=Y^{(3)}$} child[grow=right,edge from parent path={[->] (\tikzparentnode.east) |- (\tikzchildnode.west)}] {node[optional] {$p_{n+1}=q_n$,\,\,\,$q_{n+1}=q_n$} child[grow=right,level distance=30mm,edge from parent path={[->,red] (\tikzparentnode.east) |- (\tikzchildnode.west)}] {node[stabilisation] {\textcolor{red}{Stabilisation}}}}
      }
      child { node {$Y^{(2)}<Y^{(3)}$}    
         child { node {$Y^{(1)}=Y^{(2)}$} child[grow=right,edge from parent path={[->] (\tikzparentnode.east) |- (\tikzchildnode.west)}] {node[optional] {$p_{n+1}=p_n$,\,\,\,$q_{n+1}=p_n$} child[grow=right,level distance=30mm,edge from parent path={[->,red] (\tikzparentnode.east) |- (\tikzchildnode.west)}] {node[stabilisation] {\textcolor{red}{Stabilisation}}}}
      }
         child { node {$Y^{(1)}<Y^{(2)}$}
            child { node {$Y_{n+1}<Y^{(3)}$} child[grow=right,level distance=20mm,edge from parent path={[->] (\tikzparentnode.east) |- (\tikzchildnode.west)}] {node[optional] {$p_{n+1}=p_n$,\,\,\,$q_{n+1}=\min I_{n+1}$}}
      }
            child { node {$Y_{n+1}\geqslant Y^{(3)}$} child[grow=right,level distance=20mm,edge from parent path={[->] (\tikzparentnode.east) |- (\tikzchildnode.west)}] {node[optional] {$p_{n+1}=p_n$,\,\,\,$q_{n+1}=q_n$}}
      }
    }}
          child[grow=right,level distance=55mm,edge from parent path={[->] (\tikzparentnode.east) |- (\tikzchildnode.west)}] {node {$Y_{n+1}\geqslant Y^{(2)}$}}};
\end{tikzpicture}}
\caption{\label{fig:Tree}\rm\small The flow chart for computing $p_{n+1}$, $q_{n+1}$ from $p_n$, $q_n$.}
\end{figure}

We are now ready to determine the limit function of this bundle. By lemma \ref{lemma:symmetry0x11}, we have that $\hat{\Xi}\sim\hat{\Xi}$ via $\mu(x)=\frac{x}{x-1}$ and $f(z)=\frac{1}{1-x}z-\frac{x}{1-x}$. Here $\mu$ is an involution which maps the interval $(1,2)$ to $(2,\infty)$ and the interval $(0,1)$ to $(-\infty,0)$. Taking advantage of the former, after straightforward computations we obtain that
\begin{equation}\label{eq:limitin[1,infty)}
m(x)=1
\end{equation}
and
$$\tau(x)=\begin{cases}
6&\text{if }x\neq 2\\
5&\text{if }x=2
\end{cases}$$
for every $x\in[1,\infty)$.

Since we already know that $m(0)=\frac{1}{2}$ and $\tau(0)=4$, it now remains to consider 
the case $x\in(0,1)$. Notice that the sequence $\left(p_n\right)_{n=5}^\infty$ 
is non-increasing and bounded below by $0$, and therefore it converges, say to $p^+_\infty\in [0,1)$. 
Moreover, every term of the sequence $\left(p_n\right)_{n=6}^\infty$ creates 
jump discontinuities of the transit time as described in lemma \ref{lemma:Type3}. 
Now there are two cases:\medskip

\noindent\textsc{\underline{Case I}}: $p^+_\infty\in(0,1)$. In this case the sequence $\left(p_n\right)_{n=5}^\infty$ either stabilises or converges without stabilising, and so the median sequence either stabilises at or converges without stabilising to
\begin{equation}\label{eq:m}
m(x)=\begin{cases}
\frac{1}{2p^+_\infty}x+\frac{1}{2}&\text{if }0\leqslant x<p^+_\infty\\
1&\text{if }x\geqslant p^+_\infty.
\end{cases}
\end{equation}\smallskip

\noindent\textsc{\underline{Case II}}: $p^+_\infty=0$. In this case the sequence $\left(p_n\right)_{n=5}^\infty$ converges without stabilising and the median sequence converges without stabilising to
\begin{equation}\label{eq:m2}
m(x)=\begin{cases}
0&\text{if }x=0\\
1&\text{if }x>0.
\end{cases}
\end{equation}\medskip

\noindent Applying the self-equivalence to \eqref{eq:m} and \eqref{eq:m2} gives the possible expressions 
of $m(x)$ for $x\in(-\infty,0)$. Combining these with \eqref{eq:limitin[1,infty)} gives the expressions in the lemma, thereby completing the proof.
\end{proof}

Now, a single computer-aided evaluation of \mmm\ orbit shows that, e.g., for $x_0=10^{-4}$, the set 
$\hat{\Xi}\left(x_0\right)$ stabilises at $m\left(x_0\right)=\frac{2597}{5000}$ with transit time 
$\tau\left(x_0\right)=63$. 
Since the point $\left(x_0,m\left(x_0\right)\right)$ is not on the auxiliary function, 
it follows that the first assertion of lemma \ref{lemma:0x11} is true, and hence we have a 
bounded transit time as seen in figure \ref{fig:tau0x11}. 
Moreover, by computing the abscissa of the point at which the line through $o$ and 
$\left(x_0,m\left(x_0\right)\right)$ intersects the auxiliary function, one obtains that 
$p_\infty^+=\frac{1}{388}$. Finally, we have that $p_\infty^-=\mu\left(p_\infty^+\right)=-\frac{1}{387}$, 
implying that the explicit formula of the limit function of this initial bundle is in fact as follows.\medskip

\begin{theorem}\label{thm:0x11}
The strong terminating conjecture holds globally for the system $[0,x,1,1]$. 
Specifically, for every $x\in\mathbb{R}$ we have that
$\tau(x)\leqslant 63$ and 
$$
m(x)=Y_{63}(x)=
\begin{cases}
1&\text{if }x\leqslant-\frac{1}{387}\\
-\frac{387}{2}x+\frac{1}{2}&\text{if }-\frac{1}{387}<x<0\\
194x+\frac{1}{2}&\text{if }0\leqslant x<\frac{1}{388}\\
1&\text{if }x\geqslant\frac{1}{388}.
\end{cases}
$$
\end{theorem}

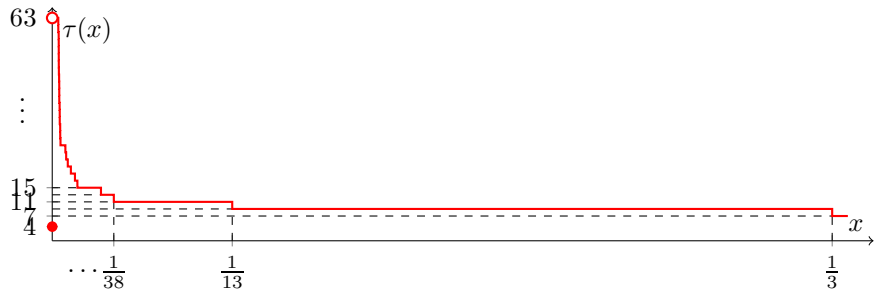
\begin{figure}
\centering
\begin{tikzpicture}
\begin{axis}[
	xmin=0,
	xmax=0.3509677419,
	ymin=0,
	ymax=66,
    xtick={0.2632e-1,0.7692e-1,0.3333},
    ytick={4,7,11,15,63},
    xticklabels={$\frac{1}{38}$,$\frac{1}{13}$,$\frac{1}{3}$},
	axis lines=middle,
    x axis line style=->,
    y axis line style=->,
	xlabel=$x$,
	ylabel=$\tau(x)$,
    width=12.5/16*16cm,
    height=12.5/16*6cm,
    clip=false
]
\node[left,xshift=-6pt] at (axis cs:0,39) {$\vdots$};
\node[below,yshift=-7pt] at (axis cs:0.01315789474,0) {$\ldots$};
\draw[dashed] (axis cs:0,7)-- (axis cs:0.3333,7);
\draw[dashed] (axis cs:0.3333,0)-- (axis cs:0.3333,7);
\draw[dashed] (axis cs:0,9)-- (axis cs:0.7692e-1,9);
\draw[dashed] (axis cs:0.7692e-1,0)-- (axis cs:0.7692e-1,9);
\draw[dashed] (axis cs:0,11)-- (axis cs:0.2632e-1,11);
\draw[dashed] (axis cs:0.2632e-1,0)-- (axis cs:0.2632e-1,11);
\draw[dashed] (axis cs:0,13)-- (axis cs:0.2083e-1,13);
\draw[dashed] (axis cs:0,15)-- (axis cs:0.1075e-1,15);
\draw[thick,red] plot coordinates  { (axis cs:-0.2592e-2, 63) (axis cs:0.2586e-2, 63) (axis cs:0.2586e-2, 61) (axis cs:0.2594e-2, 61) (axis cs:0.2594e-2, 59) (axis cs:0.2755e-2, 59) (axis cs:0.2755e-2, 57) (axis cs:0.2779e-2, 57) (axis cs:0.2779e-2, 55) (axis cs:0.2788e-2, 55) (axis cs:0.2788e-2, 53) (axis cs:0.2798e-2, 53) (axis cs:0.2798e-2, 51) (axis cs:0.2808e-2, 51) (axis cs:0.2808e-2, 49) (axis cs:0.2843e-2, 49) (axis cs:0.2843e-2, 47) (axis cs:0.2937e-2, 47) (axis cs:0.2937e-2, 45) (axis cs:0.2992e-2, 45) (axis cs:0.2992e-2, 43) (axis cs:0.3003e-2, 43) (axis cs:0.3003e-2, 41) (axis cs:0.3014e-2, 41) (axis cs:0.3014e-2, 39) (axis cs:0.3157e-2, 39) (axis cs:0.3157e-2, 37) (axis cs:0.3201e-2, 37) (axis cs:0.3201e-2, 35) (axis cs:0.3214e-2, 35) (axis cs:0.3214e-2, 33) (axis cs:0.3398e-2, 33) (axis cs:0.3398e-2, 31) (axis cs:0.3413e-2, 31) (axis cs:0.3413e-2, 29) (axis cs:0.3534e-2, 29) (axis cs:0.3534e-2, 27) (axis cs:0.5618e-2, 27) (axis cs:0.5618e-2, 25) (axis cs:0.5952e-2, 25) (axis cs:0.5952e-2, 23) (axis cs:0.6645e-2, 23) (axis cs:0.6645e-2, 21) (axis cs:0.7968e-2, 21) (axis cs:0.7968e-2, 19) (axis cs:0.9709e-2, 19) (axis cs:0.9709e-2, 17) (axis cs:0.1075e-1, 17) (axis cs:0.1075e-1, 15) (axis cs:0.2083e-1, 15) (axis cs:0.2083e-1, 13) (axis cs:0.2632e-1, 13) (axis cs:0.2632e-1, 11) (axis cs:0.7692e-1, 11) (axis cs:0.7692e-1, 9) (axis cs:.3333, 9) (axis cs:.3333, 7) (axis cs:0.34,7)};
\draw[thick,color=red,fill=white] (axis cs:0,63) circle (2pt);
\fill[red] (axis cs:0,4) circle (2pt);
\end{axis}
\end{tikzpicture} 
\caption{\label{fig:tau0x11}\small\rm The transit time of $\hat{\Xi}$ for $x>0$.}
\end{figure}
 
\section{Dynamics near active X-points}\label{section:X-points}

We now turn to the analysis of the \mmm\ dynamics near X-points. First, we develop suitable conditions under which the symmetry near a proper active X-point is inherited by the limit function (lemmas \ref{lemma:affinecombinations} and \ref{lemma:independencefromhistory}, theorem \ref{theorem:inheritance}). Then we use our knowledge of the system $[0,x,1,1]$ to describe the limit function near any active X-point, dividing the analysis into two cases according to the rank of the X-point. Since the X-point $0$ in the system $[0,x,1,1]$ is of rank $2$, we first deal with general X-points of rank $2$ or higher (theorem \ref{thm:Generalrk2}). Then we extend the result to X-points of rank $1$ (theorems \ref{thm:AuxSeq} and \ref{thm:Generalrk1}). Finally, we justify the assumptions of our theorems by illustrating various pathologies.

In section \ref{section:Symmetries}, any function not through an X-point could be its auxiliary function. In this section, given an active X-point, we need to choose a specific auxiliary function which will be relevant to the future dynamics near the X-point. For this purpose, we define the \textit{standard auxiliary function} and the \textit{standard triad} of an active X-point $p=Y_i\bowtie Y_j$ ---with a non-decreasing median sequence in its vicinity--- to be $\min\mathcal{Y}_p$ and $\left[Y_i,Y_j;\min\mathcal{Y}_p\right]$, respectively, where $\mathcal{Y}_p$ is defined by \eqref{eq:SetOfFunctionsImmediatelyAbove}. Notice that the standard triad depends on $Y_i$ and $Y_j$, whereas the standard auxiliary function does not; this is important if $p$ is not proper.

\subsection{Inheritance of symmetries and independence from previous history}\label{section:Inheritance}

In the bundle $\hat{\Xi}$, the standard triad of the proper active X-point $0$ comprises the entire bundle (ignoring multiplicities), so it is obvious that the triad symmetry given by lemma \ref{lemma:symmetry0x11} is inherited by the orbit of the bundle, leading to the functional equation \eqref{eq:FunctionalEquation}. In general, however, we work with a triad which need not be the whole bundle, so whether the symmetry is inherited is not immediately apparent.

More precisely, consider a neighbourhood of an X-point $p$ of a triad $\Omega=\left[Y_i,Y_j;Y\right]$ whose symmetry induces a self-equivalence via the pair $(\mu,f)$. Our goal is to establish a sufficient condition under which \eqref{eq:FunctionalEquation} holds near $p$ for the same pair $(\mu,f)$. For this purpose, we will prove two lemmas. The first lemma characterises functions $W$ which satisfy
\begin{equation}\label{eq:W}
W(\mu(x))=f\left(W(x)\right)
\end{equation}
for the same pair $(\mu,f)$. It turns out that these are precisely the \textit{affine combinations}\footnote{Linear combinations with coefficients adding up to unity \cite[page 428]{Roman}.} of functions in $\Omega^\updownarrow:=\left[Y_i\lor Y_j,Y_i\land Y_j,Y\right]$.

\begin{lemma}\label{lemma:affinecombinations}
An affine function $W$ which is regular except, possibly, at $p$, satisfies \eqref{eq:W} if and only if $W$ is an affine combination of functions in $\Omega^\updownarrow$.
\end{lemma}
\begin{proof}
Write $U:=Y_i\lor Y_j$, $L:=Y_i\land Y_j$, and $f(z)=A(x)z+B(x)$ as in the proof of theorem \ref{thm:GeneralSymmetry}. 
First suppose $W=\alpha U+\beta L +\gamma Y$, where $\alpha+\beta+\gamma=1$. Then by theorem \ref{thm:GeneralSymmetry} we have
\begin{eqnarray*}
W(\mu(x))&=&\alpha U(\mu(x)) +\beta L(\mu(x)) + \gamma Y(\mu(x))\\
         &=& \alpha f(U(x))+\beta f(L(x)) + \gamma f(Y(x))\\
         &=&\alpha[A(x)U(x)+B(x)] +\beta[A(x)L(x)+B(x)] +\gamma[A(x)Y(x)+B(x)]\\
         &=&A(x)[\alpha U(x)+\beta L(x)+\gamma Y(x)] +B(x)(\alpha+\beta+\gamma)\\
         &=&A(x)W(x)+B(x)\\
         &=&f(W(x)).
\end{eqnarray*}
Conversely, suppose that $W$ is a piecewise-affine function with
\begin{equation}\label{eq:equivalenceW}
W(\mu(x))=f(W(x))=A(x)W(x)+B(x)
\end{equation}
for all $x$ sufficiently close to $p$. We shall prove that $W$ is an affine combination of $U$, $L$, and $Y$, i.e., that there exists $(\alpha,\beta)\in\mathbb{R}^2$ such that the functional identity
$$W=\alpha U+\beta L+(1-\alpha-\beta)Y,$$
holds in a neighbourhood of $p$. This identity can be written as
$$
\left\{\begin{array}{rcl}
W(x)&=&\alpha U(x)+\beta L(x)+(1-\alpha-\beta) Y(x)\\
W(\mu(x))&=&\alpha U(\mu(x))+\beta L(\mu(x))+(1-\alpha-\beta) Y(\mu(x)),
\end{array}\right.
$$
i.e.,
\begin{equation}\label{1}
\left\{\begin{array}{rcl}
Y(x)-W(x)&=&[Y(x)-U(x)]\alpha+[Y(x)-L(x)]\beta\\
Y(\mu(x))-W(\mu(x))&=&[Y(\mu(x))-U(\mu(x))]\alpha\\&&+\,\,\,[Y(\mu(x))-L(\mu(x))]\beta.
\end{array}\right.
\end{equation}
But since near $p$ we have
\begin{eqnarray*}
U(\mu(x))&=&A(x)U(x)+B(x),\\
L(\mu(x))&=&A(x)L(x)+B(x),\\
Y(\mu(x))&=&A(x)Y(x)+B(x),
\end{eqnarray*}
which, together with \eqref{eq:equivalenceW}, imply
\begin{eqnarray*}
Y(\mu(x))-W(\mu(x))&=&A(x)[Y(x)-W(x)],\\
Y(\mu(x))-U(\mu(x))&=&A(x)[Y(x)-U(x)],\\
Y(\mu(x))-L(\mu(x))&=&A(x)[Y(x)-L(x)],
\end{eqnarray*}
then the second equation in \eqref{1} is redundant (it is a multiple of the first equation), so it remains to show that the first equation has a solution $(\alpha,\beta)\in\mathbb{R}^2$. For this purpose, it suffices to show that $Y-U$ is not a multiple of $Y-L$. Suppose for a contradiction that there exists $k\in\mathbb{Q}$ such that $Y-U=k(Y-L)$. If $k=1$, then $U=L$, which is a contradiction. Otherwise, we have
$$Y=\frac{1}{1-k}U-\frac{k}{1-k}L,$$
which means that the auxiliary function is an affine combination of the upper and lower concatenations of the X-point functions. 
This is also a contradiction because the auxiliary function does not pass through the X-point.
\end{proof}

Now suppose the X-point $p=Y_i\bowtie Y_j$ is proper and active (the median sequence being non-decreasing in its vicinity), and the triad $\Omega$ is standard. We prove that, starting at the time step at which $p$ stabilises, every function generated by the \mmm\ is an affine combination of functions in $\Omega^\updownarrow$, provided that the first one is. In this sense, the functional dynamics near $p$ now depends only on functions in $\Omega^\updownarrow$, rather than on all functions in its previous history.

\begin{lemma}\label{lemma:independencefromhistory}
Let $p$ be a proper active X-point with standard triad $\Omega$. If $Y_{\tau(p)}$ is an affine combination of functions in $\Omega^\updownarrow$, then so is the function $Y_n$ for every $n\geqslant\tau(p)$.
\end{lemma}
\begin{proof}
Let $p$ be a proper active X-point with standard triad $\Omega$. Assume that $Y_{\tau(p)}$ is an affine combination of functions in $\Omega^\updownarrow$. We use strong induction to prove that for every $n\geqslant\tau(p)$, the function $Y_n$ is also an affine combination of functions in $\Omega^\updownarrow$.

The base case is a part of the assumption. Now let $n\geqslant \tau(p)$ be such that every function in the set $\left[Y_{\tau(p)},\ldots,Y_n\right]$ is an affine combination of functions in $\Omega^\updownarrow$. Then, from \eqref{eq:mmm2} we find that $Y_{n+1}=(n+1)\mathcal{M}_{n}-n\mathcal{M}_{n-1}$ is an affine combination of $\mathcal{M}_{n}$ and $\mathcal{M}_{n-1}$, each of which is either a function in the set $\left[Y_{\tau(p)},\ldots,Y_n\right]\uplus \Omega^\updownarrow$ or the arithmetic mean of two such functions, and hence is an affine combination of functions in $\Omega^\updownarrow$. This implies that $Y_{n+1}$ is an affine combination of functions in $\Omega^\updownarrow$, as easily verified. Therefore, the induction is complete.
\end{proof}

\begin{algorithm}\label{remark:independencefromhistory}
Notice that lemma \ref{lemma:independencefromhistory} remains true if $\Omega^\updownarrow$ is replaced by $\Omega^\updownarrow\backslash[Y]$, where $Y$ is the standard auxiliary function of $p$, in which case $m(p)=Y_i(p)=Y_j(p)$. We will use this fact to prove proposition \ref{prop:equivalenceNF} in section \ref{section:ReducedSystem}.
\end{algorithm}

From the above two lemmas, it is clear that if $p$ is proper and active, and $Y_{\tau(p)}$ is an affine combination of functions in $\Omega^\updownarrow$, where $\Omega$ is the standard triad of $p$, then all functions generated after its stabilisation satisfy the desired equivalence. Therefore we have achieved the goal of this section.

\begin{theorem}\label{theorem:inheritance}
Let $p$ be a proper active X-point with standard triad $\Omega$. If $Y_{\tau(p)}$ is an affine combination of functions in $\Omega^\updownarrow$, then
\begin{equation}\label{eq:AffineCombinationY}
Y_n(\mu(x))=f\left(Y_n(x)\right)
\end{equation}
for every $n\geqslant\tau(p)$, and so \eqref{eq:FunctionalEquation} holds, and
\begin{equation}\label{eq:AffineCombinationTau}
\tau(\mu(x))=\tau(x),
\end{equation}
meaning that the functional orbit on the left-hand side stabilises if and only if that on the right-hand side stabilises.
\end{theorem}

We now understand why the structures of the limit function of the system $[0,x,1]$ near every rational number in $\left[\frac{1}{2},\frac{2}{3}\right]$ with denominator between $3$ and $18$ inclusive, as found in \cite[theorem 1.3]{CellarosiMunday}, are symmetric. As mentioned in section \ref{section:Preliminaries}, these numbers are all X-points (figure \ref{fig:proportionXpts}) which are proper and active. Moreover, they possess local symmetries described by theorem \ref{thm:GeneralSymmetry} ---as well as corollary \ref{cor:RegularX-point} since they are regular--- which are inherited to become local symmetries of the limit function.

\subsection{The limit function near an X-point of high rank}\label{section:Tractability}

In this subsection we generalise the dynamics near the X-point $0$ of rank $2$ in the bundle $\hat{\Xi}(x)=[0,x,1,1]$ which was discussed in the proof of lemma \ref{lemma:0x11}. More precisely, we will prove a general version of lemma \ref{lemma:0x11} which holds for active X-points of \textit{high rank}, i.e., rank at least $2$. To achieve this, we first need to identify a sufficient condition under which the dynamics near such an X-point is the exactly the same as the one described by figure \ref{fig:Tree}, without any unwanted interaction with earlier functions, i.e., those appearing before the X-point stabilises. We will work only on the right-hand side of the X-point (therefore by \textit{rank} we mean \textit{right-rank}), where the median sequence is assumed to be non-decreasing, and define some terms which can be defined analogously on the left-hand side. In future discussions, the prefixes \textit{right}- or \textit{left}- may be added to each of these terms to avoid ambiguity. An active X-point $p$ is said to be:
\begin{itemize}
\item \textit{tractable} if there exists an odd integer $\ell\geqslant \tau(p)$ such that the following three conditions are satisfied:
\begin{enumerate}
\item [\textbf{T1}] the median $\mathcal{M}_\ell$ meets the standard auxiliary function $Y$ on the right-hand side of the X-point, say at $p_*\in(p,\infty)$, and both $\mathcal{M}_\ell$ and $Y$ are regular in the interval $\mathcal{T}:=(p,p_\ast)$;
\item [\textbf{T2}] the interval $\mathcal{T}$ contains no point $Y\bowtie \overline{Y}$, where $\overline{Y}\in \Xi_{\tau(p)-1}$;
\item [\textbf{T3}] for every $n\in\{\tau(p),\ldots,\ell\}$, the function $Y_n$ has no corner below $Y$ in $\mathcal{T}$.
\end{enumerate}
The smallest such number $\ell$ and the corresponding interval $\mathcal{T}$ are referred to as the \textit{tractability index} and \textit{tractability domain} of the X-point, respectively.
\item \textit{dichotomic} in a neighbourhood\footnote{This neighbourhood may also be open, in which case the definition is modified by merely excluding $p$.} $\left[p,p_\ast\right)$ if the limit function $m:\left[p,p_\ast\right)\to\mathbb{R}$ is known to have exactly one of the following properties:
\begin{enumerate}
\item[\textrm{i)}] $m$ is continuous at $x=p$, and in $\left[p,p_\ast\right)$ we have
\begin{equation}\label{eq:Generalm1}
m(x)=\begin{cases}
\frac{Y\left(p_\infty\right)-m(p)}{p_\infty-p}x+\frac{p_\infty m(p)-pY\left(p_\infty\right)}{p_\infty-p}
&\text{if }p\leqslant x<p_\infty\\
Y(x)&\text{if }p_\infty\leqslant x<p_\ast,
\end{cases}
\end{equation}
for some $p_\infty\in\left(p,p_\ast\right]$.
\item[\textrm{ii)}] $m$ is discontinuous at $x=p$, and in $\left[p,p_\ast\right)$ we have
\begin{equation}\label{eq:Generalm2}
m(x)=\begin{cases}
m(p)&\text{if }x=p\\
Y(x)&\text{if }p<x<p_\ast.
\end{cases}
\end{equation}
\end{enumerate}
\end{itemize}
We shall see that the tractability conditions \textbf{T1}, \textbf{T2}, and \textbf{T3} suffice to guarantee that a rescaled version of the dynamics described by figure \ref{fig:Tree} takes place on the right-hand side the active X-point, resulting in a dichotomy for the limit function as in lemma \ref{lemma:0x11} in the tractability domain. More precisely, we shall prove the following theorem.

\begin{theorem}\label{thm:Generalrk2}
Any tractable X-point of rank at least $2$ is dichotomic in its tractability domain.
\end{theorem}
\begin{proof}
Let $p$ be a tractable X-point of rank at least $2$ with index $\ell$, domain $\left(p,p_\ast\right)$, and standard auxiliary function $Y$. Then every median $\mathcal{M}_{k}$, $k\geqslant\tau(p)-1$, and every function $Y_{k}$, $k\geqslant\tau(p)$, passes through the point $o:=(p,m(p))$. For every $n\geqslant\ell$, let
\begin{eqnarray*}
p_n&:=&\max\left\{\overline{p}\in\mathbb{Q}:\text{all functions in }\Xi_n\text{ passing through }o\text{ and not below }\mathcal{M}_\ell\right.\\
&&\left.\text{ are regular in }\left(p,\overline{p}\right)\right\},
\end{eqnarray*}
and
$$q_n:=\begin{cases}
p_\ast&\text{if }n=\ell\\
\max\left\{\overline{q}\in\mathbb{Q}:\Lambda_n\text{ is regular in }\left(p,\overline{q}\right)\right\}&\text{if }n\geqslant\ell+1,
\end{cases}$$
where $\Lambda_n$ denotes the core of $\Xi_n$, so that
\begin{equation}\label{eq:newDOR}
U_n:=\left(p,p_n\right)\qquad\qquad\text{and}\qquad\qquad V_n:=\left(p,q_n\right)
\end{equation}
are the domain of regularity of the participating functions in the bundle and that of the core, respectively, at time step $n$. In addition, let
\begin{equation}\label{eq:newIn}
I_n:=\left\{x>p:Y_n(x)=Y(x)\right\},\qquad\text{for every }n\geqslant\tau(p),
\end{equation}
and
\begin{equation}\label{eq:newHn}
H_n:=\bigcup_{i=\tau(p)}^n I_i,\qquad\text{for every }n\geqslant\ell+1.
\end{equation}

For every $n\geqslant\ell$, we express $p_{n+1}$ and $q_{n+1}$ in terms of $p_n$ and $q_n$ by dividing into cases in the same way as we have done for $\hat{\Xi}$, namely:\medskip

\noindent\textsc{\underline{Case I}}: $n$ is even. Letting $\Lambda_n=\left[Y_i,Y_j\right]$, where $Y_i\leqslant Y_j$, we obtain
$$Y_{n+1}=\frac{n-1}{2}\left(Y_j-Y_i\right)+Y_j,$$
which implies \eqref{eq:case1} if $Y_i=Y_j$ or \eqref{eq:case2} if $Y_i<Y_j$.\smallskip

\noindent\textsc{\underline{Case II}}: $n$ is odd. Letting $\Lambda_n=\left[Y_i,Y_j,Y_k\right]$, where
$Y_i\leqslant Y_j\leqslant Y_k$, we obtain
$$Y_{n+1}=\frac{n}{2}\left(Y_j-Y_i\right)+Y_j,$$
which implies \eqref{eq:case3} if $Y_j=Y_k$, \eqref{eq:case4} if $Y_i=Y_j$, \eqref{eq:case5} if
$Y_i<Y_j$ and $Y_{n+1}<Y_k$, or \eqref{eq:case6} if $Y_i<Y_j$ and $Y_{n+1}\geqslant Y_k$.\medskip

\noindent This division into cases is once again summarised by figure \ref{fig:Tree}, where the sets involved
are those defined in \eqref{eq:newDOR}, \eqref{eq:newIn}, and \eqref{eq:newHn}.

Since the sequence $\left(p_n\right)_{n=\ell}^\infty$ is non-increasing and bounded below by $p$, then it converges, 
say to $p_\infty\in\left[p,p_\ast\right]$. Moreover, every term of this sequence creates a jump discontinuity of the transit time as described in lemma \ref{lemma:Type3}. Now there are two cases. If $p_\infty>p$, then the sequence $\left(p_n\right)_{n=\ell}^\infty$ either stabilises or converges without stabilising, and so the median sequence either stabilises at or converges without stabilising to the function which emanates from $o$ and meets $Y$ at $p_\infty$, 
namely \eqref{eq:Generalm1}. If $p_\infty=p$, then the sequence $\left(p_n\right)_{n=\ell}^\infty$ converges without stabilising and the median sequence converges without stabilising to \eqref{eq:Generalm2}. The theorem is proved.
\end{proof}\vspace{0.5cm}

If an active X-point is dichotomic, a simple test adapted from the derivation of theorem \ref{thm:0x11} is applicable to ensure that the number $p_\infty$ exists, i.e., that the first assertion for the limit function given in the dichotomy is true. Our next goal is to establish a result similar to theorem \ref{thm:Generalrk2} for X-points of rank $1$.

\subsection{X-points of rank 1 and their auxiliary sequences}

In the tractability domain of a tractable X-point of high rank, every subsequent function intersects the standard auxiliary function at a new regular X-point. Thus, the system generates a sequence of secondary X-points lying on the standard auxiliary function, which we shall call the \textit{auxiliary sequence}. This sequence partitions the domain into 
adjacent subintervals, within which the dynamics has the simple structure prescribed by 
lemma \ref{lemma:Type3}, because the standard auxiliary function has high multiplicity. 

If the X-point is of rank $1$, then an auxiliary sequence of regular X-points is generated in exactly 
the same way. However, since the auxiliary function has unit multiplicity, between these
secondary X-points the limit function no longer has a trivial form.
In fact, as functions of high multiplicity are rare, each secondary X-point 
is typically of rank $1$, and hence in turn possesses its own auxiliary sequence of
tertiary X-points, each of which is typically of rank $1$, and so on. 
There is, therefore, a hierarchical organisation\footnote{Thus, given a system, one would like to construct a tree which contains all its X-points and describes the hierarchy rigorously. This task is not easy and hence remains open.} of X-points of rank $1$.

To be more precise, we work on the right-hand side of an X-point $p$ of (right-)rank $1$ which is tractable with index $\ell$ and domain $\left(p,p_\ast\right)$, where the median sequence is assumed to be non-decreasing. Let $o:=(p,m(p))$. The \textit{auxiliary sequence} of $p$ is the sequence $\left(p_n\right)_{n=\ell}^\infty$, where
\begin{eqnarray*}
p_n&:=&\max\left\{\overline{p}\in\mathbb{Q}:\text{all functions in }\Xi_n\text{ passing through }o\text{ and not below }\mathcal{M}_\ell\right.\\
&&\left.\text{ are regular in }\left(p,\overline{p}\right)\right\},
\end{eqnarray*}
for every $n\geqslant\ell$. Since $p$ is tractable, the evolution of this sequence, together with that of its partner sequence $\left(q_n\right)_{n=\ell}^\infty$, where
$$q_n:=\begin{cases}
p_\ast&\text{if }n=\ell\\
\max\left\{\overline{q}\in\mathbb{Q}\cup\{\infty\}:\Lambda_n\text{ is regular in }\left(p,\overline{q}\right)\right\}&\text{if }n\geqslant\ell+1,
\end{cases}$$
for every $n\geqslant\ell$, is described in figure \ref{fig:Tree}, where the sets involved are those defined in \eqref{eq:newDOR}, \eqref{eq:newIn}, and \eqref{eq:newHn}. Therefore, we know the following properties of the auxiliary sequence.

\begin{theorem}\label{thm:AuxSeq}
The auxiliary sequence of a tractable X-point of rank $1$ has the following properties:
\begin{enumerate}
\item[\textrm{i)}] The sequence is non-increasing.
\item[\textrm{ii)}] The sequence contains either infinitely or finitely many distinct regular X-points, each of which, except the last one in the latter case, is a local minimum of the limit function.
\item[\textrm{iii)}] The transit times of the distinct terms of this sequence form an increasing arithmetic progression of common difference $2$.
\item[\textrm{iv)}] The sequence converges to $p$ if and only if the limit function is discontinuous at $p$.
\end{enumerate}
\end{theorem}
\begin{proof}
Let $p$ be an X-point of rank $1$ which is tractable with index $\ell$ and standard auxiliary function $Y$. Part i) holds since the auxiliary sequence follows figure \ref{fig:Tree}. We now prove parts ii) and iii).

Since the auxiliary sequence is non-increasing and bounded below by $p$, it converges. If it stabilises, it contains only finitely many distinct X-points. Otherwise, it contains infinitely many distinct X-points. Clearly, these X-points are all regular. Now let $p_n$ be an arbitrary term of the auxiliary sequence which is not the term at which it stabilises. In its neighbourhood, the medians $\mathcal{M}_{\tau\left(p_n\right)-2}$ and $\mathcal{M}_{\tau\left(p_n\right)-1}$ are the lower concatenation and the average, respectively, of the two regular functions forming the X-point $p_n$. By lemma \ref{lemma:translation}, it follows that $Y_{\tau\left(p_n\right)}$ lies above the upper concatenation of the X-point functions (except at $p_n$) which therefore becomes $\mathcal{M}_{\tau\left(p_n\right)}$. Since the limit function is at least the latter median, we have proved that $p_n$ is a local minimum of the limit function. 

Next, if $r$ is the smallest positive integer for which $p_{n+r}<p_n$, then $\tau\left(p_{n+r}\right)=\tau\left(p_n\right)+2$ (figure \ref{fig:SeqOfMedians}), and hence the transit times of the distinct terms of $\left(p_n\right)_{n=\ell}^\infty$ form an increasing arithmetic progression of common difference $2$. The proof of parts ii) and iii) is complete.

\begin{figure}
\centering
\begin{tikzpicture}
\begin{axis}[
	xmin=0,
	xmax=6.25,
	ymin=0.5,
	ymax=4.1,
    xtick={1,4,5},
    ytick=\empty,
    xticklabels={$p$,$p_{n+r}$,$p_n$},
	axis lines=middle,
	samples=100,
	xlabel=$x$,
	ylabel=$Y_n(x)$,
	width=8cm,
	height=6cm,
	clip=false,
	axis lines=middle,
    x axis line style=->,
    y axis line style=->,
]
\addplot [thick,domain=0.5:5.5] {(7*x+5)/12}; 
\addplot [thick,domain=0.5:5.5] {(3*x+5)/8}; 
\addplot [thick,domain=0.5:5.5] {(-x+15)/4} node[right,color=black] {$Y$}; 
\draw[dashed] (axis cs:1,0.5) -- (axis cs:1,1);
\draw[dashed] (axis cs:4,0.5) -- (axis cs:4,2.75);
\draw[dashed] (axis cs:5,0.5) -- (axis cs:5,2.5);
\addplot [very thick,color=red,domain=1:5] {(3*x+5)/8};
\addplot [very thick,color=red,domain=5:5.5] {(-x+15)/4};
\addplot [very thick,color=green,domain=1:4] {(7*x+5)/12};
\addplot [very thick,color=green,domain=4:5] {(-x+15)/4};
\addplot [very thick,color=green,domain=5:5.5] {(3*x+5)/8};
\addplot [very thick,color=blue,domain=1:4] {(23*x+25)/48};
\addplot [very thick,color=blue,domain=4:5.5] {(x+35)/16};
\node[below] at (axis cs:2.5,0.385) {$\ldots$};
\node[below] at (axis cs:5.625,0.385) {$\ldots$};
\fill[black] (axis cs:1,1) circle (2pt);
\draw[thick,fill=white] (axis cs:4,2.75) circle (2pt);
\draw[thick,fill=white] (axis cs:5,2.5) circle (2pt);
\end{axis}
\end{tikzpicture}
\caption{\label{fig:SeqOfMedians}\small The medians $\mathcal{M}_{\tau\left(p_n\right)-2}$ 
(red), $\mathcal{M}_{\tau\left(p_n\right)-1}$ (blue), and 
$\mathcal{M}_{\tau\left(p_n\right)}$ (green),
showing that $\tau\left(p_{n+r}\right)=\tau\left(p_n\right)+2$.}
\end{figure}
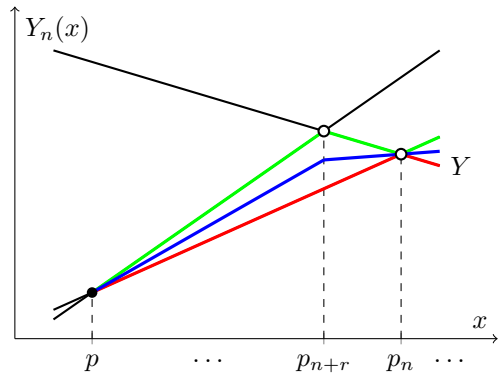

Finally, suppose the auxiliary sequence converges to $p_\infty\in\left[p,p_\ell\right]$. If $p_\infty=p$, then the limit function is discontinuous at $p$. Otherwise, the limit function connects the points $(p,m(p))$ and $\left(p_\infty,Y\left(p_\infty\right)\right)$, so it is continuous at $p$. We have therefore proved part iv).
\end{proof}

Now suppose that the auxiliary sequence stabilises. In this case, the functional orbit near $p$ stabilises and ---assuming no unwanted interactions with earlier functions--- its limit lies immediately above the second-to-last secondary X-point, making it an active high-rank X-point. By identifying its tractability domain, we can establish the shape of the limit function near $p$.

\begin{theorem}\label{thm:Generalrk1}
Let $p$ be a tractable X-point of rank $1$ with standard auxiliary function $Y$ and a stabilising auxiliary sequence, letting $\dddot{p\hspace{0pt}}$, $\ddot{p}$, and $\dot{p}$ be its last three distinct terms, assuming they exist. Suppose $\Xi_{\tau(p)-1}$ is regular in $\left(p,\dddot{p\hspace{0pt}}\right)$ and the number $$p':=\min\left\{x\in\left(\ddot{p},\dddot{p\hspace{0pt}}\right):Y_{\tau\left(\ddot{p}\right)}(x)=Y^\ast(x)\right\}$$ exists, where $Y^\ast$ is the function at which the functional orbit near $p$ stabilises.
\begin{enumerate}
\item[\textrm{i)}] If for every $\overline{Y} \in \Xi_{\tau(p)-1}\backslash[Y]$ we have either $\overline{Y}\left(\ddot{p}\right)<Y\left(\ddot{p}\right)$ or $\overline{Y}\left(p'\right)>Y^*\left(p'\right)$, then $\ddot{p}$ is a regular active X-point of high left-rank which is left-dichotomic in $\left[p,\ddot{p}\right)$ with $Y^\ast$ as its standard auxiliary function.
\item[\textrm{ii)}] If, in addition, in the interval $\left(\ddot{p},\dddot{p\hspace{0pt}}\right)$ we have that $p'<Y_{\tau\left(\dddot{p\hspace{0pt}}\right)}\bowtie Y^\ast$, then $\ddot{p}$ is of high right-rank and is right-dichotomic in $\left[\ddot{p},p'\right)$ with $Y^\ast$ as its standard auxiliary function.
\end{enumerate}
\end{theorem}
\begin{proof}
Let $p$, $\dot{p}$, $\ddot{p}$, $\dddot{p\hspace{0pt}}$, $Y$, and $Y^\ast$ be as stated. Let $t:=\tau\left(\dddot{p\hspace{0pt}}\right)$. Then, by part iii) of theorem \ref{thm:AuxSeq}, we have $\tau\left(\ddot{p}\right)=t+2$ and $\tau\left(\dot{p}\right)=t+4$. Let $Y^\Delta$ be the function in $\Xi_{t-1}$ for which $\ddot{p}=Y^\Delta\bowtie Y$.

Suppose the assumption of i) is true. Then it is clear that the X-point $\ddot{p}$ is regular, active, and has $Y^\ast$ as its standard auxiliary function. Let us now prove that it has high left-rank. In the interval $\left(p,\dot{p}\right)$, the sequence $\left(\mathcal{M}_{n}\right)_{n=|\Xi|}^{t+2}$ is strictly increasing and $\mathcal{M}_{t+2}=\mathcal{M}_{t+3}=Y^\ast$, so $Y^\ast$ belongs to $\Xi_{t+2}$ with multiplicity at least two. Moreover, since $Y^\ast>Y^\Delta$ then $\mathcal{M}_{t+1}=\left\langle Y^\ast,Y^\Delta\right\rangle>Y^\Delta=\mathcal{M}_{t}$, so by lemma \ref{lemma:translation} we have $Y_{t+2}>Y^\ast$, in particular $Y_{t+2}\neq Y^\ast$, so $Y^\ast\in\Xi_{t+1}$ with multiplicity at least two. In other words, $\Xi_{t+1}$ contains at least two functions $\overline{Y}$ for which $\overline{Y}(x)=Y^\ast(x)$ for every $x\in\left(p,\dot{p}\right)$. Since these functions are all regular in $\left(p,\ddot{p}\right)$, then we have proved that $\ddot{p}$ has high left-rank. Moreover, the fact that $\mathcal{M}_{t+2}$ meets $Y^\ast$ at $\dot{p}$, together with the assumption, implies that $\ddot{p}$ is left-tractable with index $t+2$ and domain $\left(\dot{p},\ddot{p}\right)$, which means, by theorem \ref{thm:Generalrk2}, that $\ddot{p}$ is left-dichotomic in $\left(\dot{p},\ddot{p}\right)$. But in $\left[p,\dot{p}\right]$ we have $m=Y^\ast$, so $\ddot{p}$ is left-dichotomic in $\left[p,\ddot{p}\right)$. Therefore we have proved i).

\begin{figure}
\centering
\begin{tikzpicture}
\begin{axis}[
	xmin=1.964656274,
	xmax=7.619652494,
	ymin=1.85,
	ymax=5.6,
    xtick={7.44293386231,5.32899331242,3.99843677834,2,5.98735993908},
    ytick=\empty,
    xticklabels={$\dddot{p\hspace{0pt}}$,$\ddot{p}$,$\dot{p}$,$p$,$p'$},
	samples=100,
	xlabel=$x$,
	ylabel=$Y_n(x)$,
	width=12.5/16*16cm,
	height=12.5/16*6cm,
	clip=false,
    axis lines=middle,
    x axis line style=->,
    y axis line style=->
]
\draw [double,thick] (axis cs:7.44293386231,3.8143112874352902)  node[right,black] {$Y^\ast$} -- (axis cs:2.,2.);
\draw [thick] (axis cs:7.44293386231,2.0907155643717643)-- (axis cs:2.,3.) node[left,black] {$Y$};
\draw  (axis cs:7.44293386231,2.7257245149741163)node[right,black] {$Y^\Delta$}-- (axis cs:2.,2.);
\draw  (axis cs:2.,2.)-- (axis cs:7.44293386231,2.0907155643717643);
\draw [thick] (axis cs:2.,2.)--(axis cs:5.32899331242,5.2)node[above,black,yshift=-3pt] {$Y_{t+1}$}-- (axis cs:7.44293386231,2.0907155643717643);
\draw [thick] (axis cs:2.,2.)-- (axis cs:5.32899331242,4.62053690266186)node[above,black,yshift=-3pt,xshift=-6pt] {$Y_{t}$}-- (axis cs:7.44293386231,2.0907155643717643);
\draw [thick] (axis cs:5.32899331242,2.4438657749896433)-- (axis cs:3.99843677834,3.)node[above,black,yshift=-3pt] {$Y_{t+2}$}-- (axis cs:2.,2.);
\draw [thick] (axis cs:7.44293386231,2.0907155643717643) -- (axis cs:6.96550715199,4.644358173729416)  node[above,black,yshift=-3pt] {$Y_{t+2}$}-- (axis cs:5.32899331242,2.4438657749896433);

\draw[very thick,red] (axis cs:2,2) -- (axis cs:5.98735993908,2.06645599901999578664636);
\draw[very thick,yellow] (axis cs:2,2) -- (axis cs:5.32899331242,2.2496744984316743) -- (axis cs:5.98735993908,2.2001682541431916447152);
\draw[very thick,green] (axis cs:2,2) --  (axis cs:5.32899331242,2.4438657749896433) -- (axis cs:5.98735993908,2.3338805098082990223576);
\draw[very thick,blue] (axis cs:2,2) -- (axis cs:3.99843677834,2.4663019149465306) --(axis cs:5.32899331242,2.4438657749896433) --(axis cs:5.98735993908,2.43276425058413265776424);
\draw[very thick,brown] (axis cs:2,2)-- (axis cs:3.99843677834,2.666145592780758)  -- (axis cs:5.32899331242,2.4438657749896433) -- (axis cs:5.98735993908,2.5316479915240421335364);


\draw[dashed] (axis cs:7.44293386231,1.85) -- (axis cs:7.44293386231,3.8143112874352902);
\draw[dashed] (axis cs:5.32899331242,1.85) -- (axis cs:5.32899331242,5.2);
\draw[dashed] (axis cs:3.99843677834,1.85) -- (axis cs:3.99843677834,3.);
\draw[dashed] (axis cs:2,1.85) -- (axis cs:2,3); 
\draw[dashed] (axis cs:5.98735993908,1.85) -- (axis cs:5.98735993908,3.3291199796919853); 

\fill[black] (axis cs:2,2) circle (2pt);
\draw[fill=white,thick] (axis cs:7.44293386231,2.0907155643717643) circle (2pt);
\draw[fill=white,thick] (axis cs:5.32899331242,2.4438657749896433) circle (2pt);
\draw[fill=white,thick] (axis cs:3.99843677834,2.666145592780758) circle (2pt);
\end{axis}
\end{tikzpicture}
\caption{\label{fig:Rank2NearRank1}\small The situation in a right-neighbourhood of an X-point $p$ of rank $1$ if the functional orbit stabilises. Theorem \ref{thm:Generalrk1} first describes, in part i), the limit function in $\left[p,\ddot{p}\right)$. If the extra condition in part ii) holds, then we can extend the description to $\left[p,p'\right)$. The red, yellow, green, blue, and brown functions are the medians $\mathcal{M}_{t-2}$, $\mathcal{M}_{t-1}$, $\mathcal{M}_{t}$, $\mathcal{M}_{t+1}$, and $\mathcal{M}_{t+2}$, respectively.}
\end{figure}

To prove ii), first notice that, in the open interval $\left(p,\dddot{p\hspace{0pt}}\right)$, the functions $Y_t$ and $Y_{t+1}$ are singular only at $\ddot{p}$ where we have $Y_{t+1}\left(\ddot{p}\right)>Y_t\left(\ddot{p}\right)>Y^\Delta\left(\ddot{p}\right)$ by part i) of lemma \ref{lemma:translation2} and lemma \ref{lemma:translation}, knowing that $\mathcal{M}_{t-1}\left(\ddot{p}\right)>\mathcal{M}_{t-2}\left(\ddot{p}\right)$. If the assumption holds, then ---since the function $Y_{t+2}$ is singular at $\dot{p}$ where $Y_{t+2}\left(\dot{p}\right)>Y^\ast\left(\dot{p}\right)$ by lemma \ref{lemma:translation}, at $\ddot{p}$ where $Y_{t+2}\left(\dot{p}\right)=m(p)$, and at another point on the right of $p'$---
$Y_t\left(\ddot{p}\right)>Y^\ast\left(\ddot{p}\right)$, so the situation is as in figure \ref{fig:Rank2NearRank1}, showing in particular that $\ddot{p}$ has high right-rank, since in $\Xi_{t+1}$ contains at least two functions $\overline{Y}$ for which $\overline{Y}(x)=Y^\ast(x)$ for every $x\in\left(p,p'\right)$. Since part i) of lemma \ref{lemma:translation2} guarantees that $Y_{t+3}\left(p'\right)>Y_{t+2}\left(p'\right)$ and then lemma \ref{lemma:translation} guarantees that $Y_{t+4}\left(p'\right)>Y_{t+2}\left(p'\right)$, then $\mathcal{M}_{t+4}$ meets $Y^\ast$ at $p'$, so that $\ddot{p}$ is right-tractable with index $t+4$ and domain $\left[p,p'\right)$. By theorem \ref{thm:Generalrk2}, it follows that $\ddot{p}$ is right-dichotomic in $\left[p,p'\right)$.
\end{proof}

\begin{figure}
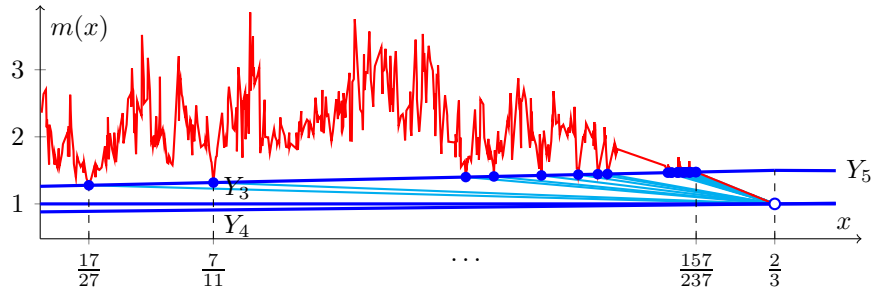

\centering
\include{fig-twothirdsrank1}
\caption{\label{fig:rank1sketch}\small\rm The situation in the left-tractability domain of the X-point $\frac{2}{3}$ of rank $1$ in the system $[0,x,1]$.}
\end{figure}

Let us apply theorems \ref{thm:AuxSeq} and \ref{thm:Generalrk1} on the left-hand side of
the active X-point $\frac{2}{3}=Y_3\bowtie Y_4$ of rank $1$ in the system $[0,x,1]$ which has limit $1$ and transit time $7$.
This X-point is left-tractable with index $9$ and domain $\left(\frac{17}{27},\frac{2}{3}\right)$. Figure \ref{fig:rank1sketch} shows the left auxiliary sequence $\left(p_n\right)_{n=9}^\infty$ 
described by theorem \ref{thm:AuxSeq} which stabilises and contains $27$ distinct regular X-points, 
namely, $\frac{17}{27}$, $\frac{7}{11}$, $\ldots$,  $\frac{626}{945}$, $\frac{157}{237}$, 
in increasing order. 
By part ii) of the theorem, the first $26$ X-points are local minima of the limit function lying on the standard auxiliary function $Y_5$. Their transit times form an arithmetic sequence of difference $2$, 
ranging from $11$ to $63$. Furthermore, the functional orbit stabilises at $Y^\ast(x)=-\frac{225}{2}x+76$, and theorem \ref{thm:Generalrk1} allows us to give a description of the limit function, firstly in
$\left(\frac{626}{945},\frac{2}{3}\right]$, in the form of a dichotomy, which is decidable by computing the orbit of a carefully chosen rational number near the X-point $\frac{626}{945}$ of rank $2$. On
the left-hand side of this X-point, we have a similar dichotomy for the limit function, decidable
analogously. This gives the following formula of the limit function in the left 
neighbourhood $\left(\frac{12310}{18583},\frac{2}{3}\right]$, which not only 
improves \cite[equation (6)]{CellarosiMunday} but also uses significantly less computer assistance:
\begin{equation}\label{eq:near2/3}
m(x)=\begin{cases}
-\frac{225}{2}x+76&\text{if }\frac{12310}{18583}<x\leqslant\frac{50110610}{75646209}\\
-\frac{75675009}{256}x+\frac{25065033}{128}&\text{if }\frac{50110610}{75646209}\leqslant x<\frac{626}{945}\\
\frac{75647841}{256}x-\frac{25055657}{128}&\text{if }\frac{626}{945}\leqslant x<\frac{50130770}{75676641}\\
-\frac{225}{2}x+76&\text{if }\frac{50130770}{75676641}\leqslant x\leqslant\frac{2}{3}.
\end{cases}
\end{equation}

A similar analysis may be performed to improve the description of the limit function near every 
X-point of rank $1$ in $\left[\frac{1}{2},\frac{2}{3}\right]$ of denominator between $3$ and $18$ given 
in \cite{CellarosiMunday}. For instance, near $\frac{7}{12}$, theorem \ref{thm:Generalrk1} gives
eight (rather than two) branches of the limit function. In section \ref{section:Computations} we will do the same for over $2000$ X-points of rank $1$ in the system.

\subsection{Pathologies}\label{section:Pathologies}

We scrutinise the assumptions of the theorems in this section, describing various pathologies.
\medskip

\noindent i) \textit{Standard triads whose symmetries are not inherited}. Consider the X-point $0=Y_1\bowtie Y_{n+1}$ in the system\footnote{For $n=1$, this system is $[0,x,1,1]$, which is studied in section \ref{section:0x11}.}
$${\hat{\Xi}}^{(n)}(x)=[\underbrace{0,0,\ldots,0}_n,x,\underbrace{1,1,\ldots,1}_{n+1}],\quad\text{where } n\geqslant 2,$$
having limit $\frac{1}{2}$ and transit time $2n+3$. This X-point is regular, monotonic, active, of rank $n+1$, but not proper. Its standard triad $\Omega=\left[Y_1,Y_{n+1};Y_{n+2}\right]$ is self-equivalent via \eqref{eq:symmetry0x11}. One checks that $Y_{2n+3}$ is not self-equivalent via the same pair. However, the triad $\overline{\Omega}=\left[Y_{2n+3},Y_{2n+4};Y_{n+1}\right]$ associated to $0=Y_{2n+3}\bowtie Y_{2n+4}$ is self-equivalent via
\begin{equation}\label{eq:exsymmetry2}
\mu(x)=\frac{nx}{x-1}\qquad\text{and}\qquad f(z)=\frac{1}{1-x}z-\frac{x}{1-x}.
\end{equation}
Since $\mathcal{M}_{2n+3}=Y_{2n+3}$ and the median sequence is non-decreasing, then every function $Y_k$, $k\geqslant 2n+3$, is an affine combination of functions in $\overline{\Omega}^\updownarrow$, and so the self-equivalence via \eqref{eq:exsymmetry2} is inherited and is satisfied by the limit function.\medskip

\noindent ii) \textit{Non-tractability and different tractability indices on different sides of an X-point}. An X-point may be both left-tractable and right-tractable with different indices, or may be tractable only on one side. Indeed, consider the systems $[0,x,\alpha x+1,\alpha x+1]$, $\alpha\in\mathbb{Q}$, with
X-point $0=Y_1\bowtie Y_2$ of rank 2, having limit $\frac{1}{2}$ and transit time $5$.

For $\alpha=-10$, the function $Y_5$ intersects the standard auxiliary function
at the point $\left(\frac{1}{13},\frac{3}{13}\right)$ only. 
Therefore, the median $\mathcal{M}_5$ meets the auxiliary function 
on the right-hand side of $0$, namely at $\frac{1}{13}$, but not on the left-hand side (figure \ref{fig:ex}). 
Thus, the X-point is right-tractable with index $\ell=5$. The next
iteration gives the function
$$Y_6(x)=\begin{cases}\frac{3}{2}x+\frac{1}{2}&
 \text{if }x\geqslant 0\\-11x+\frac{1}{2}&\text{if }x<0,\end{cases}$$
which intersects the auxiliary function at the points $\left(-\frac{1}{2},6\right)$ and
$\left(\frac{1}{23},\frac{13}{23}\right)$. Since this function eventually becomes the median $\mathcal{M}_7$, the X-point is left-tractable with index $\ell=7$.

\begin{figure}[t]
\centering
\begin{tikzpicture}
\begin{axis}[
	xmin=-0.1428571429,
	xmax=0.1428571429,
	ymin=-.4285714286,
	ymax=2.428571429,
    xtick={0.1,0.07692307692},
    ytick=\empty,
    xticklabels={$\frac{1}{10}$,$\frac{1}{13}$},
	axis lines=middle,
	samples=100,
	xlabel=$x$,
	ylabel=$Y_n(x)$,
	width=8cm,
	height=6cm,
	clip=false,
	axis line style={<->}
]
\node[left] at (axis cs:-0.125,1.25) {$Y_5$};
\node[left] at (axis cs:-0.125,2.25) {$Y_3$, $Y_4$};
\node[above left] at (axis cs:-0.125,0) {$Y_1$};
\node[below left,yshift=1pt] at (axis cs:-0.125,-0.125) {$Y_2$};
\addplot [double,thick,domain=-0.125:0.125] {-10*x+1};
\addplot [thick,domain=-0.125:0.125] {0};
\addplot [thick,domain=-0.125:0.125] {x};
\addplot [thick,domain=0:0.125] {-7/2*x+1/2};
\addplot [thick,domain=-0.125:0] {1/2-6*x};
\draw[dashed] (axis cs:0.07692307692,0) -- (axis cs:0.07692307692,0.2307692308);
\draw [red,very thick] (axis cs:-0.125,1.25) -- (axis cs:0,0.5) -- (axis cs:0.07692307692,0.2307692308) -- (axis cs:0.1,0) -- (axis cs:0.125,0);
\draw[fill=black] (axis cs:0,0) circle (2pt);
\draw[dashed] (axis cs:-0.125,0) -- (axis cs:-0.125,2.25);
\draw[dashed] (axis cs:-0.125,0) -- (axis cs:-0.125,-0.125);
\draw[dashed] (axis cs:0.125,0) -- (axis cs:0.125,0.125);
\draw[dashed] (axis cs:0.125,0) -- (axis cs:0.125,-0.25);
\end{axis}
\end{tikzpicture} 
\caption{\label{fig:ex}\rm\small The first five functions in the system $[0,x,-10x+1,-10x+1]$ and their median (red).}
\end{figure}
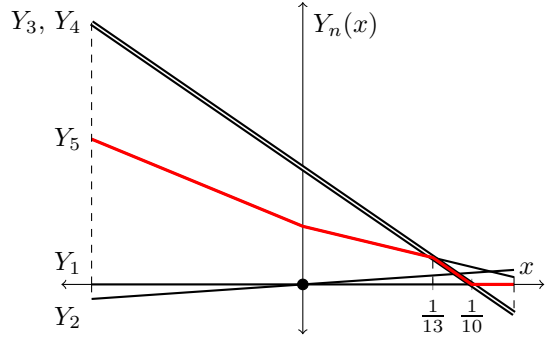

For $\alpha=-388$, the X-point is right-tractable but not left-tractable, because on
the left-hand side of the X-point the median does not intersect the auxiliary function at all.\medskip

\noindent iii) \textit{Fulfilment of all but one tractability conditions}. For an active X-point $p$ to be tractable, there must exist an odd time step $\ell\geqslant\tau(p)$ such that \textbf{T1}, \textbf{T2}, and \textbf{T3} are all satisfied. Let us now show that an odd time step $\ell\geqslant\tau(p)$ satisfying \textbf{T1} can either violate \textbf{T2} and satisfy \textbf{T3}, or violate \textbf{T3} and satisfy \textbf{T2}.

Consider the X-point $\frac{999}{1798}=Y_{12}\bowtie Y_{20}$ in the system $[0,x,1]$, having limit $\frac{4685}{3596}$ and transit time $25$. This X-point is regular, monotonic, proper, active, of rank $1$, and its standard auxiliary function is $Y_9(x)=15x-7$. On the right-hand side of this X-point, the median $\mathcal{M}_{25}$ intersects $Y_9$ at the point $\left(\frac{1867}{3360},\frac{2267}{1699}\right)$, and both $\mathcal{M}_{25}$ and $Y_9$ are regular in the interval $\mathcal{T}=\left(\frac{999}{1798},\frac{1867}{3360}\right)$. Thus, the time step $\ell=25$ satisfies \textbf{T1}. It also satisfies \textbf{T3}, since the function $Y_{25}$ has no corner below $Y_9$. However, $\frac{3991}{7183}=Y_9\bowtie Y_{23}\in\mathcal{T}$, so $\ell$ does not satisfy \textbf{T2}. See figure \ref{fig:999/1798}.

\begin{figure}
\centering
\begin{tikzpicture}
\begin{axis}[
	xmin=1.734527568,
	xmax=1.793122162,
	ymin=.6162219754, 
	ymax=41.81309255,
    xtick={1.73526,1.74300},
    ytick=\empty,
    xticklabels={$\frac{999}{1798}$,$\frac{3991}{7183}$},
	axis lines=middle,
	samples=100,
	xlabel=$x$,
	ylabel=$Y_n(x)$,
	width=8cm,
	height=6cm,
	clip=false,
    x axis line style=->,
    y axis line style=->,
]
\addplot[thick,double,domain=1.73526:1.752] {849379/400*x -36819/10} node[above,xshift=12pt,yshift=-1pt] {$Y_{32}$, $Y_{112}$};

\addplot[thick,domain=1.73526:1.78946] {171/20*x -12} node[above left] {$Y_{20}$};
\addplot[very thick,red,domain=1.73526:1.78946] {171/20*x -12};

\addplot[thick,domain=1.73526:1.78946] {3/20*x +34} node[above left] {$Y_9$};

\draw[thick] (axis cs:1.73526, 34.9555) -- (axis cs:1.75007, 33.6282); 
\addplot[thick,domain=1.73526:1.78946] {-7171/80*x+381/2} node[below left] {$Y_{23}$};

\draw[dashed] (axis cs:1.73526,.6162219754) -- (axis cs:1.73526,34.9555);
\draw[dashed] (axis cs:1.74300,.6162219754) -- (axis cs:1.74300,34.2614);
\draw[dashed] (axis cs:1.78946,.6162219754) -- (axis cs:1.78946,34.2684);

\fill[black] (axis cs:1.73526, 2.83648) circle (2pt);
\fill[black] (axis cs:1.74300, 34.2614) circle (2pt);
\end{axis}
\end{tikzpicture}
\caption{\label{fig:999/1798}\small\rm 
A violation of \textnormal{\textbf{T2}} on the right-hand side of the X-point $\frac{999}{1798}$ in the system $[0,x,1]$. The red function is the median $\mathcal{M}_{25}$.}
\end{figure}
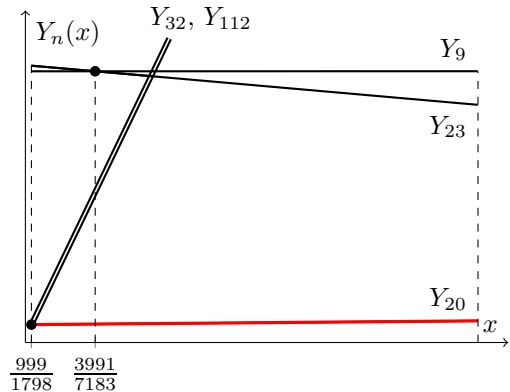

Consider the X-point $0=Y_4\bowtie Y_5$ in the system $\left[-5,-4,-3,Y_4(x),x,3,3\right]$, where
\begin{equation}\label{eq:y4}
Y_4(x)=\begin{cases}
0&\text{if }x< \frac{1}{2}\\
x-\frac{1}{2}&\text{if }x\geqslant\frac{1}{2},
\end{cases}
\end{equation}
having limit $0$ and transit time $9$. This X-point is regular, monotonic, proper, active, and of rank $2$. Its standard auxiliary function is $Y_6(x)=3$. On the right-hand side of this X-point, the median $\mathcal{M}_9$ intersects $Y_6$ at the point $(3,3)$, and both $\mathcal{M}_9$ and $Y_6$ are regular in the interval $\mathcal{T}=(0,3)$. Thus, the time step $\ell=9$ satisfies \textbf{T1}. It also satisfies \textbf{T2}, since none of the functions $Y_1$, \ldots, $Y_8$ intersect the auxiliary function at a point in $\mathcal{T}$. However, $\ell$ does not satisfy \textbf{T3}, since the function $Y_9$ has a corner $\left(\frac{1}{2},\frac{11}{4}\right)$, which lies below the auxiliary function $Y_6$. See figure \ref{fig:ex2}.

\begin{figure}[t]
\centering
\begin{tikzpicture}
\begin{axis}[
	xmin=-0.8214285714,
	xmax=4.321428571,
	ymin=-1,
	ymax=5.5,
    xtick={0,0.5,3,3.5},
    ytick=\empty,
    xticklabels={$0$,$\frac{1}{2}$,$3$,$\frac{7}{2}$},
	axis lines=middle,
	samples=100,
	xlabel=$x$,
	ylabel=$Y_n(x)$,
	width=8cm,
	height=6cm,
	clip=false,
	axis line style={<->}
]
\node[above left] at (axis cs:-0.5,0) {$Y_4$};
\node[left,yshift=-2pt] at (axis cs:-0.5,-0.5) {$Y_5$};
\addplot [double,thick,domain=-0.5:4] {3} node[right,black] {$Y_6$, $Y_7$};
\draw [thick] (axis cs:-0.5,0) -- (axis cs:0.5,0) -- (axis cs:4,3.5); 
\draw [thick] (axis cs:-0.5,-0.5) -- (axis cs:4,4); 
\draw[dashed] (axis cs:0.5,0) -- (axis cs:0.5,2.25); 
\draw[dashed] (axis cs:3,0) -- (axis cs:3,4.75); 
\draw[dashed] (axis cs:3.5,0) -- (axis cs:3.5,3);
\draw [thick] (axis cs:-0.5,1.75) -- (axis cs:0,0) -- (axis cs:0.5,2.25) -- (axis cs:3,4.75) -- (axis cs:3.5,3) -- (axis cs:4,3); 
\node[above] at (axis cs:3,4.75) {$Y_9$};
\draw[very thick,red] (axis cs:-0.5,0) -- (axis cs:0,0) -- (axis cs:3,3) -- (axis cs:4,3);
\draw[fill=black] (axis cs:0,0) circle (2pt);

\draw[dashed] (axis cs:-0.5,0) -- (axis cs:-0.5,-0.5);
\draw[dashed] (axis cs:-0.5,0) -- (axis cs:-0.5,3);
\draw[dashed] (axis cs:4,0) -- (axis cs:4,4);
\end{axis}
\end{tikzpicture}
\caption{\label{fig:ex2}\small A violation of \textnormal{\textbf{T3}} on the right-hand side of the X-point $0$ in the system $\left[-5,-4,-3,Y_4(x),x,3,3\right]$, where $Y_4$ is defined in \eqref{eq:y4}. The red function is the median $\mathcal{M}_{9}$.}
\end{figure}
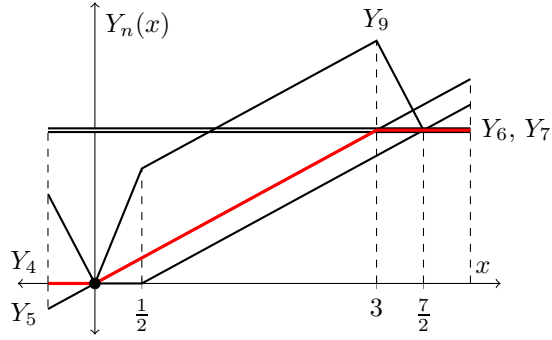

In section \ref{section:Computations} we will describe further failures of tractability. 

\section{The normal form}\label{section:ReducedSystem}

In this section we will exploit further the fact that the functional dynamics near an X-point is eventually independent from most of its earlier history. We will construct the \textit{normal form} of the \mmm, which associates to every odd integer $t\geqslant 5$ an \mmm-like dyadic rational sequence ---called the \textit{normal form orbit of order $t$}--- which is equivalent to the functional \mmm\ orbit near any proper active X-point with transit time $t$ in the system $[0,x,1]$. Later, the number of such X-points is observed to grow exponentially with $t$ (figure \ref{fig:regressions}), so that verifying the stability of a single normal form orbit means verifying the strong terminating conjecture in the vicinity of exponentially many X-points.

Let us first discuss the construction of the normal form itself. Take an odd integer $t=2h+3\geqslant 5$, and two sets $\xi^+=\left[y_1^+,\ldots,y_h^+\right]$ and $\xi^-=\left[y_1^-,\ldots,y_h^-\right]$ of positive and negative numbers, respectively, satisfying
$$\sum_{i=1}^hy_i^+ +\sum_{i=1}^hy_i^-=-1.$$ 
For any $c\geqslant 1$, define the rational set
$$
\xi:=\left(\xi^--c\right)\uplus[0,1]\uplus \left(\xi^++c\right)
$$
having even size $|\xi|=h+2+h=2h+2=t-1$, core $[0,1]$, and median $\mathcal{M}_{t-1}=\mathcal{M}(\xi)=\frac{1}{2}$. One verifies that any set obtained from $\xi$ by removing one element from $\xi^++c$ is an \mmm\ preimage of $\xi$ having median $\mathcal{M}_{t-2}=0$.

Since the medians $\mathcal{M}_{t-2}=0$ and $\mathcal{M}_{t-1}=\frac{1}{2}$ are both independent of $c$, the next iterate $y_t=t\mathcal{M}_{t-1}-(t-1)\mathcal{M}_{t-2}$ depends only on $t$, and not on $c$ or $\xi^\pm$. Moreover, if $c$ is sufficiently large, successive iterates will also depend only on $t$. Letting $c\to\infty$, we find that the whole sequence $\left(y_n\right)_{n=t}^\infty$ depends only on $t$ ---even if it is unbounded--- and we call this sequence the \textit{normal form orbit of order $t$}.

The \textit{normal form} is thus a one-parameter family of dynamical systems, the parameter being an odd integer $t\geqslant 5$, which defines the initial data
$$
\gamma_{t-1}:=[0,1],\qquad \mathcal{M}_{t-2}:=0,\qquad\mathcal{M}_{t-1}:=\frac{1}{2}
$$
and generates a sequence $\left(y_{n}\right)_{n=t}^\infty$, which depends only on $t$, via the recursion [cf.~\eqref{eq:mmm2}]
\begin{equation}\label{eq:ReducedRecursion}
y_{n}=n\mathcal{M}_{n-1}-(n-1)\mathcal{M}_{n-2}\qquad \text{and}\qquad
\gamma_n=\gamma_{n-1}\uplus \left[y_n\right],
\end{equation}
where $\mathcal{M}_n:=\mathcal{M}\left(\gamma_n\right)$, for every $n\geqslant t$. Notice that all points in a normal form orbit are dyadic rational numbers, and the associated median sequence is non-decreasing. The significance of the normal form is given by the following proposition.

\begin{proposition}\label{prop:equivalenceNF} Let $p=Y_i\bowtie Y_j\in\left[\frac{1}{2},\frac{2}{3}\right]$ be any proper active X-point in the system $\Xi_3(x):=[0,x,1]$ with the property that $t:=\tau(p)\geqslant 5$. On a side where $Y_i<Y_j$ we have
\begin{equation}\label{eq:equivalenceNF}
Y_n(x)=\left[Y_j(x)-Y_i(x)\right]y_n+Y_i(x)
\end{equation}
for every $n\geqslant t$. In particular, $\left(Y_n\right)_{n=t}^\infty$ stabilises if and only if $\left(y_n\right)_{n=t}^\infty$ stabilises.
\end{proposition}

\begin{proof}
Suppose the first sentence of the proposition holds. Consider a side of the X-point where $Y_i<Y_j$. Take an integer $n\geqslant t$. Near the X-point, $Y_t=t\mathcal{M}_{t-1}-(t-1)\mathcal{M}_{t-2}=t\left\langle Y_i,Y_j\right\rangle-(t-1)Y_i$ is an affine combination of elements of $\left[Y_i,Y_j\right]$, and hence by remark \ref{remark:independencefromhistory}, so is $Y_n$. Similarly, in the normal form, $y_t$ is an affine combination of elements of $[0,1]$, and hence so is $y_n$. Now construct the unique affine transformation under which $0\mapsto Y_i(x)$ and $1\mapsto Y_j(x)$. This transformation is given by $$z\mapsto \left[Y_j(x)-Y_i(x)\right]z+Y_i(x).$$ Since $Y_n$ depends only on elements of $\left[Y_i,Y_j\right]$, $y_n$ depends only on elements of $[0,1]$, and the \mmm\ preserves affine-equivalences, then we have $y_n\mapsto Y_n(x)$ under the same affine transformation. Therefore, we have the formula \eqref{eq:equivalenceNF}.
\end{proof}

Therefore, the stabilisation of a normal form orbit implies the stabilisation of the functional orbits near all X-points having a particular transit time. We shall exploit this fact in section \ref{section:Computations} to perform efficient computations. Meanwhile, let us conjecture that all normal form orbits stabilise.

\begin{conjecture} [The strong terminating conjecture near proper active X-points]
For every odd integer $t\geqslant 5$, the normal form orbit of order $t$ stabilises.
\end{conjecture}

Our next aim is to derive lower bounds for the limit $m_t$ and the transit time $\tau_t$ of the normal form 
of order $t$. This will be achieved by establishing the existence of an initial phase of the orbit 
---the \textit{regular phase}--- whose length depends on $t$, and which admits
an explicit description (figure \ref{fig:haltingproblem55}).

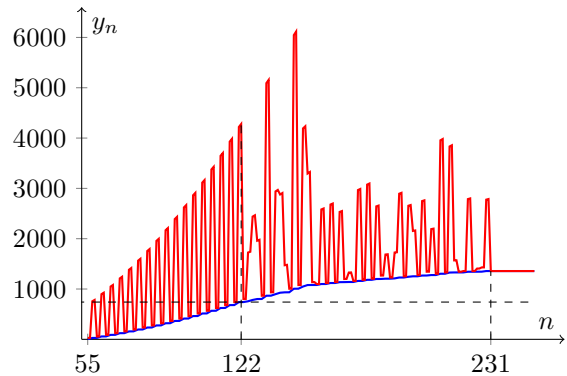
\begin{figure}
\centering
\begin{tikzpicture}
\begin{axis}[
	xmin=52.36486486,
	xmax=263.1756757,
	ymin=0,
	ymax=6600,
    xtick={55,122,231},
    ytick={1000,2000,3000,4000,5000,6000},
    yticklabels={1000,2000,3000,4000,5000,6000},
	axis lines=middle,
    x axis line style=->,
    y axis line style=->,
	samples=100,
	xlabel=$n$,
	ylabel=$y_n$,
	width=8cm,
    height=6cm
]

\draw[thick,color=blue] plot coordinates {(axis cs:55, 1.) (axis cs:56, 14.2500) (axis cs:57, 27.5000) (axis cs:58, 28.) (axis cs:59, 28.5000) (axis cs:60, 42.7500) (axis cs:61, 57.) (axis cs:62, 57.5000) (axis cs:63, 58.) (axis cs:64, 73.2500) (axis cs:65, 88.5000) (axis cs:66, 89.) (axis cs:67, 89.5000) (axis cs:68, 105.750) (axis cs:69, 122.) (axis cs:70, 122.500) (axis cs:71, 123.) (axis cs:72, 140.250) (axis cs:73, 157.500) (axis cs:74, 158.) (axis cs:75, 158.500) (axis cs:76, 176.750) (axis cs:77, 195.) (axis cs:78, 195.500) (axis cs:79, 196.) (axis cs:80, 215.250) (axis cs:81, 234.500) (axis cs:82, 235.) (axis cs:83, 235.500) (axis cs:84, 255.750) (axis cs:85, 276.) (axis cs:86, 276.500) (axis cs:87, 277.) (axis cs:88, 298.250) (axis cs:89, 319.500) (axis cs:90, 320.) (axis cs:91, 320.500) (axis cs:92, 342.750) (axis cs:93, 365.) (axis cs:94, 365.500) (axis cs:95, 366.) (axis cs:96, 389.250) (axis cs:97, 412.500) (axis cs:98, 413.) (axis cs:99, 413.500) (axis cs:100, 437.750) (axis cs:101, 462.) (axis cs:102, 462.500) (axis cs:103, 463.) (axis cs:104, 488.250) (axis cs:105, 513.500) (axis cs:106, 514.) (axis cs:107, 514.500) (axis cs:108, 540.750) (axis cs:109, 567.) (axis cs:110, 567.500) (axis cs:111, 568.) (axis cs:112, 595.250) (axis cs:113, 622.500) (axis cs:114, 623.) (axis cs:115, 623.500) (axis cs:116, 651.750) (axis cs:117, 680.) (axis cs:118, 680.500) (axis cs:119, 681.) (axis cs:120, 710.250) (axis cs:121, 739.500) (axis cs:122, 740.) (axis cs:123, 740.500) (axis cs:124, 748.375) (axis cs:125, 756.250) (axis cs:126, 769.500) (axis cs:127, 782.750) (axis cs:128, 791.875) (axis cs:129, 801.) (axis cs:130, 801.500) (axis cs:131, 802.) (axis cs:132, 834.250) (axis cs:133, 866.500) (axis cs:134, 867.) (axis cs:135, 867.500) (axis cs:136, 882.625) (axis cs:137, 897.750) (axis cs:138, 912.) (axis cs:139, 926.250) (axis cs:140, 930.125) (axis cs:141, 934.) (axis cs:142, 934.500) (axis cs:143, 935.) (axis cs:144, 970.250) (axis cs:145, 1005.50) (axis cs:146, 1006.) (axis cs:147, 1006.50) (axis cs:148, 1027.88) (axis cs:149, 1049.25) (axis cs:150, 1064.12) (axis cs:151, 1079.) (axis cs:152, 1079.38) (axis cs:153, 1079.75) (axis cs:154, 1079.88) (axis cs:155, 1080.) (axis cs:156, 1089.56) (axis cs:157, 1099.12) (axis cs:158, 1099.25) (axis cs:159, 1099.38) (axis cs:160, 1109.19) (axis cs:161, 1119.) (axis cs:162, 1119.12) (axis cs:163, 1119.25) (axis cs:164, 1127.81) (axis cs:165, 1136.38) (axis cs:166, 1136.75) (axis cs:167, 1137.12) (axis cs:168, 1138.25) (axis cs:169, 1139.38) (axis cs:170, 1139.50) (axis cs:171, 1139.62) (axis cs:172, 1150.19) (axis cs:173, 1160.75) (axis cs:174, 1160.88) (axis cs:175, 1161.) (axis cs:176, 1171.81) (axis cs:177, 1182.62) (axis cs:178, 1182.75) (axis cs:179, 1182.88) (axis cs:180, 1190.94) (axis cs:181, 1199.) (axis cs:182, 1199.38) (axis cs:183, 1199.75) (axis cs:184, 1202.38) (axis cs:185, 1205.) (axis cs:186, 1205.12) (axis cs:187, 1205.25) (axis cs:188, 1208.) (axis cs:189, 1210.75) (axis cs:190, 1219.56) (axis cs:191, 1228.38) (axis cs:192, 1228.50) (axis cs:193, 1228.62) (axis cs:194, 1235.94) (axis cs:195, 1243.25) (axis cs:196, 1247.88) (axis cs:197, 1252.50) (axis cs:198, 1252.62) (axis cs:199, 1252.75) (axis cs:200, 1260.19) (axis cs:201, 1267.62) (axis cs:202, 1268.) (axis cs:203, 1268.38) (axis cs:204, 1272.88) (axis cs:205, 1277.38) (axis cs:206, 1277.50) (axis cs:207, 1277.62) (axis cs:208, 1290.44) (axis cs:209, 1303.25) (axis cs:210, 1303.38) (axis cs:211, 1303.50) (axis cs:212, 1315.38) (axis cs:213, 1327.25) (axis cs:214, 1328.38) (axis cs:215, 1329.50) (axis cs:216, 1329.56) (axis cs:217, 1329.62) (axis cs:218, 1329.75) (axis cs:219, 1329.88) (axis cs:220, 1336.47) (axis cs:221, 1343.06) (axis cs:222, 1343.12) (axis cs:223, 1343.19) (axis cs:224, 1343.47) (axis cs:225, 1343.75) (axis cs:226, 1344.12) (axis cs:227, 1344.50) (axis cs:228, 1350.75) (axis cs:229, 1357.) (axis cs:230, 1357.) (axis cs:250,1357.)};

\draw[thick,color=red] plot coordinates {(axis cs:55, 27.5000) (axis cs:56, 28.5000) (axis cs:57, 756.250) (axis cs:58, 782.750) (axis cs:59, 57.) (axis cs:60, 58.) (axis cs:61, 897.750) (axis cs:62, 926.250) (axis cs:63, 88.5000) (axis cs:64, 89.5000) (axis cs:65, 1049.25) (axis cs:66, 1079.75) (axis cs:67, 122.) (axis cs:68, 123.) (axis cs:69, 1210.75) (axis cs:70, 1243.25) (axis cs:71, 157.500) (axis cs:72, 158.500) (axis cs:73, 1382.25) (axis cs:74, 1416.75) (axis cs:75, 195.) (axis cs:76, 196.) (axis cs:77, 1563.75) (axis cs:78, 1600.25) (axis cs:79, 234.500) (axis cs:80, 235.500) (axis cs:81, 1755.25) (axis cs:82, 1793.75) (axis cs:83, 276.) (axis cs:84, 277.) (axis cs:85, 1956.75) (axis cs:86, 1997.25) (axis cs:87, 319.500) (axis cs:88, 320.500) (axis cs:89, 2168.25) (axis cs:90, 2210.75) (axis cs:91, 365.) (axis cs:92, 366.) (axis cs:93, 2389.75) (axis cs:94, 2434.25) (axis cs:95, 412.500) (axis cs:96, 413.500) (axis cs:97, 2621.25) (axis cs:98, 2667.75) (axis cs:99, 462.) (axis cs:100, 463.) (axis cs:101, 2862.75) (axis cs:102, 2911.25) (axis cs:103, 513.500) (axis cs:104, 514.500) (axis cs:105, 3114.25) (axis cs:106, 3164.75) (axis cs:107, 567.) (axis cs:108, 568.) (axis cs:109, 3375.75) (axis cs:110, 3428.25) (axis cs:111, 622.500) (axis cs:112, 623.500) (axis cs:113, 3647.25) (axis cs:114, 3701.75) (axis cs:115, 680.) (axis cs:116, 681.) (axis cs:117, 3928.75) (axis cs:118, 3985.25) (axis cs:119, 739.500) (axis cs:120, 740.500) (axis cs:121, 4220.25) (axis cs:122, 4278.75) (axis cs:123, 801.) (axis cs:124, 802.) (axis cs:125, 1724.88) (axis cs:126, 1740.62) (axis cs:127, 2439.) (axis cs:128, 2465.50) (axis cs:129, 1959.88) (axis cs:130, 1978.12) (axis cs:131, 866.500) (axis cs:132, 867.500) (axis cs:133, 5091.25) (axis cs:134, 5155.75) (axis cs:135, 934.) (axis cs:136, 935.) (axis cs:137, 2939.62) (axis cs:138, 2969.88) (axis cs:139, 2878.50) (axis cs:140, 2907.) (axis cs:141, 1472.62) (axis cs:142, 1480.38) (axis cs:143, 1005.50) (axis cs:144, 1006.50) (axis cs:145, 6046.25) (axis cs:146, 6116.75) (axis cs:147, 1079.) (axis cs:148, 1080.) (axis cs:149, 4191.38) (axis cs:150, 4234.12) (axis cs:151, 3295.38) (axis cs:152, 3325.12) (axis cs:153, 1136.38) (axis cs:154, 1137.12) (axis cs:155, 1099.12) (axis cs:156, 1099.38) (axis cs:157, 2581.31) (axis cs:158, 2600.44) (axis cs:159, 1119.) (axis cs:160, 1119.25) (axis cs:161, 2679.19) (axis cs:162, 2698.81) (axis cs:163, 1139.38) (axis cs:164, 1139.62) (axis cs:165, 2532.06) (axis cs:166, 2549.19) (axis cs:167, 1199.) (axis cs:168, 1199.75) (axis cs:169, 1327.25) (axis cs:170, 1329.50) (axis cs:171, 1160.75) (axis cs:172, 1161.) (axis cs:173, 2966.94) (axis cs:174, 2988.06) (axis cs:175, 1182.62) (axis cs:176, 1182.88) (axis cs:177, 3074.81) (axis cs:178, 3096.44) (axis cs:179, 1205.) (axis cs:180, 1205.25) (axis cs:181, 2642.19) (axis cs:182, 2658.31) (axis cs:183, 1267.62) (axis cs:184, 1268.38) (axis cs:185, 1685.38) (axis cs:186, 1690.62) (axis cs:187, 1228.38) (axis cs:188, 1228.62) (axis cs:189, 1725.) (axis cs:190, 1730.50) (axis cs:191, 2893.94) (axis cs:192, 2911.56) (axis cs:193, 1252.50) (axis cs:194, 1252.75) (axis cs:195, 2654.56) (axis cs:196, 2669.19) (axis cs:197, 2154.38) (axis cs:198, 2163.62) (axis cs:199, 1277.38) (axis cs:200, 1277.62) (axis cs:201, 2747.69) (axis cs:202, 2762.56) (axis cs:203, 1343.75) (axis cs:204, 1344.50) (axis cs:205, 2190.88) (axis cs:206, 2199.88) (axis cs:207, 1303.25) (axis cs:208, 1303.50) (axis cs:209, 3955.44) (axis cs:210, 3981.06) (axis cs:211, 1329.62) (axis cs:212, 1329.88) (axis cs:213, 3832.88) (axis cs:214, 3856.62) (axis cs:215, 1569.12) (axis cs:216, 1571.38) (axis cs:217, 1343.06) (axis cs:218, 1343.19) (axis cs:219, 1357.) (axis cs:220, 1357.25) (axis cs:221, 2787.09) (axis cs:222, 2800.28) (axis cs:223, 1357.) (axis cs:224, 1357.12) (axis cs:225, 1406.47) (axis cs:226, 1407.03) (axis cs:227, 1428.88) (axis cs:228, 1429.62) (axis cs:229, 2775.75) (axis cs:230, 2788.25) (axis cs:231, 1357.) (axis cs:250,1357.)};

\draw[dashed] (axis cs:231,0)--(axis cs:231,1357.);
\draw[dashed] (axis cs:122,0)--(axis cs:122,4278.75);
\draw[dashed] (axis cs:50,740.5)--(axis cs:250,740.5);
\node[below] at (axis cs:143,0) {\scriptsize regular};
\end{axis}
\end{tikzpicture}
\caption{\label{fig:haltingproblem55}\small\rm The normal form orbit of order $55$ (red) 
and the corresponding median sequence (blue). 
The orbit behaves regularly up to $n=N_{55}=122$, by lemma 
\ref{lemma:haltingproblem}, which gives the explicit expressions for all terms up to this index. 
The horizontal dashed line represents the lower bound for the limit given 
by theorem \ref{thm:boundformandtau}, which is $740\frac{1}{2}$.}
\end{figure}

For this purpose, we first note that the terms of the normal form orbit can also be generated by a recursion analogous to \eqref{eq:pardeprecursion} which we shall now write down. Denote by $\lambda_n$ the core of the set $\gamma_n$, for every $n\geqslant t-1$. Then, for every $n\geqslant t+1$, we have
\begin{equation}\label{eq:newrecursion2}
y_n=\begin{cases}
y_{n-1}+\left(y_{j}-y_{i}\right)&n\text{ even, }\lambda_{n-2}=\left[y_i,y_j\right]\\
y_{n-1}+\frac{n}{2}\left(y_{k}-y_{j}\right)-\frac{n-2}{2}\left(y_{j}-y_{i}\right)&n\text{ odd, }\lambda_{n-2}=\left[y_i,y_j,y_k\right],
\end{cases}
\end{equation}
where the core elements are written in non-decreasing order. We shall also use the increasing sequence $\left(u_\ell\right)_{\ell=0}^\infty$ with
$$u_\ell:=\ell+1+\sqrt{5\ell^2+6\ell+5},\qquad\text{for every }\ell\in\mathbb{N}_0.$$
The regular phase is described by the following lemma.

\begin{lemma}\label{lemma:haltingproblem}
Let $t\geqslant 5$ be odd. Then
$$y_t=\frac{t}{2},\qquad y_{t+1}=\frac{t}{2}+1,\qquad y_{t+2}=\frac{t^2}{4},\qquad y_{t+3}=\frac{t^2}{4}+\frac{t}{2}-1,$$
with
\begin{equation}\label{eq:ordering1}
0<1<y_t<y_{t+1}<y_{t+2}<y_{t+3}.
\end{equation}
Moreover, for every $\ell\in\mathbb{N}$, if $t>u_\ell$ then
\begin{equation}\label{eq:RegularPhase}
\def\arraystretch{2.2}
\begin{array}{rcl}
y_{t+4\ell}&=&\displaystyle\frac{\ell+1}{2}t+\ell^2+\ell,\\
y_{t+4\ell+1}&=&\displaystyle\frac{\ell+1}{2}t+\ell^2+\ell+1,\\
y_{t+4\ell+2}&=&\displaystyle\frac{t^2}{4}+\frac{5}{2}\ell t+5\ell^2-\ell,\\
y_{t+4\ell+3}&=&\displaystyle\frac{t^2}{4}+\frac{5\ell+1}{2}t+5\ell^2+\ell-1,
\end{array}
\end{equation}
with
\begin{equation}\label{eq:ordering2}
y_{t+4\ell-3}<y_{t+4\ell}<y_{t+4\ell+1}<y_{t+2}\qquad\text{and}\qquad y_{t+4\ell-1}<y_{t+4\ell+2}<y_{t+4\ell+3}.
\end{equation}
\end{lemma}

\begin{proof}
The proof involves straightforward applications of \eqref{eq:newrecursion2}; we only outline the main steps. 
To prove the first assertion, we begin with $\gamma_{t-1}=[0,1]$.
Then \eqref{eq:ReducedRecursion} gives $y_t=\frac{t}{2}$ and $\gamma_t=\left[0,1,\frac{t}{2}\right]$,
and since $t\geqslant 5$, we find that $0<1<y_t$.
Next we compute $y_{t+1}$, $y_{t+2}$, $y_{t+3}$ recursively, using \eqref{eq:newrecursion2};
the condition $t\geqslant 5$ ensures that the chain of inequalities \eqref{eq:ordering1} is
satisfied at each stage, as easily verified. This proves the first assertion.

The second assertion is proved by strong induction on $\ell\in\mathbb{N}$. 
The basis for induction is the statement that if $t\geqslant 7$ then
$$y_{t+4}=t+2,\qquad y_{t+5}=t+3,\qquad y_{t+6}=\frac{t^2}{4}+\frac{5}{2}t+4,\qquad y_{t+7}=\frac{t^2}{4}+3t+5,$$
with
\begin{equation}\label{eq:ordering4}
y_{t+1}<y_{t+4}<y_{t+5}<y_{t+2}\qquad\text{and}\qquad y_{t+3}<y_{t+6}<y_{t+7}.
\end{equation}
Starting with $\gamma_{t+3}=\left[0,1,y_t,y_{t+1},y_{t+2},y_{t+3}\right]$ and the 
inequalities \eqref{eq:ordering1}, we now apply \eqref{eq:newrecursion2} for $n\in\{t+4,t+5,t+6,t+7\}$. 
The condition $t\geqslant 7$ prescribes uniquely the position in the sequence 
at which the numbers $y_{t+4}$ and $y_{t+5}$ must be inserted, as easily verified. 
We obtain the set $\gamma_{t+7}$ whose terms satisfy the combination of 
\eqref{eq:ordering1} and \eqref{eq:ordering4}, namely,
\begin{equation}\label{eq:ordering6}
0<1<y_t<y_{t+1}<y_{t+4}<y_{t+5}<y_{t+2}<y_{t+3}<y_{t+6}<y_{t+7}.
\end{equation}
The basis is proved.

Now suppose that the statements \eqref{eq:RegularPhase} and \eqref{eq:ordering2} hold for every 
$\ell\in\{1,2,\ldots,r-1\}$, for some $r\geqslant 2$. 
We prove that they also hold for $\ell=r$, assuming that $t>u_r$. 
Since $\left(u_\ell\right)_{\ell=0}^\infty$ is increasing, this implies that 
$t>u_\ell$ for every $\ell\in\{1,2,\ldots,r-1\}$, and hence our iteration 
proceeds up to and including the even-indexed term $y_{t+4r-1}$, 
while \eqref{eq:ordering6} extends to 
\begin{equation}\label{eq:ordering7}\def\arraystretch{1.3}
\begin{array}{c}
0<1<y_t<y_{t+1}<y_{t+4}<y_{t+5}<\cdots<y_{t+4r-4}<y_{t+4r-3}<y_{t+2}\\
<y_{t+3}<y_{t+6}<y_{t+7}<\cdots<y_{t+4r-2}<y_{t+4r-1}.\\
\end{array}
\end{equation}
From the inductive hypothesis, we obtain the term
$$y_{t+4r-1}=\frac{t^2}{4}+\frac{5r-4}{2}t+5r^2-9r+3$$ 
as well as the odd core
\begin{eqnarray*}
\lambda_{t+4r-2}&=&\left[y_{t+4r-7},y_{t+4r-4},y_{t+4r-3}\right]\\
  &=&\left[\frac{r-1}{2}t+r^2-3r+3,\frac{r}{2}t+r^2-r,\frac{r}{2}t+r^2-r+1\right]
\end{eqnarray*}
which contains the three central elements in \eqref{eq:ordering7} before $y_{t+4r-1}$ appears.
Applying \eqref{eq:newrecursion2}, we obtain
\begin{equation}\label{eq:yt+4r}
y_{t+4r}=\frac{r+1}{2}t+r^2+r.
\end{equation}
Using the above and the even core
$$\lambda_{t+4r-1}=\left[y_{t+4r-4},y_{t+4r-3}\right]=\left[\frac{r}{2}t+r^2-r,\frac{r}{2}t+r^2-r+1\right]
$$
which contains the two central elements in \eqref{eq:ordering7}, 
a further application of \eqref{eq:newrecursion2} yields 
\begin{equation}\label{eq:yt+4r+1}
y_{t+4r+1}=\frac{r+1}{2}t+r^2+r+1.
\end{equation}
The elements \eqref{eq:yt+4r} and \eqref{eq:yt+4r+1} are now inserted into the chain of 
inequalities \eqref{eq:ordering7} at the positions uniquely determined by
the assumption $t>u_r$. 
Finally, we apply \eqref{eq:newrecursion2} twice more, to obtain
$$y_{t+4r+2}=\frac{t^2}{4}+\frac{5}{2}rt+5r^2-r\quad\text{and}\quad y_{t+4r+3}=\frac{t^2}{4}+\frac{5r+1}{2}t+5r^2+r-1$$
whose positions in the chain are again determined. 
Our set now is $\gamma_{t+4r+3}$ with the corresponding chain of inequalities
\begin{equation}\label{eq:ordering8}\def\arraystretch{1.3}
\begin{array}{c}
0<1<y_t<y_{t+1}<y_{t+4}<y_{t+5}<\cdots<y_{t+4r-4}<y_{t+4r-3}\\
<y_{t+4r}<y_{t+4r+1}<y_{t+2}<y_{t+3}<y_{t+6}<y_{t+7}<\cdots\\
<y_{t+4r-2}<y_{t+4r-1}<y_{t+4r+2}<y_{t+4r+3}.
\end{array}
\end{equation}
which agrees with \eqref{eq:ordering2} for $\ell=r$. The lemma is proved.
\end{proof}

Lemma \ref{lemma:haltingproblem} gives the 
explicit formula of every term of the orbit up to that of even index
$$
N_t:=t+4L_t+3,
$$
where
$$
L_t:=\max\left\{\ell\in\mathbb{N}_0:t>u_\ell\right\}
=\left\lceil \frac{-t-2+\sqrt{5t^2-4t-12}}{4}\right\rceil-1,
$$
along with the associated chain of inequalities.
These terms are all distinct because each inequality is strict.
Since the next two terms $y_{N_t+1}$ and $y_{N_t+2}$ also obey 
the formulas given in lemma \ref{lemma:haltingproblem}, we find
that for every odd integer $t\geqslant 5$,
\begin{equation}\label{eq:RegularPhaseAdd}
y_{N_t+1}=\frac{L_t+2}{2}t+{L_t}^2+3L_t+2
 \quad\text{and}\quad 
y_{N_t+2}=\frac{L_t+2}{2}t+{L_t}^2+3L_t+3.
\end{equation}
Although certainly $y_{N_t+1}<y_{N_t+2}$, we do not make any statement 
regarding the positions of these two terms in the chain of inequalities, which 
may depend on $t$.

Since stabilisation has not occurred at time step $N_t+2$, we have established 
a lower bound for the transit time $\tau_t$. 
From \eqref{eq:ordering8} we know that 
$\lambda_{N_t}=\left[y_{t+4L_t},y_{t+4L_t+1}\right]$, and hence
$m_t\geqslant y_{t+4L_t+1}$, since the median sequence is non-decreasing. 
We have established the following bounds:

\begin{theorem}\label{thm:boundformandtau}
For every odd $t\geqslant 5$ we have
\begin{eqnarray*}
m_t&\geqslant&\frac{L_t+1}{2}t+{L_t}^2+L_t+1\,\sim\, \frac{t^2}{4}\\
\tau_t&\geqslant& N_t+2\,\sim\, t\sqrt{5}.
\end{eqnarray*}
\end{theorem}

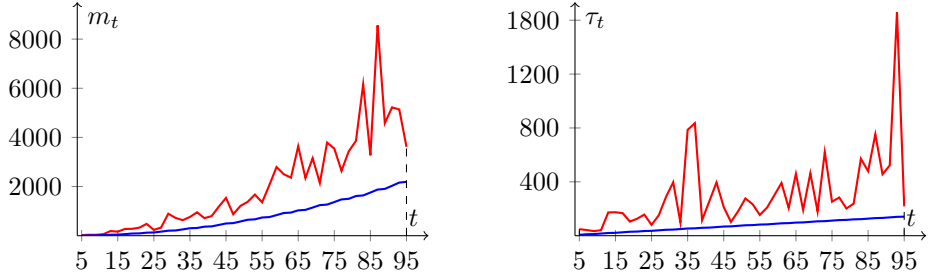
\begin{figure}[t]
\centering
\begin{tabular}{ccc}
\begin{tikzpicture}
\begin{axis}[
	xmin=3.783783784,
	xmax=101.0810811,
	ymin=0,
	ymax=9500,
    xtick={5,15,25,35,45,55,65,75,85,95},
    ytick={2000,4000,6000,8000},
    yticklabels={2000,4000,6000,8000},
	axis lines=middle,
    x axis line style=->,
    y axis line style=->,
	samples=100,
	xlabel=$t$,
	ylabel=$m_t$,
	width=12.5/16*8cm,
    height=12.5/16*6cm
]

\draw[thick,color=red] plot coordinates {(axis cs:5, 20.312) (axis cs:7, 38.500) (axis cs:9, 36.500) (axis cs:11, 61.750) (axis cs:13, 188.88) (axis cs:15, 165.12) (axis cs:17, 273.81) (axis cs:19, 279.97) (axis cs:21, 327.62) (axis cs:23, 481.38) (axis cs:25, 236.25) (axis cs:27, 330.62) (axis cs:29, 894.14) (axis cs:31, 724.56) (axis cs:33, 632.75) (axis cs:35, 762.98) (axis cs:37, 952.91) (axis cs:39, 714.75) (axis cs:41, 795.25) (axis cs:43, 1183.0) (axis cs:45, 1537.5) (axis cs:47, 871.) (axis cs:49, 1212.8) (axis cs:51, 1378.3) (axis cs:53, 1673.1) (axis cs:55, 1357.) (axis cs:57, 2059.1) (axis cs:59, 2794.) (axis cs:61, 2502.5) (axis cs:63, 2362.4) (axis cs:65, 3641.7) (axis cs:67, 2348.5) (axis cs:69, 3153.1) (axis cs:71, 2146.8) (axis cs:73, 3789.2) (axis cs:75, 3547.6) (axis cs:77, 2640.1) (axis cs:79, 3433.5) (axis cs:81, 3858.8) (axis cs:83, 6160.8) (axis cs:85, 3271.3) (axis cs:87, 8561.6) (axis cs:89, 4582.4) (axis cs:91, 5223.1) (axis cs:93, 5137.6) (axis cs:95, 3616.2)};

\draw[thick,color=blue] plot coordinates {(axis cs:5, 3.5000) (axis cs:7, 10.) (axis cs:9, 12.) (axis cs:11, 23.500) (axis cs:13, 39.) (axis cs:15, 43.) (axis cs:17, 63.500) (axis cs:19, 88.) (axis cs:21, 94.) (axis cs:23, 123.50) (axis cs:25, 130.50) (axis cs:27, 165.) (axis cs:29, 203.50) (axis cs:31, 212.50) (axis cs:33, 256.) (axis cs:35, 303.50) (axis cs:37, 314.50) (axis cs:39, 367.) (axis cs:41, 379.) (axis cs:43, 436.50) (axis cs:45, 498.) (axis cs:47, 512.) (axis cs:49, 578.50) (axis cs:51, 649.) (axis cs:53, 665.) (axis cs:55, 740.50) (axis cs:57, 757.50) (axis cs:59, 838.) (axis cs:61, 922.50) (axis cs:63, 941.50) (axis cs:65, 1031.) (axis cs:67, 1051.) (axis cs:69, 1145.5) (axis cs:71, 1244.) (axis cs:73, 1266.) (axis cs:75, 1369.5) (axis cs:77, 1477.) (axis cs:79, 1501.) (axis cs:81, 1613.5) (axis cs:83, 1638.5) (axis cs:85, 1756.) (axis cs:87, 1877.5) (axis cs:89, 1904.5) (axis cs:91, 2031.) (axis cs:93, 2161.5) (axis cs:95, 2190.5)};

\draw[dashed] (axis cs:5,0)--(axis cs:5, 20.312);
\draw[dashed] (axis cs:95,0)--(axis cs:95, 3616.2);
\end{axis}
\end{tikzpicture}&\phantom{a}&\begin{tikzpicture}
\begin{axis}[
	xmin=3.783783784,
	xmax=101.0810811,
	ymin=0,
	ymax=2600,
    xtick={5,15,25,35,45,55,65,75,85,95},
    ytick={600,1200,1800,2400},
    yticklabels={400,800,1200,1800,2400},
	axis lines=middle,
    x axis line style=->,
    y axis line style=->,
	samples=100,
	xlabel=$t$,
	ylabel=$\tau_t$,
	width=12.5/16*8cm,
    height=12.5/16*6cm
]

\draw[thick,color=red] plot coordinates {(axis cs:5, 73) (axis cs:7, 63) (axis cs:9, 51) (axis cs:11, 61) (axis cs:13, 259) (axis cs:15, 261) (axis cs:17, 251) (axis cs:19, 157) (axis cs:21, 189) (axis cs:23, 235) (axis cs:25, 121) (axis cs:27, 227) (axis cs:29, 431) (axis cs:31, 595) (axis cs:33, 147) (axis cs:35, 1179) (axis cs:37, 1253) (axis cs:39, 173) (axis cs:41, 383) (axis cs:43, 593) (axis cs:45, 319) (axis cs:47, 153) (axis cs:49, 273) (axis cs:51, 415) (axis cs:53, 351) (axis cs:55, 231) (axis cs:57, 313) (axis cs:59, 449) (axis cs:61, 587) (axis cs:63, 309) (axis cs:65, 689) (axis cs:67, 287) (axis cs:69, 695) (axis cs:71, 271) (axis cs:73, 931) (axis cs:75, 379) (axis cs:77, 425) (axis cs:79, 303) (axis cs:81, 357) (axis cs:83, 857) (axis cs:85, 719) (axis cs:87, 1129) (axis cs:89, 685) (axis cs:91, 783) (axis cs:93, 2491) (axis cs:95, 327)};

\draw[thick,color=blue] plot coordinates {(axis cs:5, 10) (axis cs:7, 16) (axis cs:9, 18) (axis cs:11, 24) (axis cs:13, 30) (axis cs:15, 32) (axis cs:17, 38) (axis cs:19, 44) (axis cs:21, 46) (axis cs:23, 52) (axis cs:25, 54) (axis cs:27, 60) (axis cs:29, 66) (axis cs:31, 68) (axis cs:33, 74) (axis cs:35, 80) (axis cs:37, 82) (axis cs:39, 88) (axis cs:41, 90) (axis cs:43, 96) (axis cs:45, 102) (axis cs:47, 104) (axis cs:49, 110) (axis cs:51, 116) (axis cs:53, 118) (axis cs:55, 124) (axis cs:57, 126) (axis cs:59, 132) (axis cs:61, 138) (axis cs:63, 140) (axis cs:65, 146) (axis cs:67, 148) (axis cs:69, 154) (axis cs:71, 160) (axis cs:73, 162) (axis cs:75, 168) (axis cs:77, 174) (axis cs:79, 176) (axis cs:81, 182) (axis cs:83, 184) (axis cs:85, 190) (axis cs:87, 196) (axis cs:89, 198) (axis cs:91, 204) (axis cs:93, 210) (axis cs:95, 212)};

\draw[dashed] (axis cs:5,0)--(axis cs:5,73);
\draw[dashed] (axis cs:95,0)--(axis cs:95, 327);
\end{axis}
\end{tikzpicture}
\end{tabular}
\caption{\label{fig:Nt}\rm\small Plots of $m_t$ (left) and $\tau_t$ (right) for $t\in\{5,7,\ldots,95\}$ with the respective lower bounds given by theorem \ref{thm:boundformandtau}.}
\end{figure}

The plots of $m_t$ and $\tau_t$ for $t\in\{5,7,\ldots,95\}$ with the respective lower 
bounds are shown in figure \ref{fig:Nt}. 
All available evidence suggests that, for sufficiently large $t$ the limit $m_t$ is 
also bounded above by the value of the largest term in the regular phase.

\begin{conjecture}\label{thm:conjecturemt}
For every odd integer $t\geqslant 9$, we have $m_t< y_{N_t}$.
\end{conjecture}

We conclude this section with some remarks on the arithmetic of normal form orbits.
Given an odd integer $t\geqslant 5$, we are interested in the smallest ring $\mathbb{K}\subset\mathbb{R}$ which contains the whole normal form orbit of order $t$.
Because \eqref{eq:ReducedRecursion} only produces rational numbers whose denominators 
are powers of $2$, it is clear that $\mathbb{K}\subset\mathbb{Z}\left[\frac{1}{2}\right]$
(the set of rationals whose denominator is a power of $2$). 
With this in mind, we recall that the $2$-\textit{adic valuation} $\nu_2:\mathbb{N}\to\mathbb{N}_0$ 
is given by \cite[page 562]{HardyWright}
$$\nu_2(b):=\max\left\{i\in\mathbb{N}_0:2^i\mid b\right\}$$
for every $b\in\mathbb{N}$.

For every $n\geqslant t$, we define the \textit{$n$-th effective exponent}\footnote{Notice that the sequence $(\kappa(n))_{n=t}^\infty$ is non-decreasing.} as the maximum $2$-adic value of the denominators of the iterates $y_t$, \ldots, $y_n$ in lowest terms, so that
$$\gamma_n\subset \frac{1}{2^{\kappa(n)}}\mathbb{Z}.$$
The $2$-adic value of the denominator of the median $\mathcal{M}_n$ is at most 
$\kappa(n)$ if $n$ is odd and at most $\kappa(n)+1$ if $n$ is even. 
Since $\kappa(t)=\kappa(t+1)=1$, an easy induction shows that for 
every $n\geqslant t$ we have \cite[proposition 2.5]{Hoseana}
$$
\kappa(n)\leqslant \left\lfloor\frac{n-t+2}{2}\right\rfloor.
$$

The above bound is very crude. 
Due to sustained cancellations, the actual exponent is much smaller, 
but its determination seems problematic. 
However, an exact expression for the effective exponent is available during the regular 
phase of the orbits. 
Since $t$ is odd, we can see in lemma \ref{lemma:haltingproblem} 
and equation \eqref{eq:RegularPhaseAdd} that the $2$-adic value of the denominator of 
each number in $\gamma_{N_t+2}$ is the maximum of the $2$-adic values of the 
denominators of the coefficients in its explicit formula. 
Since these denominators are at most $4$, we have:

\begin{corollary}\label{cor:effexp}
For every odd integer $t\geqslant 7$, we have $\kappa\left(N_t+2\right)=2$.
\end{corollary}

\section{Computations}\label{section:Computations}

We conclude by presenting an account of our computations of the limit function.
The existing results are those in \cite{CellarosiMunday}, where 
the strong terminating conjecture for $[0,x,1]$ was established in neighbourhoods 
of all $18$ rational numbers with denominator at most $18$ lying in the 
interval $I:=\left[\frac{1}{2},\frac{2}{3}\right]$. 
In these neighbourhoods, whose measure adds up to $11.75\%$ of the total,
the limit function is piecewise-affine with finitely many 
pieces and $26$ corners.

We have extended these computations considerably, determining neighbourhoods of quasi-regularity
around over $2000$ rational points and generating a sequence of lower 
bounds for the total variation of the limit function. 
Conjecture \ref{conj:Hausdorff}, stated at the end of the introduction,
represents a synthesis of our findings, which we now describe in some detail.

\subsection{Neighbourhoods of quasi-regularity}

By computing the orbit using exact rational arithmetic, we have established that the normal form orbit of order $t$ stabilises for $t\in\{5,\ldots,999\}$. This means that every 
proper active X-point $p$ with $\tau(p)\leqslant 1000$ has a neighbourhood 
in which the strong terminating conjecture holds.

This computation, however, does not yield an estimate of the size of these
neighbourhoods. We have implemented a procedure which associates
to every X-point $p$ a neighbourhood $U_p$ in which the strong terminating
conjecture holds, which we call the \textit{neighbourhood of quasi-regularity} of $p$.
Achieving optimal convergence for the total measure of the $U_p$s
requires ordering the X-points $p$ by the decreasing size of $U_p$. However, this ordering
is not available; indeed $\left|U_p\right|$ has only a tenuous connection with the \textit{height}\footnote{The maximum of the absolute value of the numerator and denominator.} of $p$ and with the transit time $\tau(p)$. We have therefore
adopted a hybrid approach.
\medskip

\noindent{\sc I: Ordering by transit time.} We have constructed the set
$$
\mathcal{X}_t=\left\{p\in I : p\text{ is an X-point
with }\tau(p)=t\right\}.
$$
for every odd integer $t$ with $5\leqslant t\leqslant 25$, by computing explicitly the bundle 
$\Xi_{23}$ and identifying all its active X-points.
Some of the data are
$$\mathcal{X}_9=\left\{\frac{7}{12}\right\},\quad
\mathcal{X}_{11}=\left\{\frac{9}{16},\frac{17}{27}\right\},\quad\text{and}\quad
\mathcal{X}_{13}=\left\{\frac{29}{54},\frac{67}{116},\frac{45}{76},\frac{19}{31},\frac{7}{11}\right\}.$$
The set 
$\mathcal{X}=\bigcup_{\substack{9\leqslant t\leqslant 25\\t\text{ odd}}}\mathcal{X}_t$
has $1176$ elements, all of which are X-points of rank $1$. 
The data
$$
\begin{array}{c|ccccccccc}
t &                 9 & 11 & 13 & 15 & 17 & 19 & 21 & 23 & 25\\
\hline
\left|\mathcal{X}_t\right| &   1 &  2 &  5 & 10 & 24 & 54 &127 &279 & 674
\end{array}
$$
suggest that $\left|\mathcal{X}_t\right|$ grows exponentially with $t$ (figure \ref{fig:regressions}):
\begin{equation}\label{eq:estimator1}
\left|\mathcal{X}_t\right| \approx 0.024 e^{0.41 t}.
\end{equation}

\begin{figure}
\centering
\begin{tikzpicture}
\begin{axis}[
        ymin=0,
        ymax=7.5,
        xmin=8.270270270,
        xmax=26.64864865,
        samples=100,
        xlabel=$t$,
        ylabel=$\ln \left|\mathcal{X}_t\right|$,
        width=12.5/16*8cm,
        height=12.5/16*6cm,
        axis lines=middle,
        x axis line style=->,
        y axis line style=->,        
        clip=false
]
\fill[red] (axis cs:9,0) circle (2pt);
\fill[red] (axis cs:11,.69315) circle (2pt);
\fill[red] (axis cs:13,1.6094) circle (2pt);
\fill[red] (axis cs:15,2.3026) circle (2pt);
\fill[red] (axis cs:17,3.1781) circle (2pt);
\fill[red] (axis cs:19,3.9890) circle (2pt);
\fill[red] (axis cs:21,4.8442) circle (2pt);
\fill[red] (axis cs:23,5.6312) circle (2pt);
\fill[red] (axis cs:25,6.5132) circle (2pt);
\addplot[thick,blue,domain=8.5:25.5] {.40852*x-3.7493};
\end{axis}
\end{tikzpicture}
\caption{\label{fig:regressions}\rm\small 
Least-square regression plot associated to \eqref{eq:estimator1}.}
\end{figure}
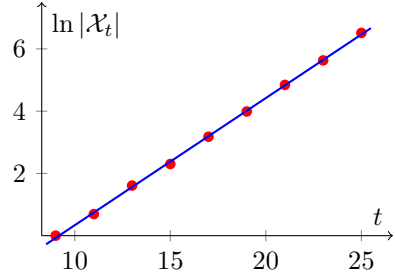

Given $p$, we specify $U_p$ as follows. After generating the required data associated to $p$, we check the validity of the following conditions on the right-hand side of $p$, as well as 
analogous conditions on the left-hand side (notation as in theorem \ref{thm:Generalrk1} and figure \ref{fig:Rank2NearRank1}).
\begin{itemize}
\item[i)] For every distinct $\overline{Y},\overline{\overline{Y}}\in\Xi_{\tau(p)-1}$ we have $\overline{Y}\bowtie\overline{\overline{Y}}\notin\left(p,\dddot{p\hspace{0cm}}\right)$.
\item[ii)] For every $\overline{Y}\in\Xi_{\tau(p)-1}\backslash[Y]$ we have either $\overline{Y}\left(\ddot{p}\right)<Y\left(\ddot{p}\right)$ or $\overline{Y}\left(p'\right)>Y^\ast\left(p'\right)$.
\item[iii)] In $\left(\ddot{p},\dddot{p\hspace{0cm}}\right)$ we have $p'<Y_{\tau\left(\dddot{p\hspace{0cm}}\right)}\bowtie Y^\ast$.
\end{itemize}

Out of all $1176$ X-points, $1092$ satisfy all three conditions on both sides, 
and for these points theorem \ref{thm:Generalrk1} applies to both sides. 
For such X-points $p$, we let $U_p$ be the neighbourhood given by the theorem.

Of the remaining $84$ X-points, $52$ only violate iii), but on both sides. 
On each side, theorem \ref{thm:Generalrk1} enables us to conclude that the strong terminating conjecture holds up to the 
second-to-last term of the auxiliary sequence. 
Therefore, $U_p$ can also be specified without difficulty.

None of the remaining $32$ X-points violate ii); they violate either i) only, 
or both i) and iii), on at least one side. On the right-hand side, 
if i) is violated, then regardless of whether iii) is violated, 
theorem \ref{thm:Generalrk1} does not apply. However, we know that the strong terminating conjecture 
holds in the interval $\left(p,\overline{p}\right)$, where
$$
\overline{p}:=\min\left\{\overline{Y}\bowtie \overline{\overline{Y}}\in (p,\dddot{p\hspace{0cm}}):
            \overline{Y},\overline{\overline{Y}}\in \Xi_{\tau(p)-1}\right\}.
$$
Treating the left-hand side similarly, we obtain $U_p$.

The total measure $\sum_{p\in\mathcal{X}}\left|U_p\right|$
of these neighbourhoods of quasi-regularity amounts to $3.12\%$ of the length of $I$. 
Combining this with the result near $\frac{1}{2}$ obtained in \cite{CellarosiMunday}, which covers $7.29\%$, and the result near $\frac{2}{3}$ in \eqref{eq:near2/3}, which covers $2.54\%$, we conclude that the strong terminating conjecture 
holds over $12.95\%$ of the domain $I$. These neighbourhoods of $1178$ rational numbers contain $7918$ corners of the limit function.\smallskip

\noindent{\sc II: Ordering by height.} An exhaustive search of X-points with larger transit time is not feasible. 
Considering that some rationals in $\mathcal{X}$ have a very large 
height (the largest being $571460$), 
we have supplemented the above data with a selected collection of X-points of higher
transit time, having moderate height.
Specifically, we have considered the set
$$
\left\{p\in\mathcal{F}_{5000}\cap I:
p\text{ is an X-point with }\tau(p)\in\{27,29,31\}\right\},
$$
where $\mathcal{F}_{5000}$ is the set of all fractions with denominator at most $5000$.
This set contains $1618$ active X-points, all of rank 1.
Of those, $1616$ satisfy all three conditions i), ii), iii) on both sides, and
hence theorem \ref{thm:Generalrk1} applies, and we let $U_p$ be the neighbourhood given
by the theorem. It is worth mentioning that $3$ of these $1616$ points are not proper. Each of these X-points is a transversal intersection of three distinct functions appearing before the X-point stabilises, and has neither an inherited symmetry nor a connection to the normal form. The local stabilisation is thus verified
by evaluating \mmm\ orbits at suitable points (cf.~establishment of theorem \ref{thm:0x11} in the paragraph preceding it).
The remaining two X-points are:
\begin{itemize}
\item $p=\frac{708}{1273}$ with $\tau(p)=29$, which violates iii) on the right-hand side;
theorem \ref{thm:Generalrk1} only enables us to conclude that the strong 
terminating conjecture holds up to the second-to-last term $\ddot{p}$ of the 
auxiliary sequence; so the right part of $U_p$ is replaced by 
$\left(p,\frac{722540}{1299143}\right)$.
\item $p=\frac{1913}{3452}$ with $\tau(p)=27$, which violates ii) on the right-hand side; 
the function $Y_{25}$ intersects the limit function 
at $\frac{20613}{37196}$, and so the right part of $U_p$ is replaced by 
$\left(p,\frac{20613}{37196}\right)$.
\end{itemize}
This additional computation adds $1618$ to the number of X-points, $11321$ to the number of 
corners of the limit function, but only $0.21\%$ to the total measure. 
In conclusion, we have proved that the strong terminating 
conjecture holds for $[0,x,1]$ in specified neighbourhoods of $2794$ X-points which cover $13.16\%$ 
of the domain $I$ and contain $19239$ corners 
of the limit function.

From this result, it seems plain that the neighbourhoods of quasi-regularity associated with the
X-points do not account for the full measure.
At the same time, the available evidence strongly suggests that these neighbourhoods form 
a \textit{dense} subset of $I$.

\subsection{Total variation of limit function and box dimension of its graph}

\begin{figure}[t]
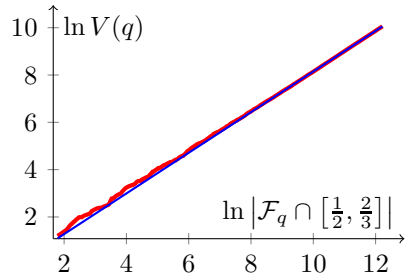

\centering
\include{fig-TotVar}
\caption{\label{fig:TotVar}
\rm\small Log-log plot of the variation of the limit function with respect to the Farey
partition versus the size of the partition. The slope of the line is approximately $0.86$.
}
\end{figure}

To gain insight on the region outside the neighbourhoods of quasi-regularity, we turn our attention to the \textit{total 
variation} \cite[section 6.3]{RoydenFitzpatrick} of the limit function.
We have computed the variation $V(q)$ of the limit function with respect to the Farey partition 
$\mathcal{F}_q\cap I$ ---the set of fractions with denominator up to $q$ in $I$--- for $q\in\{3,\ldots,2000\}$,
with the finest partition containing $202,768$ points. The data show a steady
algebraic growth of $V(q)$, with exponent $\alpha\approx 0.86$ (figure \ref{fig:TotVar}). 

The Farey partition is only approximately uniform.
For additional evidence, and to afford a dimension estimate, we
have computed the variation of the limit function with respect to 
the partition $\frac{1}{2^i}\mathbb{Z}\cap I$ for $i\in\{3,\ldots,14\}$.
Consistently, the data show an algebraic growth with a similar exponent $\alpha\approx 0.85$.

To connect variation to dimension, let us consider a uniform partition of $I$ into $n\in\mathbb{N}$ 
intervals of length $\epsilon_n:=\frac{1}{6n}$, namely $\left\{x_i\right\}_{i=0}^n$ with $x_i:=\frac{1}{2}+i\epsilon_n$ for every $i\in\{0,\ldots,n\}$. Letting $V'(n)$ be the variation of the limit function with respect to this partition, we find that
$$\frac{V'(n)}{\epsilon_n}=
 \sum_{i=1}^n\frac{\left|m\left(x_i\right)-m\left(x_{i-1}\right)\right|}{\epsilon_n}
  \xrightarrow{n\to\infty}N\left(\epsilon_n\right),$$
where $N\left(\epsilon_n\right)$ denotes the minimum number of $\epsilon_n$-boxes needed to cover the graph $$\mathcal{G}:=\{(x,m(x)):x\in I\},$$ assuming continuity. Consequently, if $V'(n)\sim cn^\alpha$ for some $c,\alpha>0$, then the
box dimension of $\mathcal{G}$ is given by
$$\text{dim}_B(\mathcal{G})=\lim_{n\to\infty}\frac{\ln N\left(\epsilon_n\right)}{-\ln \epsilon_n}=
 \lim_{n\to\infty}\frac{\ln\frac{V'(n)}{\epsilon_n}}{-\ln \epsilon_n}
=1+\lim_{n\to\infty}\frac{\ln cn^\alpha}{\ln 6n}
=1+\alpha.$$

We believe that the fractional dimension of the limit function originates entirely from the region outside the neighbourhoods of quasi-regularity.
Indeed, the portion of $\mathcal{G}$ corresponding to the neighbourhoods of quasi-regularity
can be modelled as a family of non-overlapping V-shapes, constructed recursively in such 
a way that their number grows exponentially, while their width decreases exponentially. 
Regardless of the growth of the height of these V-shapes, one finds that the resulting
set has box dimension one. This motivates our conjecture \ref{conj:Hausdorff}.

Finally, we have detected no evidence of multifractal behaviour. Our computation of the variation of the limit function in a few smaller subintervals has given algebraic growth with similar exponents.


\bigskip\noindent
{\sc Acknowledgements:} \/ 
The first author thanks Indonesia Endowment Fund for Education (LPDP) for the financial support.

\end{document}